\documentclass[reqno]{amsart}
\usepackage{color}
\usepackage{amsmath,amssymb,amsfonts,amsthm,bm}
\usepackage{mathrsfs}
\usepackage{graphicx}
\usepackage{enumerate}
\usepackage[numbers, sort&compress]{natbib}
\usepackage[shortlabels]{enumitem}
  \usepackage{newtxtext}
  \usepackage[varvw]{newtxmath}
 \usepackage{textcomp}
\usepackage[reversemp,paperwidth=155mm,paperheight=235mm,top={22mm},headheight={5.5pt},headsep={6.0mm},text={121mm,195mm},marginparsep=5mm,marginparwidth=12mm, bindingoffset=6mm, footskip=10mm]{geometry}

 \newtheorem{theorem}{Theorem}[section]
\newtheorem{lemma}[theorem]{Lemma}
\newtheorem{corollary}[theorem]{Corollary}
\newtheorem{proposition}[theorem]{Proposition}
\newtheorem{definition}[theorem]{Definition}
 \theoremstyle{definition}
\newtheorem{remark}[theorem]{Remark}

\numberwithin{equation}{section}

\newcommand{\eps}{\varepsilon}
\newcommand{\norm}[1]{\Vert#1\Vert}
\newcommand{\abs}[1]{\left\vert#1\right\vert}

\newcommand{\inner}[1]{\left(#1\right)}

\newcommand{\normm}[1]{{ \vert\kern-0.25ex \vert\kern-0.25ex \vert #1
		\vert\kern-0.25ex \vert\kern-0.25ex \vert}}
\newcommand{\divv}{\mbox{div\,}}

 \makeatletter

 \newbox \abstractbox
\renewenvironment{abstract}{\global\setbox\abstractbox=\vbox\bgroup
 \hsize=\textwidth
  \vskip 1.2cm
  \noindent\unskip \textbf{Abstract.}
 }
 {%\vskip12pt \hrule
 \egroup}

\@namedef{subjclassname@2020}{%
	\textup{2020} Mathematics Subject Classification}

\def\@settitle{%
  \bgroup
  \centering
  \vglue1cm
  \fontsize{12}{15}\fontseries{b}\selectfont
  \@title
  \vskip 20pt plus 6pt minus 8pt
  \egroup
}

\def\@setauthors{%
  \begingroup
  \trivlist
  \centering %\bfseries
 \normalsize\@topsep30\p@\relax
  \advance\@topsep by -\baselineskip
  \item\relax
  \andify\authors
 {\rmfamily\authors}%
  \endtrivlist
  \endgroup
}

\def\@setaddresses{\par
  \nobreak \begingroup
\normalsize
  \def\author##1{\nobreak\addvspace\bigskipamount}%
  \def\\{\unskip, \ignorespaces}%
  \interlinepenalty\@M
  \def\address##1##2{\begingroup
    \par\addvspace\bigskipamount\noindent
    \@ifnotempty{##1}{(\ignorespaces##1\unskip) }%
    {\ignorespaces##2}\par\endgroup}%
  \def\curraddr##1##2{\begingroup
    \@ifnotempty{##2}{\nobreak\indent{\itshape Current address}%
      \@ifnotempty{##1}{, \ignorespaces##1\unskip}\/:\space
      ##2\par}\endgroup}%
  \def\email##1##2{\begingroup
    \@ifnotempty{##2}{\nobreak\noindent{\itshape E-mail address}%
      \@ifnotempty{##1}{, \ignorespaces##1\unskip}\/:
       ##2\par}\endgroup}%
   \def\urladdr##1##2{\begingroup
    \@ifnotempty{##2}{\nobreak\indent{\itshape URL}%
      \@ifnotempty{##1}{, \ignorespaces##1\unskip}\/:\space
      \ttfamily##2\par}\endgroup}%
  \addresses
  \endgroup
}

  \renewcommand\section{\@startsection{section}{1} %
  \z@{.5\linespacing\@plus.7\linespacing}{.5\linespacing}
 {\normalfont\large\bfseries\boldmath}}

  \def\subsection{\@startsection{subsection}{2}%
  \z@{.5\linespacing\@plus.7\linespacing}{.2\linespacing}%
  {\raggedright\normalfont\bfseries\boldmath}}

\def\subsubsection{\@startsection{subsubsection}{3}%
  \z@{.5\linespacing\@plus.7\linespacing}{-.5em}%
  {\normalfont\bfseries}}

\makeatother

\begin{document}

\title[Global Gevrey solution of the  3D anisotropic Navier-Stokes system]{Global Gevrey solution of 3D anisotropic Navier-Stokes system in a strip domain}

\author[W.-X. Li, Z. Xu  and P. Zhang]{Wei-Xi Li, Zhan Xu and Ping Zhang}

\address[W.-X. Li]{School of Mathematics and Statistics \& Hubei Key Laboratory of Computational Science,  Wuhan University, Wuhan 430072, China }
\email{wei-xi.li@whu.edu.cn}

\address[Z. Xu]{School of Mathematics and Statistics,  Wuhan University, Wuhan 430072, China\newline
and \newline
Department of Applied Mathematics, The Hong Kong Polytechnic
University, Kowloon, Hong Kong, SAR, China.}
\email{xuzhan@whu.edu.cn}

\address[P. Zhang]{State Key Laboratory of Mathematical Sciences, Academy of Mathematics  \&  Systems Science,   Chinese Academy of Sciences,  Beijing 100190,China \newline
and \newline
 School of Mathematical Sciences, University of Chinese Academy of Sciences  Beijing 100049,    China}
\email{zp@amss.ac.cn}



\begin{abstract}
We investigate the three-dimensional (3D) incompressible anisotropic Navier-Stokes system with dissipation only in the horizontal variables, posed in a strip domain. To overcome the difficulties arising from the boundary terms and the absence of vertical dissipation, we impose a Gevrey-class regularity condition in the vertical direction. For the remaining directions, we prove that the solution exhibits space-time analytic or Gevrey-class regularization. Furthermore, the solution is shown to possess an enhanced Gevrey regularity in the direction of strong diffusion, which is unconstrained by boundaries.

\end{abstract}

\maketitle

 %\setcounter{tocdepth}{1}
%\tableofcontents

\section{Introduction and main results}\label{sec:main}

We consider the following three-dimensional anisotropic Navier-Stokes system in a strip domain:
\begin{equation}\label{ANS}
    \left\{
    \begin{aligned}
        &\partial_t u + (u \cdot \nabla) u - \Delta_{\mathrm{h}} u = -\nabla p, \quad (t,x) \in \mathbb{R}_+ \times [0,1] \times \mathbb{R}^2, \\
        &\operatorname{div} u = 0, \quad u|_{x_1 = 0,1} = 0, \quad u|_{t=0} = u_0,
    \end{aligned}
    \right.
\end{equation}
where $\Delta_{\mathrm{h}} \stackrel{\mathrm{def}}{=} \partial_{x_1}^2 + \partial_{x_2}^2$, and the unknowns $u$ and $p$ denote the fluid velocity and the scalar pressure, respectively. Systems of this type arise in modeling anisotropic geophysical fluids, which exhibit anisotropic diffusion scales. For further details, we refer readers to Pedlosky's book \cite[Chapter 4]{1987Geophysical}.

In comparison with the classical incompressible Navier-Stokes system:
\begin{equation}\label{NS}
    \left\{
    \begin{aligned}
        &\partial_t u + (u \cdot \nabla) u - \Delta u = -\nabla p, \quad (t,x) \in \mathbb{R}_+ \times \mathbb{R}^3, \\
        &\operatorname{div} u = 0, \quad u|_{t=0} = u_0,
    \end{aligned}
    \right.
\end{equation}
the absence of diffusion in the $x_3$-direction, coupled with the non-vanishing boundary terms, leads to the major difficulty in studying the regularizing effect in analytic or Gevrey-class spaces.

Due to the strong diffusion inherent in the heat equation, the associated analytic regularization effect and the radius of analyticity
 for the system \eqref{NS} have been extensively studied. Such parabolic regularization was established for the 3D Navier-Stokes equations \eqref{NS} on the whole space or torus by Foias and Temam \cite{MR1026858}, who proved space-time analyticity via $L^2$-energy estimates and Fourier techniques. Since then, this Fourier-based approach and subsequently developed, more modern analytic methods beyond $L^2$ have been applied to study analyticity for the Navier-Stokes equations and more general dissipative equations in various function spaces; see \cite{MR4601187,MR4216598,MR1607936,MR1938147,MR3369268, MR2836112,MR2960037, MR1748305,MR2145938}. One may consult \cite{LiZhangpreprint} and the references therein for the most recent progress in this direction.

In domains with boundaries, however, the Fourier approach is not applicable. Instead, the Taylor expansion method can be applied, which relies on quantitative estimates of all derivatives. Unlike the case of finite Sobolev regularity, the main difficulty in establishing analyticity arises from handling non-vanishing boundary terms in the quantitative estimates---even though interior analyticity for solutions of the Navier-Stokes equations is generally expected and was studied in \cite{MR247304,MR245989}. The global analyticity up to the boundary for the Navier-Stokes equations can be found in Komatsu's work \cite{MR575737}, where the analyticity, local in the time variable and global in the space variables, was established for $C^\infty$ solutions; see also Giga's work \cite{MR699168} for space-time analyticity of weak solutions based on a semigroup approach. Recently, Camliyurt, Kukavica, and Vicol \cite{MR3843293,MR4082242} proved space-time analyticity up to the boundary via a direct energy-type method.

It is of independent interest to investigate the radius of analyticity. For instance, it helps to understand
turbulence in fluid dynamics and provides insights into the study of geometric regularity for the Navier-Stokes equations (cf.~\cite{MR3001772, MR1871357,MR1383982,MR1467006,MR1331063,MR3497685,MR3369268, MR1073624, MR3215083}
 and the references therein). We recall that in the whole space, the radius $R^\sigma$ of a Gevrey-class function with index $\sigma\geq 1$ may be defined by
\begin{equation}\label{def:radius}
R^{\sigma}\stackrel{\mathrm{def}}{=}\sup\Big\{c> 0:\norm{e^{(c\abs{D_{x}})^{\frac{1}{\sigma}}}g}_{L^2(\mathbb R^3)}<+\infty\Big\},
\end{equation}
where $e^{(c\abs{D_{x}})^{\frac{1}{\sigma}}}$ is the Fourier multiplier with symbol $e^{(c\abs{\xi})^{\frac{1}{\sigma}}}.$
For a Gevrey function defined on a general domain $\Omega$, we adopt the Taylor expansion method and define the radius as
\begin{equation}\label{equradius}
R^\sigma \stackrel{\mathrm{def}}{=}\sup\Big\{c> 0: \sum_{\abs{\alpha}\geq 0}\frac{c^{\abs{\alpha}}}{(\abs{\alpha}!)^\sigma} \norm{\partial^\alpha g}_{L^2(\Omega)}^2<+\infty \Big\},
\end{equation}
which is equivalent to \eqref{def:radius} when $\Omega$ is the whole space or a torus. Using Fourier techniques, it has been shown that solutions to the classical Navier-Stokes system \eqref{NS} with homogeneous $\dot{H}^{\frac12}$-regular initial data satisfy
\begin{align*}
    \int_{\mathbb{R}^3}\abs{\xi}\left(\sup_{0\leq t\leq T}e^{\sqrt{t}\abs{\xi}}\abs{\widehat{u}(t,\xi)}\right)^2d\xi+\int^T_0\int_{\mathbb{R}^3}\abs{\xi}^3\left(e^{\sqrt{t}\abs{\xi}}\abs{\widehat{u}(t,\xi)}\right)^2d\xi dt<+\infty,
\end{align*}
see \cite{MR870865,MR2145938,MR1748305, MR2057722} and the survey book \cite{MR1938147}. This estimate reveals that the radius of analyticity of $u(t)$ is bounded from below by a constant multiple of $\sqrt{t}$, a result applicable to Fujita-Kato solutions of the Navier-Stokes equations \eqref{NS} constructed in \cite{MR166499}. Herbst and Skibsted \cite{MR2836112} later improved the instantaneous analyticity radius, showing that
\begin{align*}
    \lim_{t\to 0^+}\frac{R(t)}{\sqrt{t}}=+\infty.
\end{align*}
For subcritical data $u_0\in H^\gamma$ with $\gamma\in ]1/2,3/2[$, they obtained a logarithmically sharpened analyticity radius.

Another interesting direction concerns the long-time behavior of the analyticity radius. For solutions $u \in C([0,+\infty[; \dot{H}^{\frac12}(\mathbb{R}^3))$ to system \eqref{NS}, Chemin, Gallagher, and Zhang \cite{MR4216598} proved that the analyticity radius $R(t)$ satisfies
\begin{align*}
    \lim_{t\to +\infty}\frac{R(t)}{\sqrt{t}}=+\infty.
\end{align*}
For more refined analysis of the analyticity radius, we refer to the recent work of Li and Zhang \cite{LiZhangpreprint}.

Regarding analyticity radii in various settings, we mention the works \cite{MR3369268, MR2836112, MR1938147, MR4509058, MR1748305} for the whole space or torus; \cite{MR3843293, MR4082242, preprint-liyangzhang} for bounded domains; and \cite{MR2960037, MR1607936,MR4509058, MR1748305, MR4601187} for the $L^p$ framework or critical Besov spaces.

While the aforementioned results mainly address the classical Navier-Stokes system, few works treat the anisotropic Navier-Stokes system. In the case without vertical diffusion, even with strong horizontal diffusion, establishing horizontal analyticity for solutions at positive times remains challenging when the initial data have only finite vertical regularity. To overcome the lack of vertical diffusion, one may impose strong regularity assumptions in that variable. Recently, Liu and Zhang \cite{MR4816041} studied the anisotropic Navier-Stokes system in the whole space and established instantaneous analytic regularization in the horizontal variable under the assumption that the initial data are analytic in the vertical variable. Moreover, the horizontal analyticity radius $R_{\mathrm{h}}(t)$ satisfies
\begin{equation*}
    \lim_{t\to 0^+}\frac{R_{\mathrm{h}}(t)}{\sqrt{t}}=+\infty.
\end{equation*}

The aim of this paper is to study the anisotropic Navier-Stokes system in a domain with boundary. In contrast to the whole-space case or the classical Navier-Stokes system in domains without boundaries, the presence of boundary terms and derivative losses leads to major difficulties in treating system \eqref{ANS}. We improve upon the previous work \cite{MR4816041} by relaxing the analyticity requirement on the initial data and further consider an initial-boundary value problem rather than a Cauchy problem. Specifically, we show the analytic smoothing effect in the $x_2$ variable for initial data with Gevrey regularity in $x_3$ of index up to 2, and we prove that for general Gevrey regular initial data, the solution admits space-time Gevrey smoothing that is global in time and valid up to the boundary. Moreover, the solution exhibits better Gevrey regularity in the $x_2$ variable compared to the other variables.

\noindent\textbf{Notations.} Before stating the main results, we introduce the following notations, which will be used throughout the paper.
\begin{enumerate}[label={(\arabic*)}, leftmargin=*, widest=ii]
    \item We denote by $\norm{\cdot}_{L^2}$ and $(\cdot,\cdot)_{L^2}$ the norm and inner product in $L^2 = L^2([0,1] \times \mathbb{R}^2)$. For functions depending only on the variables $(x_1, x_2)$, we write $\norm{\cdot}_{L_{\mathrm{h}}^2}$ and $(\cdot,\cdot)_{L_{\mathrm{h}}^2}$; similarly, when only a specific variable $y$ is involved, we use $\norm{\cdot}_{L_y^2}$ and $(\cdot,\cdot)_{L_y^2}$. Analogous notations apply to $L^\infty$. We define the classical Lebesgue space $L^p_{\mathrm{h}} L^q_{\mathrm{v}} = L^p([0,1] \times \mathbb{R}_{x_2}; L^q(\mathbb{R}_{x_3}))$, and similarly for the Sobolev space $H^p_{\mathrm{h}} H^q_{\mathrm{v}}$.

    \item For any norm $\norm{\cdot}$ and a vector-valued function $\mathbf{A} = (A_1, A_2, \dots, A_k)$, we define
    \begin{equation*}
        \norm{\mathbf{A}} \stackrel{\mathrm{def}}{=} \Big( \sum_{1\leq j\leq k} \norm{A_j}^2 \Big)^{\frac12}.
    \end{equation*}

    \item For brevity, we write $\partial_i = \partial_{x_i}$ for $i = 1, 2, 3$, and $\partial^\alpha = \partial_1^{\alpha_1} \partial_2^{\alpha_2} \partial_3^{\alpha_3}$ for $\alpha = (\alpha_1, \alpha_2, \alpha_3) \in \mathbb{Z}_+^3$.
    Moreover, we denote the partial gradient by $\nabla_{\mathrm{h}} = (\partial_1, \partial_2)$.

    \item The symbols $\alpha, \beta, \gamma$ denote multi-indices in either $\mathbb{Z}_+^2$ or $\mathbb{Z}_+^3$, depending on context. For $\alpha = (\alpha_1, \dots, \alpha_d) \in \mathbb{Z}_+^d$ ($d = 2$ or $3$) with $\alpha_1 \geq 1$, we define
    \begin{equation*}
        \tilde{\alpha} = \alpha - (1, 0, \cdots, 0) = (\alpha_1 - 1, \alpha_2, \cdots, \alpha_d) \in \mathbb{Z}_+^d.
    \end{equation*}

    \item For positive quantities $a$ and $b$, the notation $a \approx b$ (respectively, $a \lesssim b$) means that there exists a positive constant $C \geq 1$, independent of any free parameters, such that $C^{-1}a \leq b \leq C a$ (respectively, $a \leq C b$).
\end{enumerate}

\begin{definition}\label{def:re1}
    Let $\rho > 0$, $N \in \mathbb{Z}_+$, and $\sigma \geq 1$. The (partial) Gevrey space $G_{\rho,\sigma,N}$ consists of all smooth vector-valued functions $\mathbf{A} = \mathbf{A}(t,x)$, where $x = (x_1, x_2, x_3)$, that are of Gevrey-class $\sigma$ in the $x_3$ variable and satisfy $\norm{\mathbf{A}}_{G_{\rho,\sigma,N}} < +\infty$, with
    \begin{equation*}
        \norm{\mathbf{A}}_{G_{\rho,\sigma,N}}^2 \stackrel{\mathrm{def}}{=} \sum_{m = 0}^{+\infty} L^2_{\rho,m} \norm{\partial_3^m \mathbf{A}}_{H^N}^2,
    \end{equation*}
    where here and below,
    \begin{equation}\label{def:L}
        L_{\rho,m} \stackrel{\mathrm{def}}{=} \frac{\rho^{m+1}(m+1)^{6+2\sigma}}{(m!)^\sigma}, \quad m \ge 0.
    \end{equation}
    Here $L_{\rho,m}$ depends on the Gevrey index $\sigma$, and $\rho$ may depend on time $t$; for simplicity of notation, we write $L_{\rho,m}$ and $\rho$ without explicit parameter dependence.
    In particular, when $\sigma = 1$, the space $G_{\rho,\sigma,N}$ consists of functions that are analytic in the $x_3$ variable.
\end{definition}
With the above notations and definitions, we now present the main results of this work, which address global well-posedness, instantaneous regularization, and refined estimates on the Gevrey radius.

\begin{theorem}[Global well-posedness]\label{thm:main1}
    Let $\sigma \ge 1$ and let the Gevrey space $G_{\rho,\sigma,N}$ be given by Definition~\ref{def:re1}. Assume the initial datum $u_0 \in G_{\rho_0,\sigma, 6}$ for some $\rho_0 > 0$ with $\operatorname{div} u_0 = 0$, compatible with the boundary condition in system~\eqref{ANS}.
    There exists a small constant $\varepsilon_0 > 0$ such that if
    \begin{equation}\label{assume:small}
        \norm{u_0}_{G_{\rho_0,\sigma,6}} \leq \varepsilon_0,
    \end{equation}
    then the anisotropic Navier-Stokes system~\eqref{ANS} admits a unique global solution $u \in L^{\infty}([0,+\infty[; G_{\rho,\sigma,3})$ for some $\rho$ with $\frac{\rho_0}{2} \leq \rho \leq \rho_0$, satisfying
    \begin{equation*}
        \forall t \ge 0, \quad \norm{u(t)}_{G_{\rho,\sigma,3}} \leq C_0 \varepsilon_0 e^{-\frac{t}{4}}
    \end{equation*}
    for some constant $C_0 > 0$.
\end{theorem}

\begin{theorem}[Regularization effect]\label{thm:smoothing}
    Under the assumptions of Theorem~\ref{thm:main1}, let $u \in L^\infty([0,+\infty[; G_{\rho,\sigma,3})$ be the solution to system~\eqref{ANS}. Then $u$ is of anisotropic Gevrey-class regularity up to the boundary at positive times, and there exist two constants $C_1, C_2 \geq 1$ such that for any $T \geq 1$, any $j \in \mathbb{Z}_+$, and any $\alpha = (\alpha_1, \alpha_2, \alpha_3) \in \mathbb{Z}_+^3$,
    \begin{equation}\label{result:2}
    \begin{split}
        \sup_{0 < t \leq T} e^{\frac{t}{4}} & t^{j + \frac{\alpha_1 + \alpha_2}{2}} \norm{\partial_t^j \partial^{\alpha} u}_{H^2} \\
        & \leq \varepsilon_0 C_1^{j + \alpha_2 + \alpha_3 + 1} \big(T^{\frac{1}{2}} C_2\big)^{\alpha_1} [(j + \alpha_1 + \alpha_3)!]^{\sigma} (\alpha_2!)^{\delta},
    \end{split}
    \end{equation}
    where here and throughout the paper, $\delta = \delta(\sigma)$ is defined by
    \begin{equation}\label{gamma1}
        \delta \stackrel{\mathrm{def}}{=} \max\Big\{1, \frac{\sigma + 1}{3}, \frac{2\sigma - 1}{4}\Big\} =
        \begin{cases}
            1, & \text{if } 1 \leq \sigma \leq 2, \\
            \frac{\sigma + 1}{3}, & \text{if } 2 < \sigma \leq \frac{7}{2}, \\
            \frac{2\sigma - 1}{4}, & \text{if } \sigma > \frac{7}{2}.
        \end{cases}
    \end{equation}
\end{theorem}

In view of \eqref{equradius}, estimate \eqref{result:2} implies that the Gevrey radius of $u$ in time is bounded below by a constant multiple of $t$, while the Gevrey radius in $(x_1, x_2)$ is bounded below by $\sqrt{t}$. The following result refines the estimate for the Gevrey radius $R_{x_2}^\delta$ in the $x_2$-direction, defined by
\begin{equation}\label{+def:radius}
    R_{x_2}^{\delta}(t) \stackrel{\mathrm{def}}{=} \sup\Big\{c > 0: \sum_{k \geq 0} \frac{c^{k}}{(k!)^\delta} \norm{\partial_{2}^k u(t)}_{L^2} < +\infty \Big\}.
\end{equation}

\begin{theorem}[Refined radius estimate]\label{thm:radius}
    Under the assumptions of Theorem~\ref{thm:main1}, let $u \in L^\infty([0,+\infty[; G_{\rho,\sigma,3})$ be the solution to system~\eqref{ANS}. Suppose in addition that the initial datum $u_0$ satisfies
    \begin{equation*}
        \partial_1^2 u_0|_{x_1 = 0,1} = \partial_1^4 u_0|_{x_1 = 0,1} = 0.
    \end{equation*}
    Then there exist a small time $T_0 > 0$ and a constant $C_3 > 0$ such that
    \begin{equation}\label{result:radius2}
        \forall 0 < t \leq T_0, \quad R_{x_2}^{\delta}(t) \geq C_3 \sqrt{t |\ln t|},
    \end{equation}
    where $R_{x_2}^{\delta}(t)$ is defined in \eqref{+def:radius}, being the Gevrey radius of the solution $u$ in the $x_2$ variable.
\end{theorem}

\begin{remark}
    The constants $C_0$, $C_1$, and $C_2$ depend only on $\rho_0$, $\sigma$, and Sobolev embedding constants, and are independent of $\varepsilon_0$ and $T$ in \eqref{result:2}. In particular, if $\alpha_1 = 0$ in \eqref{result:2}, we obtain a Gevrey-class smoothing effect in the $(t, x_2)$ variables that is global in time and valid up to the boundary. As shown in Lemma~\ref{lem-pre} and Remark~\ref{remlocal} below, the nontrivial treatment of the pressure prevents a global-in-time radius estimate in $x_1$. However, for the anisotropic Navier-Stokes equations on a torus, one may expect a result similar to Theorem~\ref{thm:main1} with the improved Gevrey estimate:
    \begin{equation}\label{+result:2}
        \sup_{t \geq 0} t^{j + \frac{\alpha_1 + \alpha_2}{2}} \norm{\partial_t^j \partial^{\alpha} u}_{H^2}
        \leq \varepsilon_0 C^{j + \abs{\alpha} + 1} [(j + \alpha_3)!]^{\sigma} [(\alpha_1 + \alpha_2)!]^{\delta},
    \end{equation}
    which is global in time and implies Gevrey-class regularity in $x_1$ with index $\delta$.
\end{remark}

\begin{remark}
    The methods here also apply to the half-space $\mathbb{R}^3_+ = \{ (x_1, x_2, x_3): x_1 > 0 \}$ or the whole space. In these settings, using the enhanced dissipation property from \cite{MR4816041}, global well-posedness can be established based on polynomial time decay of solutions, and moreover the global solution satisfies estimate \eqref{+result:2}.
\end{remark}

\begin{remark}
    In view of \eqref{gamma1}, one may expect an analytic smoothing effect in $x_2$ even without analytic regularity in $x_3$. This, together with estimate \eqref{result:radius2}, generalizes the results in \cite{MR4816041}.
\end{remark}

The rest of this paper is organized as follows. We prove Theorem~\ref{thm:main1} in Section~\ref{sec:wellposdeness}. The proof of Theorem~\ref{thm:smoothing} is presented in Sections~\ref{sec:time}--\ref{sec:x1}. Section~\ref{sec:radius} is devoted to the proof of Theorem~\ref{thm:radius}. Finally, Appendices~\ref{sec:appendix} and~\ref{sec:algebra} provide the proofs of several straightforward inequalities and the algebra property of the space $\mathcal{H}^3$ defined in \eqref{def:mathcalH3}.

To simplify notation, throughout this paper we use the capital letter $C \geq 1$ to denote a generic positive constant that may vary from line to line. This constant depends only on the Sobolev embedding constants, the Gevrey index $\sigma$, and the radius $\rho_0$ of the initial datum $u_0$, and is independent of the parameters $r$ and $\lambda$, which will be specified later.

\section{Proof of Theorem \ref{thm:main1}: global existence and uniqueness}\label{sec:wellposdeness}

This section is devoted to proving Theorem \ref{thm:main1} which addresses  the global-in-time existence and uniqueness of system \eqref{ANS} in Gevrey setting.  To prove Theorem \ref{thm:main1}, it suffices to derive certain \textit{a priori} estimates for system \eqref{ANS}. The existence and uniqueness asserted in Theorem \ref{thm:main1} will then follow via a standard approximation argument. Therefore, for brevity, we restrict our presentation to the proof of the  \emph{a priori} estimate in Theorem \ref{prop:wellposedness} below, omitting the details of the approximation procedure.

\subsection{Auxiliary norms and statement of \textit{a priori} estimate}

We introduce a radius function $\rho = \rho(t)$ and two auxiliary norms, which will be used throughout the subsequent analysis. Let $\rho_0$ be as given in the hypothesis of Theorem \ref{thm:main1}, and define

\begin{equation}\label{def:rerho}
\rho(t)\stackrel{\rm def}{=}\frac{\rho_0}{2}+\frac{\rho_0}{2}e^{-\frac{t}{4}},
\end{equation}
which satisfies
the following fact:
\begin{equation}\label{boundrho}
  \forall\ t\ge0,\quad  \frac{\rho_0}{2}\leq \rho(t)\leq \rho_0.
\end{equation}

\begin{definition}\label{def:normx3}
Let $0 < r < 1$ be a given parameter to be determined later. We define two auxiliary norms $\abs{\cdot}_{X_{\rho}}$ and  $\abs{\cdot}_{Y_{\rho}}$  as follows:
 \begin{equation*}
     \left\{
     \begin{aligned}
&\abs{g}_{X_{\rho}}^2 \stackrel{\rm def}{=}\sum_{\stackrel{(j,   \alpha) \in\mathbb Z_+\times \mathbb{Z}_+^3}{j+\abs{ \alpha }\leq 3}}\ \sum_{m= 0}^{+\infty}r^{ \alpha_1}L^2_{\rho,m}\norm{\partial_3^m\partial_t^j\partial^{ \alpha }g}^2_{L^2},\\
&\abs{g}_{Y_{\rho}}^2\stackrel{\rm def}{=} \sum_{\stackrel{(j,   \alpha) \in\mathbb Z_+\times \mathbb{Z}_+^3}{j+\abs{ \alpha }\leq 3}}\ \sum_{m= 0}^{+\infty} r^{ \alpha_1}(m+1)L^2_{\rho,m}\norm{\partial_3^m\partial_t^j\partial^{ \alpha}g}^2_{L^2},
     \end{aligned}
     \right.
 \end{equation*}
 where $L_{\rho,m}$ is defined in \eqref{def:L}.
\end{definition}

\begin{remark}
From the definitions of   $\abs{g}_{X_{\rho}}$ and  $\abs{g}_{Y_{\rho}}$, it follows that
    \begin{equation}\label{xyone}
        \abs{g}_{X_{\rho}}\leq \abs{g}_{Y_{\rho}}.
    \end{equation}
  Furthermore, if  $u$ is a solution to the anisotropic Navier-Stokes system \eqref{ANS} with initial datum $u_0$, then a direct computation shows that
  \begin{equation}\label{inda}
  	\lim_{t\rightarrow 0} \abs{u(t)}_{X_{\rho}} \leq C_*\norm{u_0}_{G_{\rho_0,\sigma,6}}+C_* \norm{u_0}_{G_{\rho_0,\sigma,6}}^4
  \end{equation}
  for some constant $C_*\geq 1$ which depends   only on  the Sobolev embedding constants.
\end{remark}

\begin{remark}
    \label{algebra}
    The following algebra property holds for the pointwise product of functions:
    \begin{equation}\label{prpalge}
       \norm{  fg }_{\mathcal H^3}\leq C\norm{  f }_{\mathcal H^3}\norm{  g }_{\mathcal H^3},
    \end{equation}
   where   $C>0$ is a constant  depending   only on  the Sobolev embedding constants, and
    \begin{equation}\label{def:mathcalH3}
        \norm{f}_{\mathcal H^3}\stackrel{\rm def}{=}\sum_{\stackrel{(j,   \alpha) \in\mathbb Z_+\times \mathbb{Z}_+^3}{j+\abs{ \alpha }\leq 3}}\   \norm{ \partial_t^j\partial^{ \alpha }f}_{L^2}.
    \end{equation}
 The proof is straightforward; see Appendix \ref{sec:algebra} for details.
\end{remark}

\begin{remark}\label{remk}
Compared with the norm $\norm{\cdot}_{G_{\rho,\sigma,N}}$ in Definition \ref{def:re1},  time derivatives are involved in the norm $\abs{\cdot}_{X_{\rho}}$ defined above. This is because we will use space-time derivatives to handle boundary terms arising from integration by parts in $x_1$. For instance, when estimating the $L^2$-energy of $\partial_1^3 u$, we encounter the boundary term
\begin{equation*}
	\partial_1^3\partial_1 u|_{x_1=0,1}
\end{equation*}
which could lead to a loss of derivatives in  $x_1$.  However,  using the identity
 \begin{equation*}
 	\partial_1^3\partial_1 u|_{x_1=0,1}=\partial_1^2\partial_1^2 u|_{x_1=0,1}=\partial_1^2 \partial_tu|_{x_1=0,1}+\partial_1^2 \big ((u\cdot\nabla)u+\nabla p-\partial_2^2u\big)\big|_{x_1=0,1},
 \end{equation*}
  we can control this boundary term via an $L^2$-energy estimate for $\partial_t\partial_1^2u$. Similarly, the boundary term
 \begin{equation*}
	\partial_t\partial_1^2\partial_1 u|_{x_1=0,1}
\end{equation*}
 can be handled by estimating $\partial_t^2 \partial_1 u$, and ultimately reduced to an energy estimate for $\partial_t^3 u$, which vanishes on the boundary.   The small factor $r^{\alpha_1}$ in definitions of $\abs{\cdot}_{X_{\rho}}$ and $\abs{\cdot}_{Y_{\rho}}$  is introduced to absorb the highest-order derivatives induced by these boundary terms. Further details can be found in Lemmas \ref{lem:S3} and \ref{lem:S4} below.
\end{remark}

Using the above norms, the key \emph{a priori} estimate for proving Theorem \ref{thm:main1} is stated as follows.

\begin{theorem}[{\it A priori} estimate]\label{prop:wellposedness}
Let $\sigma\ge 1$ and  let $\rho$ be defined as in \eqref{def:rerho}. Then there exist a small constant $\eps_0>0$ and a small parameter $0<r<1$ such that
  if $u\in L^{\infty}([0,+\infty[, G_{\rho,\sigma,3})$ is a solution to system  \eqref{ANS}  and   the initial datum $u_0$ satisfies
 \begin{equation}\label{pri:initial1}
\norm{u_0}_{G_{\rho_0,\sigma,6}}\leq (2C_*)^{-1} \varepsilon_0
 \end{equation}
 with $C_*\geq 1$ the constant  given in \eqref{inda},  then
\begin{equation}\label{pri:ret1}
\forall\ t\ge 0,\quad e^{\frac{t}{2}}\abs{u(t)}_{X_{\rho}}^2+\int^t_0e^{\frac{s}{2}}\abs{\nabla_{\rm h}u(s)}_{X_{\rho}}^2ds\leq  \varepsilon_0^2.
\end{equation}
We recall  $r$ is the parameter in definitions of the norms $\abs{\cdot}_{X_{\rho}}$ and $\abs{\cdot}_{Y_{\rho}}$ (see Definition \ref{def:normx3}).
\end{theorem}

We now proceed to prove Theorem \ref{prop:wellposedness} in Subsections \ref{subpure}-\ref{subspcetime}.  We will  use the  following version of Young's inequality for discrete convolution:
\begin{equation}\label{young}
\bigg[\sum^{+\infty}_{m=0}\Big{(}\sum^m_{j=0}p_jq_{m-j}\Big{)}^2\bigg]^{\frac12}\leq \Big{(}\sum^{+\infty}_{m=0}q_m^2\Big{)}^{\frac12}\sum^{+\infty}_{j=0}p_j,
\end{equation}
where $\{p_j\}_{j\ge 0}$ and $\{q_j\}_{j\ge 0}$ are   sequences of nonnegative real numbers.

\subsection{Estimate on pure spatial derivatives}\label{subpure}
This subsection is devoted to estimating the terms in $\abs{u}_{X_\rho}$ involving purely spatial derivatives. The main result is as follows.

\begin{proposition}\label{purespatial}
	Under the hypothesis of Theorem \ref{prop:wellposedness}, for any $\alpha=(\alpha_1,\alpha_2,\alpha_3)\in\mathbb Z_+^3$ with $\abs{\alpha}\leq 3$,  we have that
	\begin{equation*}
		\begin{aligned}
		     &	\frac{1}{2}\frac{d}{dt}\sum_{m=0}^{+\infty} r^{\alpha_1}L_{\rho,m}^2\norm{\partial_3^{m}\partial^{\alpha} u}_{L^2}^2+\sum_{m=0}^{+\infty} r^{\alpha_1}L_{\rho,m}^2 \norm{\partial_3^{m}\partial^{\alpha} \nabla_{\rm h} u}_{L^2}^2\\
		     &\qquad - \frac{\rho'}{\rho}\sum_{m=0}^{+\infty} r^{\alpha_1}(m+1)L_{\rho,m}^2\norm{\partial_3^{m}\partial^{\alpha} u}_{L^2}^2\\
		     &\leq C r^{-\frac{9}{2}}\abs{u}_{X_{\rho}}\abs{u}_{Y_{\rho}}^2+  C\Big(r^{-\frac{9}{2}}\abs{u}_{X_{\rho}}  +r^{\frac{1}{4}} \Big)\abs{\nabla_{\rm h}u}_{X_{\rho}}^2,
		\end{aligned}
	\end{equation*}
	where the norms $\abs{\cdot}_{X_{\rho}}$ and $\abs{\cdot}_{Y_{\rho}}$ are given in Definition \ref{def:normx3}.
\end{proposition}

It suffices to prove Proposition \ref{purespatial}  for multi-indices $\alpha=(\alpha_1,\alpha_2,\alpha_3)$ with $\alpha_1\geq 1.$ The case   $\alpha=(0,\alpha_2,\alpha_3)$ will follow from a similar yet simpler argument, as no boundary term appears in that situation. For details, we refer to  Lemma \ref{nonvani} and Remark \ref{vanishing} below.

\begin{lemma}\label{nonvani}
For any $\alpha=(\alpha_1,\alpha_2,\alpha_3)\in\mathbb{Z}_+^3$ with $\alpha_1\geq 1$ and $\abs\alpha \leq 3,$
we have
\begin{multline*}
	  \inner{-\partial_3^{m}\partial^{\alpha} \Delta_{\rm h}u+ \partial_3^{m}\partial^{\alpha} \nabla p, \ \partial_3^{m}\partial^{\alpha}  u}_{L^2}\\
	 =\norm{\partial_3^{m}\partial^{\alpha} \nabla_{\rm h} u}_{L^2}^2-\inner{\partial_3^m\partial ^{\tilde\alpha} \nabla p, \ \partial_3^{m}\partial^{\alpha} \partial_1 u}_{L^2}\\
	  - \int_{\mathbb R^2} \big[ (\partial_3^{m}\partial^{\alpha}  u) \partial_3^m\partial ^{\tilde\alpha} \big(\partial_tu+(u\cdot\nabla) u-\partial_2^2u\big)\big]\big|^{x_1=1}_{x_1=0}\ dx_2dx_3,
\end{multline*}
where  $ \tilde\alpha =\alpha-(1,0,0)=(\alpha_1-1,\alpha_2, \alpha_3)$ for $\alpha_1\geq 1.$
\end{lemma}

\begin{remark}\label{vanishing}
 For any multi-index   $\alpha=(\alpha_1,\alpha_2,\alpha_3)\in\mathbb{Z}_+^3$ with $\alpha_1=0,$   it follows from $\divv u=0$ that
\begin{equation*}%\label{alpha10}
	  \inner{-\partial_3^{m}\partial^{\alpha} \Delta_{\rm h}u+ \partial_3^{m}\partial^{\alpha} \nabla p, \ \partial_3^{m}\partial^{\alpha}  u}_{L^2} =\norm{\partial_3^{m}\partial^{\alpha} \nabla_{\rm h} u}_{L^2}^2.
\end{equation*}
\end{remark}

\begin{proof}[Proof of Lemma \ref{nonvani}]
For any $\alpha=(\alpha_1,\alpha_2,\alpha_3)\in\mathbb Z_+^3$ with $\alpha_1\geq1$,  we have $\partial^\alpha=\partial_1\partial^{\tilde\alpha}$ and then  use integration by parts to  write
\begin{align*}
	& \inner{-\partial_3^{m}\partial^{\alpha} \Delta_{\rm h}u+ \partial_3^{m}\partial^{\alpha} \nabla p, \ \partial_3^{m}\partial^{\alpha}  u}_{L^2}\\
	&=\norm{\partial_3^{m}\partial^{\alpha} \nabla_{\rm h} u}_{L^2}^2-\inner{\partial_3^m\partial ^{\tilde\alpha} \nabla p, \ \partial_3^{m}\partial^{\alpha} \partial_1 u}_{L^2}\\
	&\quad + \int_{\mathbb R^2} \Big[(\partial_3^m \partial^{\tilde\alpha}\nabla p )(\partial_3^{m}\partial^{\alpha}  u)\big|^{x_1=1}_{x_1=0}-(\partial_3^{m}\partial^{\alpha} \partial_{1} u)(\partial_3^{m}\partial^{\alpha}  u)\big|^{x_1=1}_{x_1=0}\Big]dx_2dx_3.
\end{align*}
Moreover, it follows from the fact   $\partial^\alpha\partial_1=\partial^{\tilde\alpha}\partial_1^2$ that
\begin{equation*}
	\begin{aligned}
	& \partial_3^{m}\partial^{\alpha} \partial_{1} u|_{x_1=0,1}	=\partial_3^m\partial ^{\tilde\alpha} \big(\partial_tu+(u\cdot\nabla) u-\partial_2^2u\big)|_{x_1=0,1}+\partial_3^m\partial ^{\tilde\alpha}  \nabla p |_{x_1=0,1}.
	\end{aligned}
\end{equation*}
Thus combining the two equations above yields assertion of Lemma \ref{nonvani}. The proof is thus completed.
\end{proof}

The following discussion is devoted to proving Proposition \ref{purespatial}. To do so,
applying $\partial_3^{m}\partial^{\alpha}$ to the first equation in system \eqref{ANS} yields
\begin{equation*}%\label{eqmalpha}
 \partial_t\partial_3^{m}\partial^{\alpha}  u+\partial_3^{m}\partial^{\alpha} \big[(u\cdot\nabla) u\big]-\partial_3^{m}\partial^{\alpha} \Delta_{\rm h}u+\partial_3^{m}\partial^{\alpha} \nabla p=0,
\end{equation*}
and thus,
\begin{multline*}
    \frac{1}{2}\frac{d}{dt}\norm{\partial_3^{m}\partial^{\alpha} u}_{L^2}^2+ \inner{-\partial_3^{m}\partial^{\alpha} \Delta_{\rm h}u+ \partial_3^{m}\partial^{\alpha} \nabla p, \ \partial_3^{m}\partial^{\alpha}  u}_{L^2}\\
    =-\inner{\partial_3^{m}\partial^{\alpha} \big((u\cdot\nabla) u\big),\ \partial_3^{m}\partial^{\alpha}  u}_{L^2}.
\end{multline*}
Multiplying  both sides   by $r^{\alpha_1}L_{\rho,m}^2$ and summing over   $m\in\mathbb{Z}_+$, and using the identity
\begin{align*}
    \frac{1}{2}\frac{d}{dt}L_{\rho,m}^2=\frac{\rho^{'}}{\rho}(m+1)L_{\rho,m}^2,
\end{align*}
we obtain
\begin{equation*}
	\begin{aligned}
	 &\frac{1}{2}\frac{d}{dt}\sum_{m=0}^{+\infty} r^{\alpha_1}L_{\rho,m}^2\norm{\partial_3^{m}\partial^{\alpha} u}_{L^2}^2- \frac{\rho'}{\rho}\sum_{m=0}^{+\infty} r^{\alpha_1}(m+1)L_{\rho,m}^2\norm{\partial_3^{m}\partial^{\alpha} u}_{L^2}^2\\
	 &\qquad+\sum_{m=0}^{+\infty} r^{\alpha_1}L_{\rho,m}^2 \inner{-\partial_3^{m}\partial^{\alpha} \Delta_{\rm h}u+ \partial_3^{m}\partial^{\alpha} \nabla p, \ \partial_3^{m}\partial^{\alpha}  u}_{L^2}\\
  &  =-\sum_{m=0}^{+\infty} r^{\alpha_1}L_{\rho,m}^2\inner{\partial_3^{m}\partial^{\alpha} \big((u\cdot\nabla) u\big),\ \partial_3^{m}\partial^{\alpha}  u}_{L^2}.	
	\end{aligned}
\end{equation*}
Combining this with Lemma \ref{nonvani}, we conclude that for any  $\alpha\in\mathbb Z_+^3$ with $\alpha_1\geq 1,$
\begin{multline}
	\label{s}
	\frac{1}{2}\frac{d}{dt}\sum_{m=0}^{+\infty} r^{\alpha_1}L_{\rho,m}^2\norm{\partial_3^{m}\partial^{\alpha} u}_{L^2}^2 +\sum_{m=0}^{+\infty} r^{\alpha_1}L_{\rho,m}^2 \norm{\partial_3^{m}\partial^{\alpha} \nabla_{\rm h} u}_{L^2}^2\\
		 - \frac{\rho'}{\rho}\sum_{m=0}^{+\infty} r^{\alpha_1}(m+1)L_{\rho,m}^2\norm{\partial_3^{m}\partial^{\alpha} u}_{L^2}^2     \leq \sum_{j=1}^4 S_j,
\end{multline}
with
\begin{equation}\label{S1-S4}
    \left\{
    \begin{aligned}
&S_1= - \sum_{m= 0}^{+\infty} r^{\alpha_1} L^2_{\rho,m}\inner{\partial_3^{m}\partial^{\alpha}  \big ((u_{\rm h}\cdot\nabla_{\rm h}) u\big ),\ \partial_3^{m}\partial^{\alpha}  u}_{L^2},\\
&S_2=- \sum_{m= 0}^{+\infty} r^{\alpha_1} L^2_{\rho,m}\inner{\partial_3^{m}\partial^{\alpha}  (u_3\partial_3u),\ \partial_3^{m}\partial^{\alpha}  u}_{L^2},\\
&S_3= \sum_{m= 0}^{+\infty}r^{\alpha_1} L^2_{\rho,m}\inner{\partial_3^m\partial ^{\tilde\alpha}\nabla p, \ \partial_3^{m}\partial^{\alpha} \partial_1 u}_{L^2},\\
&S_4= \sum_{m= 0}^{+\infty}r^{\alpha_1}L^2_{\rho,m}\int_{\mathbb R^2} \big[ (\partial_3^{m}\partial^{\alpha}  u) \partial_3^m\partial ^{\tilde\alpha} \big(\partial_tu+(u\cdot\nabla) u-\partial_2^2u\big)\big]\big|^{x_1=1}_{x_1=0}\ dx_2dx_3,
    \end{aligned}
    \right.
\end{equation}
where $ \tilde\alpha = \alpha-(1,0,0)$ in terms $S_3$ and $S_4.$
We will proceed to estimate $S_j$ for $ 1\leq j\leq 4 $ through the following lemmas. The main difficulty lies in handling
 $S_3$ and $S_4$
  which involve boundary terms.

\begin{lemma}[Estimate on $S_1$]\label{lem:S1}
 Let $S_1$ be given in \eqref{S1-S4}.  For any $\alpha=(\alpha_1,\alpha_2,\alpha_3)\in\mathbb Z_+^3$ with $\abs\alpha \leq 3,$  it holds that
 \begin{equation}\label{est:S1}
S_1\leq C r^{-\frac{9}{2}}\abs{u}_{X_{\rho}}^2\abs{\nabla_{\rm h}u}_{X_{\rho}}\leq C r^{-\frac{9}{2}}\abs{u}_{X_{\rho}}\abs{u}_{Y_{\rho}}^2+ Cr^{-\frac{9}{2}}\abs{u}_{X_{\rho}} \abs{\nabla_{\rm h}u}_{X_{\rho}}^2,
 \end{equation}
 recalling $\abs{\cdot}_{X_{\rho}}$ and $\abs{\cdot}_{Y_{\rho}}$ are given in Definition \ref{def:normx3}.
\end{lemma}

\begin{proof}
Note $H^3$ is an algebra under pointwise multiplication. Then we use Leibniz's formula to get that,   for any $ \alpha \in \mathbb{Z}_+^3$ with $\abs{ \alpha}\leq 3$,
\begin{align*}
\left|\inner{\partial_3^{m}\partial^{\alpha}  \big ((u_{\rm h}\cdot\nabla_{\rm h}) u\big ),\ \partial_3^{m}\partial^{\alpha}  u}_{L^2}\right|\leq C\sum^m_{k=0}\binom{m}{k}\norm{\partial_3^ku}_{H^3}\norm{\partial_3^{m-k}\nabla_{\rm h}u}_{H^3}\norm{\partial_3^mu}_{H^3}.
\end{align*}
This gives
 \begin{equation}\label{S1}
\begin{aligned}
S_1\leq &C\sum_{m=0}^{+\infty}\,\sum_{k=0}^{[\frac{m}{2}]}\binom{m}{k}L_{\rho,m}^2\norm{\partial_3^ku}_{H^3}\norm{\partial_3^{m-k}\nabla_{\rm h}u}_{H^3}\norm{\partial_3^mu}_{H^3}\\
&+C\sum_{m=0}^{+\infty}\,\sum_{k=[\frac{m}{2}]+1}^{m}\binom{m}{k}L_{\rho,m}^2\norm{\partial_3^ku}_{H^3}\norm{\partial_3^{m-k}\nabla_{\rm h}u}_{H^3}\norm{\partial_3^mu}_{H^3}\\
\stackrel{\rm def}{=}&S_{1,1}+S_{1,2},
\end{aligned}
\end{equation}
where here and below $[p]$ stands for the largest integer less than or equal to $p$.  To deal with $S_{1,1}$ in \eqref{S1},  we use the  inequality (see Appendix \ref{sec:appendix} for the proof)
\begin{equation}\label{ineq1}
 \forall\ k\in\mathbb Z_+\ \textrm{ with }\ k\leq \big[\frac{m}{2}\big],\quad    \binom{m}{k}\frac{L_{\rho,m}}{L_{\rho,k}L_{\rho,m-k}}\leq\frac{C}{(k+1)^6},
\end{equation}
 to obtain that
\begin{equation}\label{S11}
    \begin{aligned}
S_{1,1}
 &\leq C\sum_{m= 0}^{+\infty}\sum_{k=0}^{[\frac{m}{2}]}\binom{m}{k}\frac{L_{\rho,m}}{L_{\rho,k}L_{\rho,m-k}}\left(L_{\rho,k}\norm{\partial_3^{k}u}_{H^3}\right)\Big (L_{\rho,m-k}\norm{\partial_3^{m-k}\nabla_{\rm h}u}_{H^3}\Big)\\
&\qquad\times \Big (L_{\rho,m}\norm{\partial_3^{m}u}_{H^3}\Big)\\
&\leq C\sum_{m= 0}^{+\infty} \bigg (\sum_{k=0}^{[\frac{m}{2}]}\frac{L_{\rho,k}\norm{\partial_3^{k}u}_{H^3}}{(k+1)^6}L_{\rho,m-k}\norm{\partial_3^{m-k}\nabla_{\rm h}u}_{H^3}\bigg )\Big (L_{\rho,m}\norm{\partial_3^{m}u}_{H^3}\Big)\\
&\leq Cr^{-\frac{3}{2}}\bigg [\sum_{m= 0}^{+\infty}\Big (\sum_{k=0}^{m}\frac{L_{\rho,k}\norm{\partial_3^{k}u}_{H^3}}{(k+1)^6}L_{\rho,m-k}\norm{\partial_3^{m-k}\nabla_{\rm h}u}_{H^3}\Big )^2\bigg ]^\frac{1}{2}\abs{u}_{X_{\rho}},
    \end{aligned}
\end{equation}
the last line using Cauchy inequality as well as the definition of $\abs{\cdot}_{X_{\rho}}$ (see Definition \ref{def:normx3}). Moreover, it follows from  Young's inequality \eqref{young} for discrete convolution that
\begin{equation}\label{applyYong}
    \begin{aligned}
&\bigg [\sum_{m= 0}^{+\infty}\Big (\sum_{k=0}^{m}\frac{L_{\rho,k}\norm{\partial_3^{k}u}_{H^3}}{(k+1)^6}L_{\rho,m-k}\norm{\partial_3^{m-k}\nabla_{\rm h}u}_{H^3}\Big )^2\bigg ]^\frac{1}{2}\\
&\leq C\bigg  (\sum_{m= 0}^{+\infty}\frac{L_{\rho,m}\norm{\partial_3^{m}u}_{H^3}}{(m+1)^6}\bigg )\Big (\sum_{m= 0}^{+\infty}L_{\rho,m}^2\norm{\partial_3^{m}\nabla_{\rm h}u}_{H^3}^2\Big )^\frac{1}{2}\\
&\leq Cr^{-\frac{3}{2}}\abs{\nabla_{\rm h}u}_{X_{\rho}} \Big (\sum_{m= 0}^{+\infty}L_{\rho,m}^2\norm{\partial_3^{m}u}_{H^3}^2\Big )^\frac{1}{2}\leq Cr^{-3}\abs{u}_{X_{\rho}}\abs{\nabla_{\rm h}u}_{X_{\rho}}.
    \end{aligned}
\end{equation}
Consequently,  substituting \eqref{applyYong}  into \eqref{S11} yields
\begin{equation}\label{est:S11}
S_{1,1}\leq Cr^{-\frac{9}{2}}\abs{u}_{X_{\rho}}^2\abs{\nabla_{\rm h}u}_{X_{\rho}}.
\end{equation}
Following the treatment of $S_{1,1}$, we use the inequality (see Appendix \ref{sec:appendix} for the proof)
\begin{equation}\label{ineq2}
 \forall\ k\in\mathbb Z_+\, \textrm{ with } \ \big [\frac{m}{2}\big]+1\leq k\leq m,
 \quad \binom{m}{k}\frac{L_{\rho,m}}{L_{\rho,k}L_{\rho,m-k}}\leq\frac{C}{(m-k+1)^6},
\end{equation}
 to obtain
\begin{equation*}%\label{est:S12}
S_{1,2}\leq Cr^{-\frac{9}{2}}\abs{u}_{X_{\rho}}^2\abs{\nabla_{\rm h}u}_{X_{\rho}}.
\end{equation*}
Substituting this and estimate \eqref{est:S11} into \eqref{S1} yields
\begin{equation}\label{teces}
\begin{aligned}
	S_1& \leq C r^{-\frac{9}{2}}\abs{u}_{X_{\rho}}^2\abs{\nabla_{\rm h}u}_{X_{\rho}}\leq C r^{-\frac{9}{2}}\abs{u}_{X_{\rho}}\abs{u}_{Y_{\rho}}\abs{\nabla_{\rm h}u}_{X_{\rho}}\\
	&\leq   C r^{-\frac{9}{2}}\abs{u}_{X_{\rho}}\abs{u}_{Y_{\rho}}^2+ Cr^{-\frac{9}{2}}\abs{u}_{X_{\rho}} \abs{\nabla_{\rm h}u}_{X_{\rho}}^2,
	\end{aligned}
\end{equation}
the second inequality holding because of  the fact \eqref{xyone}. Thus
 estimate \eqref{est:S1} follows. The proof of Lemma \ref{lem:S1} is completed.
\end{proof}

\begin{lemma}[Estimate on $S_2$]\label{lem:S2}
Let $S_2$ be given in \eqref{S1-S4}, namely,
\begin{equation*}
	S_2=-\sum_{m= 0}^{+\infty} r^{\alpha_1} L^2_{\rho,m}\inner{\partial_3^{m}\partial^{\alpha}  (u_3\partial_3u),\ \partial_3^{m}\partial^{\alpha}  u}_{L^2}.
\end{equation*}
Then for any $\alpha=(\alpha_1,\alpha_2,\alpha_3)\in\mathbb Z_+^3$ with $\abs\alpha \leq 3,$  it holds that
\begin{equation}\label{est:S2}
    S_2\leq Cr^{-\frac{9}{2}}\abs{u}_{X_{\rho}}\abs{u}_{Y_{\rho}}^2,
\end{equation}
 where the norms $\abs{\cdot}_{X_{\rho}}$ and $\abs{\cdot}_{Y_{\rho}}$ are given in Definition \ref{def:normx3}.
\end{lemma}

\begin{proof}
We split $S_2$ as  follows:
\begin{equation*}
\begin{aligned}
 S_2=&- \sum_{m= 0}^{+\infty}r^{  \alpha_1}L^2_{\rho,m}\inner{\partial^{ \alpha }(u_3\partial_3^{m+1}u),\ \partial_3^m\partial^{ \alpha }u}_{L^2}\\
&- \sum_{m= 0}^{+\infty}\sum_{k=1}^{[\frac{m}{2}]}r^{\alpha_1}\binom{m}{k}L^2_{\rho,m}\inner{\partial^{\alpha}[(\partial_3^ku_3)\partial_3^{m-k+1}u],\ \partial_3^{m}\partial^{\alpha}u}_{L^2}\\
&-\sum_{m= 0}^{+\infty}\,\sum_{k=[\frac{m}{2}]+1}^{m}r^{\alpha_1}\binom{m}{k}L^2_{\rho,m}\inner{\partial^{\alpha}[(\partial_3^ku_3)\partial_3^{m-k+1}u],\ \partial_3^{m}\partial^{\alpha}u}_{L^2}.
\end{aligned}
\end{equation*}
This yields, for any     $\alpha\in\mathbb Z_+^3$ with $\abs\alpha\leq 3,$
\begin{equation}\label{+s2}
    \begin{aligned}
        S_2\leq &   \sum_{m= 0}^{+\infty} L^2_{\rho,m} \left|\inner{\partial^{ \alpha }(u_3\partial_3^{m+1}u),\ \partial_3^m\partial^{ \alpha }u}_{L^2}\right|\\
        &+C\sum_{m= 0}^{+\infty}\sum_{k=1}^{[\frac{m}{2}]} \binom{m}{k}L^2_{\rho,m} \norm{\partial_3^ku}_{H^3}\norm{\partial_3^{m-k+1}u}_{H^3}\norm{\partial_3^mu}_{H^3}\\
        &+C \sum_{m= 0}^{+\infty}\,\sum_{k=[\frac{m}{2}]+1}^{m} \binom{m}{k}L^2_{\rho,m}\norm{\partial_3^ku}_{H^3}\norm{\partial_3^{m-k+1}u}_{H^3}\norm{\partial_3^mu}_{H^3}.
    \end{aligned}
\end{equation}
We begin by estimating the first term on the right-hand side of \eqref{+s2}.  For any $\alpha\in\mathbb Z_+^3$ with $\abs\alpha\leq 3,$ Using Leibniz's formula and  integration by parts yields
\begin{align*}
   & \left|\inner{\partial^{ \alpha }(u_3\partial_3^{m+1}u),\ \partial_3^m\partial^{ \alpha }u}_{L^2}\right|
   \\
  &\leq \left|\inner{ u_3\partial_3^{m+1}\partial^{ \alpha }u ,\ \partial_3^m\partial^{ \alpha }u}_{L^2}\right|+C\sum_{0<\gamma\leq \alpha}\binom{\alpha}{\gamma}\left|\inner{(\partial^\gamma u_3)\partial_3^{m+1}\partial^{ \alpha-\gamma }u,\ \partial_3^m\partial^{ \alpha }u}_{L^2}\right|\\
  &\leq \left|\inner{ (\partial_3u_3)\partial_3^{m}\partial^{ \alpha }u ,\ \partial_3^m\partial^{ \alpha }u}_{L^2}\right|+C\sum_{0<\gamma\leq \alpha} \left|\inner{(\partial^\gamma u_3)\partial_3^{m+1}\partial^{ \alpha-\gamma }u,\ \partial_3^m\partial^{ \alpha }u}_{L^2}\right|
  \\
   &\leq C \norm{u}_{H^3} \norm{\partial_3^mu}_{H^3}^2,
\end{align*}
where the last inequality uses  the following  Sobolev embedding  inequalities:
\begin{equation*}
	\norm{f}_{L^\infty}\leq C\norm{f}_{H^2}\ \textrm{ and }\ \norm{fg}_{L^2}\leq   C\norm{f}_{H^1}\norm{g}_{H^1}.
\end{equation*}
Combining this with the fact that   $\norm{u}_{H^3} \leq Cr^{-\frac32} \abs{u}_{X_{\rho}}$ by  \eqref{boundrho}, we obtain
\begin{multline}\label{est:S21}
 \sum_{m= 0}^{+\infty} L^2_{\rho,m} \left|\inner{\partial^{ \alpha }(u_3\partial_3^{m+1}u),\ \partial_3^m\partial^{ \alpha }u}_{L^2}\right|\leq Cr^{-\frac32} \abs{u}_{X_{\rho}}\sum_{m= 0}^{+\infty}L^2_{\rho,m}  \norm{\partial_3^mu}_{H^3}^2\\
\leq Cr^{-\frac{9}{2}}\abs{u}_{X_{\rho}}^3\leq Cr^{-\frac{9}{2}}\abs{u}_{X_{\rho}}\abs{u}_{Y_{\rho}}^2,
\end{multline}
where the last inequality follows from \eqref{xyone}.

To estimate the remaining terms on the right-hand side of \eqref{+s2},  we use the following two inequalities (see Appendix \ref{sec:appendix} for the proofs). The first one   states  that for any $k\in\mathbb Z_+$ with $1\leq k\leq [\frac{m}{2}],$
\begin{equation}\label{ineq3}
 \binom{m}{k}\frac{L_{\rho,m}}{L_{\rho,k}L_{\rho,m-k+1}(m-k+2)^\frac{1}{2}(m+1)^\frac{1}{2}}\leq \frac{C}{(k+1)^6},
\end{equation}
while the second  one is that
\begin{equation}\label{ineq4}
\forall\ k\in\mathbb Z_+\ \textrm{ with }\  \Big[\frac{m}{2}\Big]+1\leq k\leq m,\quad  \binom{m}{k}\frac{L_{\rho,m}}{L_{\rho,k}L_{\rho,m-k+1}}\leq \frac{C}{(m-k+1)^6}.
\end{equation}
These two inequalities, together with    Young's inequality \eqref{young} for discrete convolution,  allow us to follow a similar argument to those in \eqref{S11} and \eqref{applyYong}, to obtain
\begin{equation*}%\label{est:S22}
    \begin{aligned}
&\sum_{m= 0}^{+\infty}\sum_{k=1}^{[\frac{m}{2}]} \binom{m}{k}L^2_{\rho,m} \norm{\partial_3^ku}_{H^3}\norm{\partial_3^{m-k+1}u}_{H^3}\norm{\partial_3^mu}_{H^3}\\
&\leq C\sum_{m=0}^{+\infty}\,\sum_{k=1}^{[\frac{m}{2}]}\binom{m}{k}\frac{L_{\rho,m}}{L_{\rho,k}L_{\rho,m-k+1}(m-k+2)^\frac{1}{2}(m+1)^\frac{1}{2}}(L_{\rho,k}\norm{\partial_3^ku}_{H^3})\\
&\quad\times\big[(m-k+2)^\frac{1}{2}L_{\rho,m-k+1}\norm{\partial_3^{m-k+1}u}_{H^3}\big]\times\big[(m+1)^\frac{1}{2}L_{\rho,m}\norm{\partial_3^mu}_{H^3}\big]\\
& \leq  Cr^{-\frac{9}{2}}\abs{u}_{X_{\rho}}\abs{u}_{Y_{\rho}}^2
    \end{aligned}
\end{equation*}
and
\begin{equation*}%\label{est:S23}
    \begin{aligned}
&\sum_{m= 0}^{+\infty}\,\sum_{k=[\frac{m}{2}]+1}^{m} \binom{m}{k}L^2_{\rho,m}\norm{\partial_3^ku}_{H^3}\norm{\partial_3^{m-k+1}u}_{H^3}\norm{\partial_3^mu}_{H^3}\\
&\leq C\sum_{m=0}^{+\infty}\,\sum_{k=[\frac{m}{2}]+1}^{m}\binom{m}{k}\frac{L_{\rho,m}}{L_{\rho,k}L_{\rho,m-k+1}}(L_{\rho,k}\norm{\partial_3^ku}_{H^3})\\
&\quad\times\big(L_{\rho,m-k+1}\norm{\partial_3^{m-k+1}u}_{H^3}\big)\big(L_{\rho,m}\norm{\partial_3^mu}_{H^3}\big)\\
& \leq  Cr^{-\frac{9}{2}}\abs{u}_{X_{\rho}}^3 \leq  Cr^{-\frac{9}{2}}\abs{u}_{X_{\rho}}\abs{u}_{Y_{\rho}}^2,
    \end{aligned}
\end{equation*}
where the last inequality follows from \eqref{xyone}.
Substituting the above two estimates and  \eqref{est:S21}  into \eqref{+s2} yields the desired estimate \eqref{est:S2}. The proof of Lemma \ref{lem:S2} is thus completed.
\end{proof}

\begin{lemma}[Estimate on $S_3$]\label{lem:S3}
Let $S_3$ be given in \eqref{S1-S4}, that is,
\begin{equation*}
	S_3= \sum_{m= 0}^{+\infty}r^{\alpha_1} L^2_{\rho,m}\inner{\partial_3^m\partial ^{\tilde\alpha}\nabla p, \ \partial_3^{m}\partial^{\alpha} \partial_1 u}_{L^2}\ \textrm{ with }\ \tilde\alpha =\alpha-(1,0,0).
\end{equation*}
  Then for any $ \alpha \in\mathbb{Z}_+^3$ with $\abs{ \alpha }\leq 3$ and $ \alpha_1\ge 1$, it holds that
\begin{equation}\label{est:S3}
    S_3\leq Cr^{-\frac92}\abs{u}_{X_{\rho}} \abs{u}_{Y_{\rho}}^2+C\Big( r^{-\frac92}\abs{u}_{X_{\rho}} +r^{\frac{1}{4}}\Big)\abs{\nabla_{\rm h}u}_{X_{\rho}}^2,
\end{equation}
 where the norms $\abs{\cdot}_{X_{\rho}}$ and $\abs{\cdot}_{Y_{\rho}}$ are given in Definition \ref{def:normx3}.
\end{lemma}

\begin{proof} For any   $ \alpha \in\mathbb{Z}_+^3$ with $\abs{ \alpha }\leq 3$ and $ \alpha_1\ge 1$, we have
\begin{align*}
S_3&\leq \sum_{m= 0}^{+\infty}r^{  \alpha_1}L^2_{\rho,m}\norm{\partial_3^m\partial^{ \tilde \alpha}  \nabla p}_{L^2}\norm{ \partial_3^{m}\partial^{\alpha}\partial_1u}_{L^2}\\
&\leq r^{\frac{\alpha_1}{2}} \Big(\sum_{m= 0}^{+\infty}L^2_{\rho,m}\norm{\partial_3^m\partial^{ \tilde \alpha}  \nabla p}_{L^2}^2\Big)^{\frac12}\Big(\sum_{m= 0}^{+\infty}r^{  \alpha_1}L^2_{\rho,m} \norm{ \partial_3^{m}\partial^{\alpha}\partial_1u}_{L^2}^2\Big)^{\frac12}\\
&\leq C r^{\frac{ \alpha_1}{2}} \abs{\nabla_{\rm h} u}_{X_{\rho}}\Big(\sum_{m= 0}^{+\infty}L^2_{\rho,m}\norm{\partial_3^m\partial^{ \tilde \alpha}  \nabla p}_{L^2}^2\Big)^{\frac12}.
\end{align*}
On the other hand, for any  $ \alpha \in\mathbb{Z}_+^3$ with $\abs{ \alpha }\leq 3$ and $ \alpha_1\ge 1$, we claim that (recalling $\tilde{\alpha} = \alpha - (1,0,0)$)
\begin{equation}\label{est:Pone}
\Big(\sum_{m= 0}^{+\infty}L^2_{\rho,m}\norm{\partial_3^m\partial^{ \tilde \alpha}  \nabla p}_{L^2}^2\Big)^{\frac12}\leq Cr^{-\frac92}\abs{u}_{X_{\rho}}^2+Cr^{-\frac{2\alpha_1-1}{4}}\abs{\nabla_{\rm h}u}_{X_{\rho}}
	\end{equation}
	with the proof postponed to the following four steps.  Combining the above inequalities  yields
	\begin{align*}
		S_3 &\leq  Cr^{-\frac92}\abs{u}_{X_{\rho}}^2\abs{\nabla_{\rm h}u}_{X_{\rho}}+Cr^{\frac{1}{4}}\abs{\nabla_{\rm h}u}_{X_{\rho}}^2\\
		&\leq Cr^{-\frac92}\abs{u}_{X_{\rho}} \abs{u}_{Y_{\rho}}^2+C\Big( r^{-\frac92}\abs{u}_{X_{\rho}} +r^{\frac{1}{4}}\Big)\abs{\nabla_{\rm h}u}_{X_{\rho}}^2,
	\end{align*}
	the last inequality following from the same argument as that in \eqref{teces}. This gives the desired
	 estimate \eqref{est:S3}.

{\it Step 1.} This step, together with the other three ones that follow,  is devoted to proving estimate \eqref{est:Pone}.
Observe $p$ satisfies the elliptic equation $-\Delta p=\divv \big( (u\cdot\nabla ) u \big )$ and thus
 \begin{equation}\label{eq:ptwo}
 -\partial_3^m\partial^{\tilde{\alpha}} \Delta p=\partial_3^m\partial^{\tilde{\alpha}}\divv \big( (u\cdot\nabla ) u \big ),
 \end{equation}
 which is complemented with the boundary condition that
 \begin{equation}\label{boundary:Pone}
\partial_3^m\partial^{\tilde  \alpha } \partial_1 p|_{x_1=0,1}=-\partial_3^m\partial^{\tilde{\alpha}}\big (\partial_tu_1+(u\cdot\nabla) u_1-\Delta_{\rm h}u_1\big )\big|_{x_1=0,1}.
 \end{equation}
Taking the $L^2$ inner product with $\partial_3^m\partial^{ \tilde \alpha} p$ on both sides of  equation \eqref{eq:ptwo}, we obtain
 \begin{align*}
 & \norm{\partial_3^m\partial^{ \tilde \alpha}  \nabla p}_{L^2}^2-\int_{\mathbb R^2}   (\partial_3^m\partial^{\tilde  \alpha } \partial_1 p)(\partial_3^m\partial^{\tilde  \alpha} p)\big|^{x_1=1}_{x_1=0} \  dx_2dx_3\\
&  =-\inner{\partial_3^m\partial^{\tilde{\alpha}}  \big( (u\cdot\nabla ) u \big ),  \partial_3^m\partial^{\tilde\alpha} \nabla p }_{L^2}+\int_{\mathbb R^2}    \Big (\partial_3^m\partial^{\tilde{\alpha}}  \big( (u\cdot\nabla ) u_1 \big )\Big)(\partial_3^m\partial^{\tilde \alpha} p)\big |^{x_1=1}_{x_1=0}   dx_2dx_3.
\end{align*}
Combining the above identity  with \eqref{boundary:Pone}, and  multiplying  both sides  by $L_{\rho,m}^2$ and then summing over  $m\in\mathbb{Z}_+$, we obtain
\begin{equation}\label{p+}
	\begin{aligned}
		& \sum_{m= 0}^{+\infty}L^2_{\rho,m}\norm{\partial_3^m\partial^{ \tilde \alpha}  \nabla p}_{L^2}^2\leq  \sum_{m=0}^{+\infty}L_{\rho,m}^2\norm{\partial_3^m\partial^{\tilde{\alpha}} \big (u\cdot\nabla) u\big)}_{L^2}\norm{\partial_3^m\partial^{\tilde{\alpha}} \nabla p}_{L^2}\\
		&\qquad\qquad\qquad+\sum_{m= 0}^{+\infty}L^2_{\rho,m}\bigg|\int_{\mathbb R^2}   (\partial_3^m\partial^{\tilde{\alpha}} \partial_tu_1)(\partial_3^m\partial^{\tilde  \alpha} p)\big|^{x_1=1}_{x_1=0} \  dx_2dx_3\bigg|\\
		&\qquad\qquad\qquad+\sum_{m= 0}^{+\infty}L^2_{\rho,m}\bigg|\int_{\mathbb R^2}   (\partial_3^m\partial^{\tilde{\alpha}} \Delta_{\rm h} u_1)(\partial_3^m\partial^{\tilde  \alpha} p)\big|^{x_1=1}_{x_1=0} \  dx_2dx_3\bigg|.
	\end{aligned}
\end{equation}
We begin by estimating  the first term on the right-hand side of \eqref{p+}.  Since $\tilde \alpha=\alpha-(1,0,0)$ and $\abs\alpha\leq 3$, it follows that  $|\tilde \alpha|\leq 2 $. Then
 \begin{align*}
\norm{\partial_3^m\partial^{\tilde{\alpha}} \big ( (u\cdot\nabla) u\big)}_{L^2}\leq \sum_{i=1}^3\norm{\partial_3^m\partial^{\tilde{\alpha}}\partial_i   (  u_i  u )}_{L^2}\leq C\sum_{k=0}^m\binom{m}{k}\norm{\partial_3^ku}_{H^3}\norm{\partial_3^{m-k}u}_{H^3},
 \end{align*}
where the second inequality uses  the fact that $\divv u=0.$ Consequently,
\begin{equation}\label{est:tildeS1}
	\begin{aligned}
	&\sum_{m=0}^{+\infty}L_{\rho,m}^2\norm{\partial_3^m\partial^{\tilde{\alpha}} \big (u\cdot\nabla) u\big)}_{L^2}\norm{\partial_3^m\partial^{\tilde{\alpha}} \nabla p}_{L^2}\\
	&\leq \sum_{m=0}^{+\infty}\sum_{k=0}^m\binom{m}{k}L_{\rho,m}^2\norm{\partial_3^ku}_{H^3}\norm{\partial_3^{m-k}u}_{H^3}\norm{\partial_3^m\partial^{\tilde{\alpha}}\nabla p}_{L^2}	\\
	&\leq Cr^{-\frac92}\abs{u}_{X_{\rho}}^2 \Big(\sum_{m= 0}^{+\infty}L^2_{\rho,m}\norm{\partial_3^m\partial^{ \tilde \alpha}  \nabla p}_{L^2}^2\Big)^{\frac12},
	\end{aligned}
\end{equation}
where the last inequality follows by   the same argument as  in the proof of Lemma \ref{lem:S1}.

{\it Step 2.}  Here we estimate the second term on the right-hand side of \eqref{p+} and claim that (recalling $\tilde\alpha=(\alpha_1-1,\alpha_2,\alpha_3)$)
\begin{multline}\label{df}
	\sum_{m= 0}^{+\infty}L^2_{\rho,m}\bigg|\int_{\mathbb R^2}   (\partial_3^m\partial^{\tilde{\alpha}} \partial_tu_1)(\partial_3^m\partial^{\tilde  \alpha} p)\big|^{x_1=1}_{x_1=0} \  dx_2dx_3\bigg|\\
	 \leq  C r^{-\frac{2\alpha_1-1}{4}} \abs{\nabla_{\rm h} u}_{X_{\rho}}  \Big(\sum_{m=0}^{+\infty}L_{\rho,m}^2   \norm{\partial_3^m\partial^{\tilde\alpha}\nabla p}_{L^{2}} ^2\Big)^{\frac12}.
\end{multline}
Since $
   u_1|_{x_1=0,1}=  \partial_1u_1|_{x_1=0,1}=0,$ then for multi-indices $\tilde\alpha=\alpha-(1,0,0)\in\mathbb Z_+^3$  with $\tilde\alpha_1\leq 1$ we have
   \begin{equation*}
	\sum_{m= 0}^{+\infty}L^2_{\rho,m}\bigg|\int_{\mathbb R^2}   (\partial_3^m\partial^{\tilde{\alpha}} \partial_tu_1)(\partial_3^m\partial^{\tilde  \alpha} p)\big|^{x_1=1}_{x_1=0} \  dx_2dx_3\bigg| = 0.
\end{equation*}
This, with the fact $|\tilde\alpha|\leq 2$,  yields the validity of \eqref{df} for  multi-indices   $ \tilde \alpha\in\mathbb Z_+^3$ with $\tilde\alpha\neq (2,0,0).$

Now we consider the case $\tilde\alpha=  (2,0,0)$ and  use the fact $\partial^{\tilde\alpha}=\partial_1^2$ and  the identity
$ \partial^{\tilde\alpha}u_1= \partial_1^2u_1 =- \partial_1\partial_2u_2-\partial_1\partial_3u_3
$  to write
 \begin{equation}\label{+tildes2}
 \begin{aligned}
&\sum_{m= 0}^{+\infty}L^2_{\rho,m}\bigg|\int_{\mathbb R^2}   (\partial_3^m\partial^{\tilde{\alpha}} \partial_tu_1)(\partial_3^m\partial^{\tilde  \alpha} p)\big|^{x_1=1}_{x_1=0} \  dx_2dx_3\bigg| \\
&\leq  2\sup_{x_1\in[0,1]} \sum_{m=0}^{+\infty}L_{\rho,m}^2\left|\int_{\mathbb{R}^2}(\partial_3^m\partial_t\partial_1\partial_2u_2)(\partial_3^m\partial_1^{2} p)dx_2dx_3\right|\\
&\quad+2\sup_{x_1\in[0,1]} \sum_{m=0}^{+\infty}L_{\rho,m}^2\left|\int_{\mathbb{R}^2}(\partial_3^m\partial_t\partial_1\partial_3u_3)(\partial_3^m\partial_1^{2} p)dx_2dx_3\right|.
 \end{aligned}
 \end{equation}
To estimate the first term on the right-hand side of \eqref{+tildes2},  we  denote by $\widehat{h}(t,x_1,\xi_2,x_3)$ the (partial) Fourier transform of $h(t, x_1,x_2,x_3)$ with respect to $x_2\in\mathbb R,$ and by   $\xi_2$   the   Fourier dual variable of $x_2$. Then by the Plancherel Theorem and the following  one-dimensional  Gagliardo-Nirenberg inequality
\begin{equation}\label{GN}
\norm{h}_{L_y^{\infty}}\leq C \norm{h}_{L_y^2}^{\frac{1}{2}}\left(\norm{h}_{L_y^2}^{\frac{1}{2}}+\norm{\partial_yh}_{L_y^2}^{\frac{1}{2}}\right),
\end{equation}
we obtain
\begin{multline}\label{q1q2}
	\sup_{x_1\in[0,1]} \sum_{m=0}^{+\infty}L_{\rho,m}^2\left|\int_{\mathbb{R}^2}(\partial_3^m\partial_t\partial_1\partial_2u_2)(\partial_3^m\partial_1^{2} p)dx_2dx_3\right|	 \\
	  \leq \sup_{x_1\in[0,1]} \sum_{m=0}^{+\infty}L_{\rho,m}^2\left|\int_{\mathbb{R}^2}\xi_2 (\partial_3^m\partial_t\partial_1\widehat u_2)(\partial_3^m\partial_1^{2} \widehat p) d\xi_2 dx_3\right|\leq Q_1+Q_2,
\end{multline}
where
\begin{multline*}
	Q_1=C\sum_{m=0}^{+\infty}L_{\rho,m}^2 \int_{\mathbb R^2}\abs{\xi_2}\norm{
\partial_3^m\partial_t\partial_1\widehat{ u}_2}_{L_{x_1}^{2}}^\frac{1}{2}\Big(\norm{
\partial_3^m\partial_t\partial_1\widehat{ u}_2}_{L_{x_1}^{2}}^\frac{1}{2}+\norm{
\partial_3^m\partial_t\partial_1^2\widehat{ u}_2}_{L_{x_1}^{2}}^\frac{1}{2}\Big)\\ \times\norm{\partial_3^m\partial_1^2\widehat{p}}_{L_{x_1}^{2}}d\xi_2dx_3,
\end{multline*}
and
\begin{multline*}
	Q_2=C \sum_{m=0}^{+\infty}L_{\rho,m}^2 \int_{\mathbb R^2}\abs{\xi_2}\norm{
\partial_3^m\partial_t\partial_1\widehat{ u}_2}_{L_{x_1}^{2}}^\frac{1}{2}\Big(\norm{
\partial_3^m\partial_t\partial_1\widehat{ u}_2}_{L_{x_1}^{2}}^\frac{1}{2}+\norm{
\partial_3^m\partial_t\partial_1^2\widehat{ u}_2}_{L_{x_1}^{2}}^\frac{1}{2}\Big)\\
\times\norm{\partial_3^m\partial_1^2\widehat{p}}_{L_{x_1}^{2}}^\frac{1}{2} \norm{\partial_3^m\partial_1^3\widehat{p}}_{L_{x_1}^{2}}^\frac{1}{2} d\xi_2dx_3.
\end{multline*}
Moreover,  using $ \abs{\xi_2}\norm{\partial_3^m\partial_1^2\widehat{p}}_{L_{x_1}^{2}} \leq  \norm{\partial_3^m\partial_1^2\widehat{\partial_2p}}_{L_{x_1}^{2}}\leq  \norm{\partial_3^m\partial_1^2\widehat{\nabla p}}_{L_{x_1}^{2}}$ gives
\begin{equation}\label{q1}
\begin{aligned}
	Q_1&\leq C \sum_{m=0}^{+\infty}L_{\rho,m}^2 \int_{\mathbb R^2}
 \big(\norm{
\partial_3^m\partial_t\partial_1\widehat{ u}_2}_{L_{x_1}^{2}} +\norm{
\partial_3^m\partial_t\partial_1^2\widehat{ u}_2}_{L_{x_1}^{2}} \big)\norm{\partial_3^m\partial_1^2 \widehat{\nabla p}}_{L_{x_1}^{2}} d\xi_2dx_3\\
&\leq C \sum_{m=0}^{+\infty}L_{\rho,m}^2   \big(\norm{
\partial_3^m\partial_t\partial_1 u _2}_{L^{2}} +\norm{
\partial_3^m\partial_t\partial_1^2u_2}_{L^{2}} \big)\norm{\partial_3^m\partial_1^2\nabla p}_{L^{2}} \\
&\leq C r^{-\frac12} \abs{\nabla_{\rm h} u}_{X_{\rho}}  \Big(\sum_{m=0}^{+\infty}L_{\rho,m}^2   \norm{\partial_3^m\partial_1^2\nabla p}_{L^{2}} ^2\Big)^{\frac12},
\end{aligned}
\end{equation}
the last inequality using the definition of $\abs{\cdot}_{X_{\rho}}$ (see Definition \ref{def:normx3}).
Similarly, observing
\begin{equation*}
	\abs{\xi_2}\norm{
\partial_3^m\partial_t\partial_1\widehat{ u}_2}_{L_{x_1}^{2}}^\frac{1}{2}\norm{\partial_3^m\partial_1^2\widehat{p}}_{L_{x_1}^{2}}^\frac{1}{2}\leq \norm{\partial_3^m\partial_t\partial_1\widehat{ \partial_2u}_2}_{L_{x_1}^{2}}^\frac{1}{2}\norm{\partial_3^m\partial_1^2\widehat{\partial_2p}}_{L_{x_1}^{2}}^\frac{1}{2},
\end{equation*}
we have
\begin{equation}\label{q2}
	\begin{aligned}
Q_2&\leq C	\sum_{m=0}^{+\infty}L_{\rho,m}^2 \int_{\mathbb R^2} \norm{
\partial_3^m\partial_t\partial_1\widehat{ \partial_2u}_2}_{L_{x_1}^{2}}^\frac{1}{2}\big(\norm{
\partial_3^m\partial_t\partial_1\widehat{ u}_2}_{L_{x_1}^{2}}^\frac{1}{2}+\norm{
\partial_3^m\partial_t\partial_1^2\widehat{ u}_2}_{L_{x_1}^{2}}^\frac{1}{2}\big)\\
&\qquad\qquad\qquad\qquad \qquad\qquad\qquad\times\norm{\partial_3^m\partial_1^2\widehat{\partial_2p}}_{L_{x_1}^{2}}^\frac{1}{2} \norm{\partial_3^m\partial_1^3\widehat{p}}_{L_{x_1}^{2}}^\frac{1}{2} d\xi_2dx_3\\
&\leq  C \sum_{m=0}^{+\infty}L_{\rho,m}^2  \norm{\partial_3^m\partial_t\partial_1\partial_2u}_{L^2}^{\frac{1}{2}}\big(\norm{\partial_3^m\partial_t\partial_1u}_{L^2}^{\frac{1}{2}}+\norm{\partial_3^m\partial_t\partial_1^2u}_{L^2}^{\frac{1}{2}}\big)\norm{\partial_3^m\partial_1^2\nabla p}_{L^2}\\
&\leq C r^{-\frac12} \abs{\nabla_{\rm h} u}_{X_{\rho}}  \Big(\sum_{m=0}^{+\infty}L_{\rho,m}^2   \norm{\partial_3^m\partial_1^2\nabla p}_{L^{2}} ^2\Big)^{\frac12}.
	\end{aligned}
\end{equation}
Substituting  \eqref{q1} and \eqref{q2} into \eqref{q1q2} yields
\begin{multline*}%\label{Aug20}
	\sup_{x_1\in[0,1]} \sum_{m=0}^{+\infty}L_{\rho,m}^2\left|\int_{\mathbb{R}^2}(\partial_3^m\partial_t\partial_1\partial_2u_2)(\partial_3^m\partial_1^{2} p)dx_2dx_3\right|\\
	\leq
	C r^{-\frac12} \abs{\nabla_{\rm h} u}_{X_{\rho}}  \Big(\sum_{m=0}^{+\infty}L_{\rho,m}^2   \norm{\partial_3^m\partial_1^2\nabla p}_{L^{2}} ^2\Big)^{\frac12}.
\end{multline*}
In the same way,   performing the (partial) Fourier transform with respect to $x_3$ and then following  the same argument as above, we have
\begin{multline*}
\sup_{x_1\in[0,1]} \sum_{m=0}^{+\infty}L_{\rho,m}^2\left|\int_{\mathbb{R}^2}(\partial_3^m\partial_t\partial_1\partial_3u_3)(\partial_3^m\partial_1^{2} p)dx_2dx_3\right|\\
\leq C r^{-\frac12} \abs{\nabla_{\rm h} u}_{X_{\rho}}  \Big(\sum_{m=0}^{+\infty}L_{\rho,m}^2   \norm{\partial_3^m\partial_1^2\nabla p}_{L^{2}} ^2\Big)^{\frac12}.
\end{multline*}
Combining   the   two estimates above  with \eqref{+tildes2}, we obtain that, for $\tilde\alpha=\alpha-(1,0,0)=(2,0,0)$,
\begin{multline*}
	 \sum_{m= 0}^{+\infty}L^2_{\rho,m}\bigg|\int_{\mathbb R^2}   (\partial_3^m\partial^{\tilde{\alpha}} \partial_tu_1)(\partial_3^m\partial^{\tilde  \alpha} p)\big|^{x_1=1}_{x_1=0} \  dx_2dx_3\bigg|\\
	  \leq C r^{-\frac12} \abs{\nabla_{\rm h} u}_{X_{\rho}}  \Big(\sum_{m=0}^{+\infty}L_{\rho,m}^2   \norm{\partial_3^m\partial_1^2\nabla p}_{L^{2}} ^2\Big)^{\frac12}\\
	  \leq C r^{-\frac{2\alpha_1-1}{4}} \abs{\nabla_{\rm h} u}_{X_{\rho}}  \Big(\sum_{m=0}^{+\infty}L_{\rho,m}^2   \norm{\partial_3^m\partial^{\tilde\alpha}\nabla p}_{L^{2}} ^2\Big)^{\frac12}.	
	\end{multline*}
This with the validity  for $\tilde\alpha\neq(2,0,0)$
yields the desired assertion \eqref{df}.

 {\it Step 3.} It remains to treat the last term  in \eqref{p+} and prove that, for $\tilde\alpha=\alpha-(1,0,0)\in\mathbb Z_+^3$ with $\abs\alpha\leq3,$
 \begin{multline}\label{est:tildeS3}
 	\sum_{m= 0}^{+\infty}L^2_{\rho,m}\bigg|\int_{\mathbb R^2}   (\partial_3^m\partial^{\tilde{\alpha}} \Delta_{\rm h} u_1)(\partial_3^m\partial^{\tilde  \alpha} p)\big|^{x_1=1}_{x_1=0} \  dx_2dx_3\bigg|\\
 	\leq C r^{-\frac{2\alpha_1-1}{4} }\abs{\nabla_{\rm h}u}_{X_{\rho}}\Big( \sum_{m=0}^{+\infty}L_{\rho,m}^2  \norm{\partial_3^m\partial^{\tilde{\alpha}}\nabla p}_{L^2}^2\Big)^{\frac12}.
 \end{multline}
 To do so, using the fact that $\partial^{\tilde\alpha}\partial_1=\partial^\alpha$ and $\divv u=0$, we split the left-hand side term into three parts:
 \begin{equation}\label{+tildes3}
 \begin{aligned}
	&\sum_{m= 0}^{+\infty}L^2_{\rho,m}\bigg|\int_{\mathbb R^2}   (\partial_3^m\partial^{\tilde{\alpha}} \Delta_{\rm h} u_1)(\partial_3^m\partial^{\tilde  \alpha} p)\big|^{x_1=1}_{x_1=0} \  dx_2dx_3\bigg|\\
	 & \leq  2\sup_{x_1\in[0,1]} \sum_{m=0}^{+\infty}L_{\rho,m}^2\left|\int_{\mathbb{R}^2}(\partial_3^{m}\partial^{\alpha}\partial_2u_2)(\partial_3^m\partial^{\tilde  \alpha} p) dx_2dx_3\right|\\
&\qquad +2\sup_{x_1\in[0,1]} \sum_{m=0}^{+\infty}L_{\rho,m}^2\left|\int_{\mathbb{R}^2}(\partial_3^{m}\partial^{\alpha}\partial_3u_3)(\partial_3^m\partial^{\tilde  \alpha} p)dx_2dx_3\right|\\
&\qquad\qquad+2\sup_{x_1\in[0,1]} \sum_{m=0}^{+\infty}L_{\rho,m}^2\left|\int_{\mathbb{R}^2}(\partial_3^m \partial^{\tilde{\alpha}}\partial_2^2u_1)(\partial_3^m\partial^{\tilde  \alpha} p) dx_2dx_3\right|.
 \end{aligned}
 \end{equation}
Applying  the  Plancherel Theorem and the  Gagliardo-Nirenberg inequality \eqref{GN},  and denoting by  $\hat h$  the Fourier transform of $h$ with respect to $x_2$, we estimate the first term on the right-hand side of \eqref{+tildes3} as follows:
\begin{equation}\label{r1r2}
	\sup_{x_1\in[0,1]} \sum_{m=0}^{+\infty}L_{\rho,m}^2\left|\int_{\mathbb{R}^2}(\partial_3^{m}\partial^{\alpha}\partial_2u_2)(\partial_3^m\partial^{\tilde  \alpha} p) dx_2dx_3\right|\leq R_1+R_2,
\end{equation}
where
\begin{multline*}
	R_1=C\sum_{m=0}^{+\infty} L_{\rho,m}^2   \int_{\mathbb R^2} \abs{\xi_2}\norm{
\partial_3^m\widehat{ \partial^{\alpha}u}_2}_{L_{x_1}^{2}}^\frac{1}{2}\Big(\norm{
\partial_3^m\widehat{ \partial^{\alpha}u}_2}_{L_{x_1}^{2}}^\frac{1}{2}+\norm{
\partial_3^m\partial_1\widehat{ \partial^{\alpha}u}_2}_{L_{x_1}^{2}}^\frac{1}{2}\Big)\\
 \times  \norm{\partial_3^m\widehat{\partial^{\tilde{\alpha}}p}}_{L_{x_1}^{2}} d\xi_2dx_3,
 \end{multline*}
 and
 \begin{multline*}
 	R_2=  C\sum_{m=0}^{+\infty} L_{\rho,m}^2\int_{\mathbb R^2} \abs{\xi_2}\norm{
\partial_3^m\widehat{ \partial^{\alpha}u}_2}_{L_{x_1}^{2}}^\frac{1}{2}\Big(\norm{
\partial_3^m\widehat{ \partial^{\alpha}u}_2}_{L_{x_1}^{2}}^\frac{1}{2}+\norm{
\partial_3^m\partial_1\widehat{ \partial^{\alpha}u}_2}_{L_{x_1}^{2}}^\frac{1}{2}\Big)\\
 \times   \norm{\partial_3^m\widehat{\partial^{\tilde{\alpha}}p}}_{L_{x_1}^{2}}^\frac{1}{2}\norm{\partial_3^m\partial_1\widehat{\partial^{\tilde{\alpha}}p}}_{L_{x_1}^{2}}^\frac{1}{2} d\xi_2dx_3.
 \end{multline*}
By repeating the proofs of \eqref{q1} and \eqref{q2}, and using the identity $\partial^\alpha = \partial^{\tilde{\alpha}} \partial_1$, we obtain
    \begin{align*}
R_1&\leq C \sum_{m=0}^{+\infty}L_{\rho,m}^2  \norm{\partial_3^{m}\partial^{\alpha}u_2}_{L^2}^{\frac12}\Big(\norm{\partial_3^{m}\partial^{\alpha}u_2}_{L^2}^{\frac12} +\norm{\partial_3^m\partial_1\partial^{\alpha}u_2}_{L^2}^{\frac12} \Big)\norm{\partial_3^m\partial^{\tilde{\alpha}}\nabla p}_{L^2}\\
&\leq C \sum_{m=0}^{+\infty}L_{\rho,m}^2  \norm{\partial_3^{m}\partial^{\tilde\alpha}\nabla_{\rm h}u_2}_{L^2}^{\frac12}\Big(\norm{\partial_3^{m}\partial^{\alpha}u_2}_{L^2}^{\frac12} +\norm{\partial_3^m\partial_1\partial^{\alpha}u_2}_{L^2}^{\frac12} \Big)\norm{\partial_3^m\partial^{\tilde{\alpha}}\nabla p}_{L^2}\\
&\leq Cr^{-\frac{\alpha_1-1}{4}-\frac{\alpha_1}{4}}\abs{\nabla_{\rm h} u}_{X_{\rho}} \Big( \sum_{m=0}^{+\infty}L_{\rho,m}^2  \norm{\partial_3^m\partial^{\tilde{\alpha}}\nabla p}_{L^2}^2\Big)^{\frac12},
\end{align*}
and
\begin{align*}
R_{2}
&\leq C \sum_{m=0}^{+\infty}L_{\rho,m}^2  \norm{\partial_3^{m}\partial^{\alpha}\partial_2u_2}_{L^2}^{\frac12}\Big(\norm{\partial_3^{m}\partial^{\alpha}u_2}_{L^2}^{\frac12} +\norm{\partial_3^m\partial_1\partial^{\alpha}u_2}_{L^2}^{\frac12} \Big)\\
&\qquad\qquad \times\norm{\partial_3^m \partial^{\tilde{\alpha}}\partial_2p}_{L^2}^\frac{1}{2}\norm{\partial_3^m\partial^{\tilde{\alpha}}\partial_1p}_{L^2}^\frac{1}{2}\\
&\leq C \sum_{m=0}^{+\infty}L_{\rho,m}^2  \norm{\partial_3^m\partial^{\tilde\alpha}\partial_2 \nabla_{\rm h}u}_{L^2}^{\frac12}\Big(\norm{\partial_3^{m}\partial^{\alpha}u}_{L^2}^{\frac12} +\norm{\partial_3^m\partial_1\partial^{\alpha}u}_{L^2}^{\frac12} \Big) \norm{\partial_3^m\partial^{\tilde{\alpha}}\nabla_{\rm h}p}_{L^2} \\
&\leq Cr^{-\frac{\alpha_1-1}{4}-\frac{\alpha_1}{4}} \abs{\nabla_{\rm h} u}_{X_{\rho}}\Big( \sum_{m=0}^{+\infty}L_{\rho,m}^2  \norm{\partial_3^m\partial^{\tilde{\alpha}}\nabla p}_{L^2}^2\Big)^{\frac12},
\end{align*}
where the factor $r^{-\frac{\alpha_1-1}{4}-\frac{\alpha_1}{4}}$ arises from the definition of $\abs{\cdot}_{X_{\rho}}$ (see Definition \ref{def:normx3}). Substituting the two estimates above into \eqref{r1r2} yields that, for any $\alpha\in\mathbb Z_+^3$ with $\abs\alpha\leq 3$ and  $ \alpha_1\ge 1,$
\begin{multline*}%\label{est:tildes31}
\sup_{x_1\in[0,1]} \sum_{m=0}^{+\infty}L_{\rho,m}^2\left|\int_{\mathbb{R}^2}(\partial_3^{m}\partial^{\alpha}\partial_2u_2)(\partial_3^m\partial^{\tilde  \alpha} p) dx_2dx_3\right|\\
 \leq C r^{-\frac{2\alpha_1-1}{4}} \abs{\nabla_{\rm h}u}_{X_{\rho}}\Big( \sum_{m=0}^{+\infty}L_{\rho,m}^2  \norm{\partial_3^m\partial^{\tilde{\alpha}}\nabla p}_{L^2}^2\Big)^{\frac12}.
    \end{multline*}
Similarly,  for any $\alpha\in\mathbb Z_+^3$ satisfying $\abs\alpha\leq 3$ and $ \alpha_1\ge 1$,
\begin{align*}
&\sup_{x_1\in[0,1]} \sum_{m=0}^{+\infty}L_{\rho,m}^2\left|\int_{\mathbb{R}^2}(\partial_3^{m}\partial^{\alpha}\partial_3u_3)(\partial_3^m\partial^{\tilde  \alpha} p)dx_2dx_3\right|\\
&\qquad\qquad+ \sup_{x_1\in[0,1]} \sum_{m=0}^{+\infty}L_{\rho,m}^2\left|\int_{\mathbb{R}^2}(\partial_3^m \partial^{\tilde{\alpha}}\partial_2^2u_1)(\partial_3^m\partial^{\tilde  \alpha} p) dx_2dx_3\right|\\
&\leq Cr^{-\frac{2\alpha_1-1}{4}}\abs{\nabla_{\rm h}u}_{X_{\rho}}\Big( \sum_{m=0}^{+\infty}L_{\rho,m}^2  \norm{\partial_3^m\partial^{\tilde{\alpha}}\nabla p}_{L^2}^2\Big)^{\frac12}.
\end{align*}
Substituting the two estimates above  into \eqref{+tildes3} yields \eqref{est:tildeS3}.

{\it Step 4.} Combining   estimates \eqref{est:tildeS1}, \eqref{df} and \eqref{est:tildeS3} with \eqref{p+}, we obtain  that, for any    $\tilde\alpha=\alpha-(1,0,0)$ with $\abs{ \alpha }\leq 3$ and $ \alpha_1\ge 1$,
\begin{multline*}
 \sum_{m= 0}^{+\infty}L^2_{\rho,m}\norm{\partial_3^m\partial^{ \tilde \alpha}  \nabla p}_{L^2}^2 \leq   Cr^{-\frac92}\abs{u}_{X_{\rho}}^2\Big( \sum_{m=0}^{+\infty}L_{\rho,m}^2  \norm{\partial_3^m\partial^{\tilde{\alpha}}\nabla p}_{L^2}^2\Big)^{\frac12}\\
 +Cr^{-\frac{2\alpha_1-1}{4}}\abs{\nabla_{\rm h}u}_{X_{\rho}}\Big( \sum_{m=0}^{+\infty}L_{\rho,m}^2  \norm{\partial_3^m\partial^{\tilde{\alpha}}\nabla p}_{L^2}^2\Big)^{\frac12}.
\end{multline*}
Then  the desired estimate \eqref{est:Pone} follows. This completes   the proof of Lemma \ref{lem:S3}.
\end{proof}

\begin{lemma}[Estimate on $S_4$]\label{lem:S4}
Let $S_4$ be given in \eqref{S1-S4}, that is,
\begin{equation*}
S_4= \sum_{m= 0}^{+\infty}r^{\alpha_1}L^2_{\rho,m}\int_{\mathbb R^2} \big[ (\partial_3^{m}\partial^{\alpha}  u) \partial_3^m\partial ^{\tilde\alpha} \big(\partial_tu+u\cdot\nabla u-\partial_2^2u\big)\big]\big|^{x_1=1}_{x_1=0}\ dx_2dx_3,
\end{equation*}
where $\tilde\alpha=\alpha-(1,0,0)=(\alpha_1-1,\alpha_2,\alpha_3)$.
Then for any $\alpha\in\mathbb Z_+^3$ with $\abs\alpha \leq 3 $ and $\alpha_1\geq 1,$ it holds that
\begin{equation*}%\label{est:S4}
\begin{aligned}
S_4&\leq Cr^{-\frac{9}{2}}\abs{u}_{X_{\rho}}^2\abs{\nabla_{\rm h}u}_{X_{\rho}}+Cr^{\frac{1}{4}}\abs{\nabla_{\rm h}u}_{X_{\rho}}^2\\
&\leq Cr^{-\frac92}\abs{u}_{X_{\rho}} \abs{u}_{Y_{\rho}}^2+C\Big( r^{-\frac92}\abs{u}_{X_{\rho}} +r^{\frac{1}{4}}\Big)\abs{\nabla_{\rm h}u}_{X_{\rho}}^2,
\end{aligned}
\end{equation*}
 recalling the norms $\abs{\cdot}_{X_{\rho}}$ and $\abs{\cdot}_{Y_{\rho}}$ are given in Definition \ref{def:normx3}.
\end{lemma}

\begin{proof}
Note that
  $\tilde \alpha =\alpha-(1,0,0)$  and
\begin{equation*}
	    \partial_3^m\partial^{\tilde{\alpha}} \partial_tu |_{x_1=0,1}=\partial_3^m\partial^{\tilde{\alpha}}  \partial_2^2 u |_{x_1=0,1} = 0 \textrm{ for }\ \tilde\alpha_1=0.
\end{equation*}
Then we split $S_4$ as
\begin{equation}\label{+S4}
     S_4\leq  2 r^{ \alpha_1}\big(S_{4,1}+S_{4,2}+S_{4,3}\big),
\end{equation}
where
 \begin{equation*}
     \left\{
     \begin{aligned}
&S_{4,1}={\bf 1}_{\{\alpha_1\ge 2\}}\sup_{x_1\in[0,1]}\sum_{m=0}^{+\infty}L_{\rho,m}^2\left|\int_{\mathbb R^2}   (\partial_3^{m}\partial^{\alpha}  u) \partial_3^m\partial ^{\tilde\alpha}  \partial_tu \,  dx_2dx_3\right|,\\
&S_{4,2}={\bf 1}_{\{ \alpha_1\ge 2\}}\sup_{x_1\in[0,1]}\sum_{m=0}^{+\infty}L_{\rho,m}^2\left|\int_{\mathbb R^2}   (\partial_3^{m}\partial^{\alpha}  u) \partial_3^m\partial ^{\tilde\alpha}\partial_2^2u\,  dx_2dx_3\right|,\\
&S_{4,3}= \sum_{m=0}^{+\infty}L_{\rho,m}^2\left|\int_{\mathbb R^2} \big[  (\partial_3^{m}\partial^{\alpha}  u) \partial_3^m\partial ^{\tilde\alpha} \big( (u\cdot\nabla) u\big ) \big]\big|_{x_1=0,1} \, dx_2dx_3\right|.
     \end{aligned}
     \right.
 \end{equation*}
Here and below ${\bf 1}_{A}$ stands for the indicator function of set $A.$

For the term $S_{4,1}$, it  follows from  the Gagliardo-Nirenberg inequality \eqref{GN} that, observing $\partial^{\tilde\alpha}\partial_1=\partial^{\alpha},$
\begin{align*}
&  S_{4,1}
  \leq C {\bf 1}_{\{\alpha_1\ge 2\}} \sum_{m=0}^{+\infty}L_{\rho,m}^2 \norm{\partial_3^m \partial^{\tilde{\alpha}}\partial_tu}_{L^2}^\frac{1}{2}\Big (\norm{\partial_3^m \partial^{\tilde{\alpha}}\partial_tu}_{L^2}^\frac{1}{2}+\norm{\partial_3^{m}\partial^{\alpha} \partial_tu}_{L^2}^\frac{1}{2}\Big )\\
&\qquad\qquad\qquad\times\norm{\partial_3^{m}\partial^{\alpha}u}_{L^2}^\frac{1}{2}\Big(\norm{\partial_3^{m}\partial^{\alpha}u}_{L^2}^\frac{1}{2}+\norm{\partial_3^{m}\partial^{\alpha}\partial_1u}_{L^2}^\frac{1}{2}\Big).
\end{align*}
Moreover, for $\alpha_1\geq 2$ we have $\tilde\alpha_1=\alpha_1-1\geq 1$ and thus
\begin{align*}
	&\sum_{m=0}^{+\infty}L_{\rho,m}^2 \norm{\partial_3^m \partial^{\tilde{\alpha}}\partial_tu}_{L^2}^\frac{1}{2}\Big (\norm{\partial_3^m \partial^{\tilde{\alpha}}\partial_tu}_{L^2}^\frac{1}{2}+\norm{\partial_3^{m}\partial^{\alpha} \partial_tu}_{L^2}^\frac{1}{2}\Big )\\
&\qquad\qquad\qquad\times\norm{\partial_3^{m}\partial^{\alpha}u}_{L^2}^\frac{1}{2}\Big(\norm{\partial_3^{m}\partial^{\alpha}u}_{L^2}^\frac{1}{2}+\norm{\partial_3^{m}\partial^{\alpha}\partial_1u}_{L^2}^\frac{1}{2}\Big)\\
&= \sum_{m=0}^{+\infty}\Big(r^{\frac{\alpha_1-2}{4}}L_{\rho,m}^\frac{1}{2}\norm{\partial_3^m \partial^{\tilde{\alpha}}\partial_tu}_{L^2}^\frac{1}{2}\Big)  r^{\frac{\alpha_1-1}{4}}L_{\rho,m}^\frac{1}{2} \Big (\norm{\partial_3^m \partial^{\tilde{\alpha}}\partial_tu}_{L^2}^\frac{1}{2}+\norm{\partial_3^{m}\partial^{\alpha} \partial_tu}_{L^2}^\frac{1}{2}\Big )\\
&\qquad\quad \times \Big( r^{\frac{\alpha_1-1}{4}}L_{\rho,m}^\frac{1}{2}\norm{\partial_3^{m}\partial^{\alpha}u}_{L^2}^\frac{1}{2}\Big)   r^{\frac{\alpha_1}{4}}L_{\rho,m}^\frac{1}{2}\Big(\norm{\partial_3^{m}\partial^{\alpha}u}_{L^2}^\frac{1}{2}+\norm{\partial_3^{m}\partial^{\alpha}\partial_1u}_{L^2}^\frac{1}{2}\Big)  \frac{1}{r^{\alpha_1-1}} \\
&\leq C r^{-\alpha_1+1} \abs{\nabla_{\rm h} u}_{X_{\rho}}^2.
\end{align*}
Combining the above estimates yields
\begin{equation}\label{est:S41}
 \begin{aligned}
     S_{4,1} \leq  C{\bf 1}_{ \{\alpha_1 \ge 2\}}r^{-\alpha_1+1} \abs{\nabla_{\rm h}u}_{X_{\rho}}^2.
 \end{aligned}
\end{equation}
The estimate for  $S_{4,2}$  follows from a similar argument  to that  in \eqref{est:tildeS3}. By repeating the proof of \eqref{est:tildeS3} and using the identity $\partial^{\alpha}=\partial^{\tilde\alpha}\partial_1$ for $\alpha_1\geq 1,$  we have
\begin{equation}\label{est:S42}
    \begin{aligned}
S_{4,2}
&\leq C{\bf 1}_{\{\alpha_1\ge 2\}}r^{-\frac{2 \alpha_1-1}{4}}\abs{\nabla_{\rm h}u}_{X_{\rho}}\Big(\sum_{m=0}^{+\infty}L_{\rho,m}^2 \norm{\partial_3^{m}\partial^{\alpha}  \nabla u}_{L^2}^2\Big)^{\frac12} \\
&\leq C{\bf 1}_{\{\alpha_1\ge 2\}}r^{-\frac{2 \alpha_1-1}{4}}\abs{\nabla_{\rm h}u}_{X_{\rho}}  r^{-\frac{\alpha_1}{2}}\Big(\sum_{m=0}^{+\infty} r^{\alpha_1}L_{\rho,m}^2 \norm{\partial_3^m \partial^{ \tilde \alpha}  \nabla \partial_1u}_{L^2}^2\Big)^{\frac12} \\
&\leq C{\bf 1}_{\{\alpha_1\ge 2\}}r^{- \alpha_1+\frac{1}{4}}\abs{\nabla_{\rm h}u}_{X_{\rho}}^2.
    \end{aligned}
\end{equation}
It remains to estimate the term $S_{4,3}$.  The boundary condition $u|_{x_1=0,1}=0$ yields
\begin{align*}
	\partial_3^m \partial^{\tilde{\alpha}}\big ((u\cdot\nabla) u\big)|_{x_1=0,1}
	=\sum_{   \gamma\leq \tilde\alpha,\, \abs\gamma>0}\,\sum_{k=0}^m\binom{\tilde\alpha}{\gamma}\binom{m}{k} \big[ (\partial^{ \gamma}\partial_3^k u\cdot\nabla) \partial^{\tilde\alpha-\gamma} \partial_3^{m-k}u \big]\big|_{x_1=0,1}.
\end{align*}
This with the fact that $|\tilde\alpha|=\abs\alpha-1\leq 2$ implies
\begin{align*}
\big\|\partial_3^m \partial^{\tilde{\alpha}}((u\cdot\nabla) u) |_{x_1=0,1}\big\|_{L_{x_2, x_3}^2}\leq C\sum_{k=0}^m\binom{m}{k}\norm{\partial_3^ku}_{H^3}\norm{\partial_3^{m-k}u}_{H^3}.
\end{align*}
As a result,
\begin{equation}\label{est:S43}
    \begin{aligned}
  S_{4,3}= &\sum_{m=0}^{+\infty}L_{\rho,m}^2\left|\int_{\mathbb R^2} \big[  (\partial_3^{m}\partial^{\alpha}  u) \partial_3^m\partial ^{\tilde\alpha} \big( (u\cdot\nabla) u\big ) \big]\big|_{x_1=0,1} \, dx_2dx_3\right|\\
 \leq & C \sum_{m=0}^{+\infty}L_{\rho,m}^2\big\|\partial_3^m \partial^{\tilde{\alpha}}((u\cdot\nabla) u) |_{x_1=0,1}\big\|_{L_{x_2, x_3}^2}\big\|\partial_3^{m}\partial^{\alpha}u |_{x_1=0,1}\big\|_{L_{x_2, x_3}^2}\\\leq&C \sum_{m=0}^{+\infty}\sum_{k=0}^m\binom{m}{k}L_{\rho,m}^2\norm{\partial_3^ku}_{H^3}\norm{\partial_3^{m-k}u}_{H^3} \norm{\partial_3^m\nabla_{\rm h}u}_{H^3}\\
 \leq &Cr^{-\frac{9}{2}}\abs{u}_{X_{\rho}}^2\abs{\nabla_{\rm h} u}_{X_{\rho}},
    \end{aligned}
\end{equation}
where the last inequality  follows from a similar argument to that used in the proof of Lemma \ref{lem:S1}.
Substituting  estimates \eqref{est:S41}, \eqref{est:S42} and \eqref{est:S43} into \eqref{+S4} yields
\begin{align*}
    S_4\leq &C r^{\alpha_1}\Big(r^{-\alpha_1+\frac{1}{4}}\abs{\nabla_{\rm h}u}_{X_{\rho}}^2+r^{-\frac{9}{2}}\abs{u}_{X_{\rho}}^2\abs{\nabla_{\rm h} u}_{X_{\rho}}\Big)\\
 \leq &Cr^{-\frac{9}{2}}\abs{u}_{X_{\rho}}^2\abs{\nabla_{\rm h}u}_{X_{\rho}}+Cr^{\frac{1}{4}}\abs{\nabla_{\rm h}u}_{X_{\rho}}^2\\
 \leq & Cr^{-\frac92}\abs{u}_{X_{\rho}} \abs{u}_{Y_{\rho}}^2+C\Big( r^{-\frac92}\abs{u}_{X_{\rho}} +r^{\frac{1}{4}}\Big)\abs{\nabla_{\rm h}u}_{X_{\rho}}^2,
\end{align*}
 the last inequality using the same argument as that in \eqref{teces}.  The proof of Lemma \ref{lem:S4} is thus completed.
\end{proof}

\begin{proof}[Completing the proof of Proposition \ref{purespatial}]
Substituting  the estimates for terms $S_j$ from Lemmas \ref{lem:S1}-\ref{lem:S4} into \eqref{s}, we obtain that, for any $\alpha=(\alpha_1,\alpha_2,\alpha_3)\in\mathbb Z_+^3$ with $\alpha_1\geq1,$ 	
\begin{equation*}
	\begin{aligned}
		&\frac{1}{2}\frac{d}{dt}\sum_{m=0}^{+\infty} r^{\alpha_1}L_{\rho,m}^2\norm{\partial_3^{m}\partial^{\alpha} u}_{L^2}^2+\sum_{m=0}^{+\infty} r^{\alpha_1}L_{\rho,m}^2 \norm{\partial_3^{m}\partial^{\alpha} \nabla_{\rm h} u}_{L^2}^2\\
		       &\qquad- \frac{\rho'}{\rho}\sum_{m=0}^{+\infty} r^{\alpha_1}(m+1)L_{\rho,m}^2\norm{\partial_3^{m}\partial^{\alpha} u}_{L^2}^2\\
		       &\leq C r^{-\frac{9}{2}}\abs{u}_{X_{\rho}}\abs{u}_{Y_{\rho}}^2+ C\Big(r^{-\frac{9}{2}}\abs{u}_{X_{\rho}}  +r^{\frac{1}{4}} \Big)\abs{\nabla_{\rm h}u}_{X_{\rho}}^2.
	\end{aligned}
\end{equation*}
A similar, though simpler, argument shows that the above estimate remains valid for any $\alpha\in\mathbb Z_+^3$ with $\alpha_1=0$ (see  Remark \ref{vanishing}). This yields  the desired assertion of Proposition \ref{purespatial},  completing the proof.
\end{proof}

\subsection{Estimate on space-time derivatives and completing the proof of Theorem \ref{prop:wellposedness}}\label{subspcetime}

To complete the proof of Theorem \ref{prop:wellposedness}, it remains to estimate the terms involving time derivatives in the norm $\abs{u}_{X_\rho}$. We state it  as follows.

\begin{proposition}\label{mixder}
	Under the hypothesis of Theorem \ref{prop:wellposedness}, for any $(j, \alpha)\in \mathbb Z_+\times\mathbb Z_+^3$ with $j+\abs\alpha\leq 3$ and $j\geq 1,$  we have that
	\begin{equation*}
		\begin{aligned}
		     &	\frac{1}{2}\frac{d}{dt}\sum_{m=0}^{+\infty} r^{\alpha_1}L_{\rho,m}^2\norm{\partial_3^{m}\partial_t^j\partial^{\alpha} u}_{L^2}^2+\sum_{m=0}^{+\infty} r^{\alpha_1}L_{\rho,m}^2 \norm{\partial_3^{m}\partial_t^j\partial^{\alpha} \nabla_{\rm h} u}_{L^2}^2\\
		     &\qquad- \frac{\rho'}{\rho}\sum_{m=0}^{+\infty} r^{\alpha_1}(m+1)L_{\rho,m}^2\norm{\partial_3^{m}\partial_t^j\partial^{\alpha} u}_{L^2}^2 \\
		     &\leq C r^{-\frac{9}{2}}\abs{u}_{X_{\rho}}\abs{u}_{Y_{\rho}}^2+  C\Big(r^{-\frac{9}{2}}\abs{u}_{X_{\rho}}  +r^{\frac{1}{4}} \Big)\abs{\nabla_{\rm h}u}_{X_{\rho}}^2,
		\end{aligned}
	\end{equation*}
	where the norms $\abs{\cdot}_{X_{\rho}}$ and $\abs{\cdot}_{Y_{\rho}}$ are given in Definition \ref{def:normx3}.
\end{proposition}

\begin{proof}
	 The proof is analogous  to that of Proposition \ref{purespatial},  and  there is  no any new difficulty with the help of Remarks \ref{algebra} and \ref{remk}.   So we omit it for brevity.
\end{proof}

\begin{proof}[Completing the proof of Theorem \ref{prop:wellposedness}] By the definitions of $\abs{\cdot}_{X_{\rho}}$ and $\abs{\cdot}_{Y_{\rho}}$  from  Definition \ref{def:normx3}, we combine Proposition \ref{purespatial} and  Proposition \ref{mixder} to obtain
\begin{multline}\label{ces}
 \frac{1}{2}\frac{d}{dt}\abs{u}_{X_{\rho}}^2+\abs{\nabla_{\rm h}u}_{X_{\rho}}^2-\frac{\rho^{'}}{\rho} \abs{u}_{Y_{\rho}}^2\\
 \leq C r^{-\frac{9}{2}}\abs{u}_{X_{\rho}}\abs{u}_{Y_{\rho}}^2+  C\Big(r^{-\frac{9}{2}}\abs{u}_{X_{\rho}}  +r^{\frac{1}{4}} \Big)\abs{\nabla_{\rm h}u}_{X_{\rho}}^2.
\end{multline}
On the other hand, we assert that
\begin{equation}\label{Poin}
     \abs{u}_{X_{\rho}}^2\leq \abs{\nabla_{\rm h}u}_{X_{\rho}}^2.
\end{equation}
To confirm  this, recall that
\begin{align*}
    \abs{u}_{X_{\rho}}^2= \sum_{\stackrel{(j,   \alpha) \in\mathbb Z_+\times \mathbb{Z}_+^3}{j+\abs{ \alpha }\leq 3}}\ \sum_{m= 0}^{+\infty}r^{ \alpha_1}L^2_{\rho,m}\norm{\partial_3^m\partial_t^j\partial^{ \alpha }u}^2_{L^2}.
\end{align*}
Then for any $(j,\alpha)\in\mathbb Z_+\times\mathbb Z_+^3$ with $j+\abs\alpha\leq 3$, the following holds:  if $\alpha_1 \geq 1$, then
\begin{multline*}
r^{\alpha_1} \norm{\partial_3^m\partial_t^j\partial^{\alpha }u}_{L^2}^2=r^{\alpha_1}\norm{\partial_3^m\partial_t^j\partial_1^{\alpha_1-1}\partial_2^{\alpha_2}\partial_3^{\alpha_3}(\partial_1u)}_{L^2}^2\\
\leq r^{\alpha_1-1}\norm{\partial_3^m\partial_t^j\partial_1^{\alpha_1-1}\partial_2^{\alpha_2}\partial_3^{\alpha_3} \nabla_{\rm h}u }_{L^2}^2;
\end{multline*}
 if $\alpha_1=0$ then by the Poincar\'e inequality and the boundary condition  $u|_{x_1 = 0,1} = 0$, we have
\begin{align*}
r^{\alpha_1}\norm{\partial_3^m\partial_t^j\partial^{\alpha }u}_{L^2}^2\leq   r^{\alpha_1}\norm{\partial_3^m\partial_t^j\partial^{\alpha } \partial_1u }_{L^2}^2\leq   r^{\alpha_1}\norm{\partial_3^m\partial_t^j\partial^{\alpha }\nabla_{\rm h}u}_{L^2}^2.
\end{align*}
   Combining the above estimates yields inequality \eqref{Poin}.

It follows from \eqref{ces} and \eqref{Poin}  that
\begin{multline*}
	 \frac{1}{2}\frac{d}{dt}\abs{u}_{X_{\rho}}^2+\frac{1}{4} \abs{u}_{X_{\rho}}^2
	 +\frac 34 \abs{\nabla_{\rm h}u}_{X_{\rho}}^2\\
 \leq \Big(\frac{\rho^{'}}{\rho}+C r^{-\frac{9}{2}}\abs{u}_{X_{\rho}}\Big)\abs{u}_{Y_{\rho}}^2+  C\Big(r^{-\frac{9}{2}}\abs{u}_{X_{\rho}}  +r^{\frac{1}{4}} \Big)\abs{\nabla_{\rm h}u}_{X_{\rho}}^2,
\end{multline*}
and therefore
\begin{equation}\label{DD}
\begin{aligned}
	& \frac{1}{2}\frac{d}{dt} e^{\frac{t}{2}}\abs{u}_{X_{\rho}}^2
	 +\frac34 e^{\frac{t}{2}} \abs{\nabla_{\rm h}u}_{X_{\rho}}^2\\
&\quad\leq \Big(\frac{\rho^{'}}{\rho}+C r^{-\frac{9}{2}}\abs{u}_{X_{\rho}}\Big)e^{\frac{t}{2}}\abs{u}_{Y_{\rho}}^2+   C\Big(r^{-\frac{9}{2}}\abs{u}_{X_{\rho}}  +r^{\frac{1}{4}} \Big)e^{\frac{t}{2}}  \abs{\nabla_{\rm h}u}_{X_{\rho}}^2.
 \end{aligned}
\end{equation}
This allows us to apply a standard bootstrap argument to prove assertion \eqref{pri:ret1} in Theorem \ref{prop:wellposedness}. Assume that
\begin{equation}\label{pri:ass1}
\forall\ t\ge 0,\quad e^{\frac{t}{2}}\abs{u(t)}_{X_{\rho}}^2+\int^t_0e^{\frac{s}{2}}\abs{\nabla_{\rm h}u(s)}_{X_{\rho}}^2ds\leq 4\varepsilon_0^2.
\end{equation}
Then, combining \eqref{DD} with \eqref{pri:ass1} and noting from  \eqref{def:rerho}  and \eqref{boundrho} that
\begin{equation*}
-\frac14 e^{-\frac{t}{4}}\leq 	\frac{\rho'}{\rho}=-\frac{  \rho_0}{8\rho} e^{-\frac{t}{4}} \leq -\frac{ 1}{8}e^{-\frac{t}{4}},
\end{equation*}
we conclude that, for any $t\geq 0,$
\begin{multline*}%\label{DDDD}
\frac{1}{2}\frac{d}{dt} e^{\frac{t}{2}}\abs{u}_{X_{\rho}}^2
	 +\frac 34 e^{\frac{t}{2}} \abs{\nabla_{\rm h}u}_{X_{\rho}}^2\\
	 \leq  \Big(-\frac18 e^{-\frac{t}{4}} +  2 Cr^{-\frac{9}{2}} \eps_0 e^{-\frac{t}{4}}  \Big)e^{\frac{t}{2}}\abs{u}_{Y_{\rho}}^2
	 +\Big(2Cr^{-\frac{9}{2}} \eps_0  +Cr^{\frac{1}{4}} \Big)e^{\frac{t}{2}} \abs{\nabla_{\rm h}u}_{X_{\rho}}^2.
\end{multline*}
We first choose $r > 0$ sufficiently small such that
\begin{equation*}
	Cr^{\frac14} \leq \frac 18,
\end{equation*}
and then  choose $\eps_0$ small enough so that
\begin{equation*}
2 Cr^{-\frac{9}{2}} \eps_0 \leq \frac18.
\end{equation*}
As a consequence, the above inequalities imply
\begin{equation}\label{3direct}
 \forall\ t\ge 0,\quad \frac12\frac{d}{dt} e^{\frac{t}{2}} \abs{u}_{X_{\rho}}^2+\frac{1}{2}e^{\frac{t}{2}} \abs{\nabla_{\rm h}u}_{X_{\rho}}^2\leq 0.
\end{equation}
On the other hand,  it follows from \eqref{inda}  that
\begin{align*}
\lim_{t\rightarrow 0}e^{\frac{t}{2}}\abs{u(t)}_{X_{\rho}}^2=\lim_{t\rightarrow 0}\abs{u(t)}_{X_{\rho}}^2
\leq C_*^2\Big(\norm{u_0}_{G_{\rho_0,\sigma,6}}+\norm{u_0}_{G_{\rho_0,\sigma,6}}^4\Big)^2
 \leq \varepsilon_0^2,
\end{align*}
where $C_*$ is the constant in \eqref{inda} and  the last inequality uses the smallness assumption \eqref{pri:initial1}.
Combining the above estimate with \eqref{3direct}, we obtain
\begin{equation*}
\forall\ t\ge 0,\quad e^{\frac{t}{2}}\abs{u(t)}_{X_{\rho}}^2+\int^t_0e^{\frac{s}{2}}\abs{\nabla_{\rm h}u(s)}_{X_{\rho}}^2ds
\leq\varepsilon_0^2.
\end{equation*}
Thus, assertion \eqref{pri:ret1} follows. This completes the proof of Theorem \ref{prop:wellposedness}.
\end{proof}

 \section{Proof of Theorem \ref{thm:smoothing}: smoothing effect for the time variable}\label{sec:time}

This section establishes the Gevrey smoothing effect in time. We will adapt the previous section's argument and emphasize  the differences. To do so, we first introduce two auxiliary norms corresponding to those in Definition \ref{def:normx3}.

\begin{definition}\label{def:normt}
Let $\rho$ be defined as in \eqref{def:rerho} and $r$ be the parameter specified  in  Theorem  \ref{prop:wellposedness},  and let  $0<\lambda<1$ be a given parameter to be determined later.   We define  the  space-time derivatives $H_\lambda^\beta$ by setting
\begin{equation}
	\label{hlambda}
	H_{\lambda}^\beta\stackrel{\rm def}{=}\lambda^{\beta_1}t^{\beta_1}\partial_t^{\beta_1}\partial_3^{\beta_2}\  \textrm{ for }\ \beta=(\beta_1,\beta_2)\in \mathbb Z_+^2,
\end{equation}
and introduce two norms $\abs{\cdot}_{\tilde{X}_{\rho,\lambda}}$ and $\abs{\cdot}_{\tilde{Y}_{\rho,\lambda}}$ as follows:
\begin{equation*}
    \left\{
    \begin{aligned}
&\abs{g}_{\tilde{X}_{\rho,\lambda}}^2\stackrel{\rm def}{=}\sum_{\stackrel{(j,\alpha)\in\mathbb Z_+\times \mathbb{Z}_+^3}{j+\abs{\alpha}\leq 3}}\ \sum_{\beta\in\mathbb Z_+^2}r^{\alpha_1}L_{\rho,\abs\beta}^2\norm{H_{\lambda}^{\beta}\partial_t^j\partial^{\alpha}g}_{L^2}^2,\\
&\abs{g}_{\tilde{Y}_{\rho,\lambda}}^2\stackrel{\rm def}{=}\sum_{\stackrel{(j,\alpha)\in\mathbb Z_+\times \mathbb{Z}_+^3}{j+\abs{\alpha}\leq 3}} \ \sum_{\beta\in\mathbb Z_+^2} r^{\alpha_1}(\abs\beta+1)L_{\rho,\abs\beta}^2\norm{H_{\lambda}^{\beta}\partial_t^j\partial^{\alpha}g}_{L^2}^2,
    \end{aligned}
    \right.
\end{equation*}
where $L_{\rho,m}$ is defined in  \eqref{def:L}.
\end{definition}

\begin{remark}\label{rek:xyt}
 Recalling $\abs{\cdot}_{X_{\rho}}$ is the norm from   Definition \ref{def:normx3}, we have
 \begin{align*}
 \abs{g}_{\tilde{X}_{\rho,\lambda}}^2
=\abs{g}_{X_{\rho}}^2+\sum_{\stackrel{(j,\alpha)\in\mathbb Z_+\times \mathbb{Z}_+^3}{j+\abs{\alpha}\leq 3}}\ \sum_{\stackrel{\beta\in\mathbb Z_+^2}{\beta_1\geq 1}}r^{\alpha_1}L_{\rho,\abs\beta}^2\norm{H_{\lambda}^{\beta}\partial_t^j\partial^{\alpha}g}_{L^2}^2,
 \end{align*}
and similarly for $\abs{g}_{\tilde{Y}_{\rho,\lambda}}^2$. Moreover, it holds that
\begin{equation}\label{xyt}
 \abs{g}_{\tilde{X}_{\rho,\lambda}}\leq \abs{g}_{\tilde{Y}_{\rho,\lambda}}.
\end{equation}
\end{remark}

Using the  norms defined above,  we state the main result on the Gevrey smoothing effect in the time variable as follows.

\begin{proposition}\label{prop:t}
Let $u\in L^{\infty}([0,+\infty[, G_{\rho,\sigma,3})$ be the solution to system \eqref{ANS} constructed in Theorem \ref{thm:main1}.  Then there exists a small constant $0<\lambda<1$ such that,  shrinking  the number $r$ in Definition \ref{def:normx3} and the number $\eps_0$ in  Theorem \ref{prop:wellposedness} if necessary,
\begin{equation}\label{pri:ret3}
\forall\ t\ge 0,\quad e^{\frac{t}{2}}\abs{u(t)}_{\tilde{X}_{\rho,\lambda}}^2+\int^t_0e^{\frac{s}{2}}\abs{\nabla_{\rm h}u(s)}_{\tilde{X}_{\rho,\lambda}}^2ds\leq \varepsilon_0^2,
\end{equation}
recalling $\abs{\cdot}_{\tilde{X}_{\rho,\lambda}}$ is given in Definition \ref{def:normt}.
\end{proposition}

To prove this proposition, we  begin with   two preliminary lemmas.

\begin{lemma}
	\label{vecx}
	Under the hypothesis of Proposition \ref{prop:t}, we have that
	\begin{multline}\label{des}
\frac12 \frac{d}{dt}\abs{u}_{\tilde{X}_{\rho,\lambda}}^2+ \abs{\nabla_{\rm h} u}_{\tilde{X}_{\rho,\lambda}}^2-\frac{\rho'}{\rho}\abs{\nabla_{\rm h} u}_{\tilde{Y}_{\rho,\lambda}}^2  \\
\leq Cr^{-\frac{9}{2}}\abs{u}_{\tilde{X}_{\rho,\lambda}}\abs{u}_{\tilde{Y}_{\rho,\lambda}}^2
    +C\Big(r^{-\frac{9}{2}}\abs{u}_{\tilde{X}_{\rho,\lambda}} +r^{\frac{1}{4}}+\lambda\Big)\abs{\nabla_{\rm h}u}_{\tilde{X}_{\rho,\lambda}}^2,
	\end{multline}
    recalling the norms $\abs{\cdot}_{\tilde{X}_{\rho,\lambda}}$ and $\abs{\cdot}_{\tilde{Y}_{\rho,\lambda}}$ are from Definition \ref{def:normt}.
\end{lemma}

\begin{proof}
The proof of Lemma \ref{vecx} is quite analogous  to those of Propositions \ref{purespatial} and   \ref{mixder}. Similar to the derivation of \eqref{s}, we   have that, for any $\alpha\in\mathbb Z_+^3$ with $\abs\alpha\leq 3,$
\begin{multline}\label{est:t}
 \frac{1}{2}\frac{d}{dt} \sum_{\stackrel{\beta \in \mathbb Z_+^2}{\beta_1\ge 1}} r^{\alpha_1}L_{\rho,\abs{\beta}}^2\norm{H_\lambda^\beta \partial^{\alpha}  u}_{L^2}^2+ \sum_{\stackrel{\beta \in \mathbb Z_+^2}{\beta_1\ge 1}}r^{\alpha_1}L_{\rho,\abs\beta}^2\norm{ H_\lambda^\beta \partial^{ \alpha} \nabla_{\rm h} u}_{L^2}^2\\
 -\frac{\rho^{'}}{\rho}  \sum_{\stackrel{\beta \in \mathbb Z_+^2}{\beta_1\ge 1}}r^{\alpha_1}(\abs{\beta }+1)L_{\rho,\abs\beta}^2\norm{H_\lambda^\beta \partial^{\alpha} u}_{L^2}^2\leq \sum_{\ell=1}^5K_\ell,
 \end{multline}
where
\begin{equation*}%\label{K1-K5}
	  \left\{
    \begin{aligned}
&K_1= - \sum_{\stackrel{\beta \in \mathbb Z_+^2}{\beta_1\ge 1}}r^{\alpha_1}L_{\rho,\abs\beta}^2  \inner{H_\lambda^\beta\partial^{\alpha}\big ((u_{\rm h}\cdot\nabla_{\rm h}) u\big ),\   H_\lambda^\beta\partial^{\alpha} u}_{L^2},\\
&K_2=- \sum_{\stackrel{\beta \in \mathbb Z_+^2}{\beta_1\ge 1}}r^{\alpha_1}L_{\rho,\abs\beta}^2\inner{H_\lambda^\beta\partial^{\alpha} (u_3\partial_3u),\   H_\lambda^\beta\partial^{\alpha} u}_{L^2},\\
&K_3= \sum_{\stackrel{\beta \in \mathbb Z_+^2}{\beta_1\ge 1}}r^{\alpha_1}L_{\rho,\abs\beta}^2\inner{H_\lambda^\beta\partial^{\tilde \alpha}\nabla p, \   H_\lambda^\beta\partial^{\alpha}\partial_1 u}_{L^2},\\
&K_4= \sum_{\stackrel{\beta \in \mathbb Z_+^2}{\beta_1\ge 1}}r^{\alpha_1}L_{\rho,\abs\beta}^2\int_{\mathbb R^2} \big[ (H_\lambda^\beta\partial^{\alpha} u)  H_\lambda^\beta\partial^{\tilde \alpha} \big(\partial_tu+u\cdot\nabla u-\partial_2^2u\big)\big]\big|^{x_1=1}_{x_1=0}\ dx_2dx_3,\\
&K_5= \sum_{\stackrel{\beta \in \mathbb Z_+^2}{\beta_1\ge 1}} r^{\alpha_1} \frac{ \beta_1}{t} L_{\rho,\abs\beta}^2\norm{H_\lambda^\beta\partial^{\alpha} u}_{L^2}^2,
    \end{aligned}
    \right.
\end{equation*}
recalling  $\tilde\alpha=\alpha-(1,0,0)=(\alpha_1-1,\alpha_2,\alpha_3).$ Moreover, observe
\begin{equation}\label{factorial}
	\binom{\beta}{\gamma}\leq \binom{\abs\beta}{\abs\gamma}\ \textrm{for}\ 0\leq\gamma\leq\beta.
\end{equation}
This enables us to
repeat the proofs of Lemmas \ref{lem:S1}-\ref{lem:S4} to conclude that
\begin{equation}\label{est:K14}
   \sum_{\ell=1}^4 K_\ell \leq Cr^{-\frac{9}{2}}\abs{u}_{\tilde{X}_{\rho,\lambda}}\abs{u}_{\tilde{Y}_{\rho,\lambda}}^2+Cr^{-\frac{9}{2}}\abs{u}_{\tilde{X}_{\rho,\lambda}}\abs{\nabla_{\rm h}u}_{\tilde{X}_{\rho,\lambda}}^2 +Cr^{\frac{1}{4}}\abs{\nabla_{\rm h}u}_{\tilde{X}_{\rho,\lambda}}^2.
\end{equation}
It remains to estimate $K_5$. Recall $\beta\in\mathbb Z_+^2$ with $\beta_1\geq 1$ and $\tilde\beta=\beta-(1,0)$. Then we use   \eqref{hlambda} to write
\begin{equation*}%\label{hlambdatilde}
	H_\lambda^\beta=\lambda t H_\lambda^{\tilde \beta}\partial_t.
\end{equation*}
 This yields
\begin{equation}\label{sk5}
	\begin{aligned}
		K_5=  \sum_{\stackrel{\beta \in \mathbb Z_+^2}{\beta_1\ge 1}} r^{\alpha_1} \frac{ \beta_1}{t} L_{\rho,\abs\beta}^2\norm{H_\lambda^\beta\partial^{\alpha} u}_{L^2}^2 =\lambda \sum_{\stackrel{\beta \in \mathbb Z_+^2}{\beta_1\ge 1}} r^{\alpha_1}  \beta_1 L_{\rho,\abs\beta}^2\inner{H_\lambda^{\tilde\beta}\partial^{\alpha} \partial_t u, H_\lambda^\beta\partial^{\alpha} u}_{L^2}.
	\end{aligned}
\end{equation}
We claim that,  for any multi-index  $\alpha\in\mathbb Z_+^3$ with  $\abs\alpha\leq 3,$
\begin{equation}\label{k5}
  K_5\leq  C  r^{-\frac{9}{2}}\abs{u}_{\tilde{X}_{\rho,\lambda}}\abs{u}_{\tilde{Y}_{\rho,\lambda}}^2+C  r^{-\frac{9}{2}}\abs{u}_{\tilde{X}_{\rho,\lambda}}\abs{\nabla_{\rm h} u}_{\tilde{X}_{\rho,\lambda}}^2+C\lambda  \abs{\nabla_{\rm h}u}_{\tilde{X}_{\rho,\lambda}}^2.
\end{equation}
We first prove \eqref{k5} for multi-indices $\alpha\in\mathbb Z_+^3$ satisfying  $\abs\alpha\leq 3$ and $\alpha_1\geq 1$. For such $\alpha$ we have   $\partial^\alpha=\partial^{\tilde\alpha}\partial_1$ with $\tilde\alpha=(\tilde{\alpha}_1,\tilde{\alpha}_2,\tilde{\alpha}_3)=\alpha-(1,0,0)$. Then using \eqref{sk5} gives
\begin{equation*}
	\begin{aligned}
		K_5&=\lambda \sum_{\stackrel{\beta \in \mathbb Z_+^2}{\beta_1\ge 1}} r^{\alpha_1}  \beta_1 L_{\rho,\abs\beta}^2\inner{H_\lambda^{\tilde\beta}\partial^{\tilde\alpha} \partial_t \partial_1 u, H_\lambda^\beta\partial^{\tilde\alpha}\partial_1 u}_{L^2}\\
		&\leq C\lambda \sum_{\stackrel{\beta \in \mathbb Z_+^2}{\beta_1\ge 1}}\Big( r^{\frac{\tilde\alpha_1}{2}}  L_{\rho,|\tilde\beta|}     \norm{H_\lambda^{\tilde\beta}\partial^{\tilde\alpha} \partial_t \partial_1 u}_{L^2}\Big)\Big( r^{\frac{\tilde \alpha_1}{2}}      L_{\rho,\abs\beta}\norm{H_\lambda^\beta\partial^{\tilde \alpha}\partial_1u}_{L^2}\Big)\\
		&\leq C\lambda\abs{\nabla_{\rm h}u}_{\tilde{X}_{\rho,\lambda}}^2,
	\end{aligned}
\end{equation*}
the second inequality using the fact that $\beta_1L_{\rho,|\beta|} \leq C L_{\rho, |\tilde\beta|}$ and the last line following from the definition of $\abs{\cdot}_{\tilde{X}_{\rho,\lambda}}$ (see Definition \ref{def:normt}) as well as the fact that $|\tilde\alpha|\leq 2.$  Thus \eqref{k5} holds for any multi-index  $\alpha\in\mathbb Z_+^3$ with  $\abs\alpha\leq 3$ and $\alpha_1\geq 1$.  On the other hand, for multi-indices $\alpha\in\mathbb Z_+^3$ satisfying  $\abs\alpha\leq 3$ and $\alpha_1=0$,    we use the fact that
\begin{equation*}
	\partial_tu=\Delta_{\rm h}u-(u\cdot \nabla) u-\nabla p
\end{equation*}
as well as \eqref{sk5} to write
\begin{multline*}
		K_5  = \lambda \sum_{\stackrel{\beta \in \mathbb Z_+^2}{\beta_1\ge 1}} r^{\alpha_1}  \beta_1 L_{\rho,\abs\beta}^2\inner{H_\lambda^{\tilde\beta}\partial^{\alpha} (\Delta_{\rm h} u-(u\cdot \nabla) u-\nabla p), H_\lambda^\beta\partial^{\alpha} u}_{L^2}\\
		 = -\lambda\sum_{\stackrel{\beta \in \mathbb Z_+^2}{\beta_1\ge 1}} r^{\alpha_1}  \beta_1 L_{\rho,\abs\beta}^2\inner{H_\lambda^{\tilde\beta}\partial^{\alpha} \nabla_{\rm h} u, H_\lambda^\beta\partial^{\alpha} \nabla_{\rm h}u}_{L^2}\\
		  -\lambda \sum_{\stackrel{\beta \in \mathbb Z_+^2}{\beta_1\ge 1}} r^{\alpha_1}  \beta_1 L_{\rho,\abs\beta}^2\inner{H_\lambda^{\tilde\beta}\partial^{\alpha}\big( (u\cdot \nabla) u\big), H_\lambda^\beta\partial^{\alpha} u}_{L^2},
	\end{multline*}
where	the last equality  follows from the integration by parts for $\alpha_1=0$ and the fact that $\divv u=0.$
Moreover, applying the inequality $\beta_1 L_{\rho,|\beta|} \leq C L_{\rho,|\tilde{\beta}|}$ for $\tilde{\beta} = \beta - (1,0)$ with $\beta_1 \geq 1$, we derive
\begin{equation*}
		-\lambda \sum_{\stackrel{\beta \in \mathbb Z_+^2}{\beta_1\ge 1}} r^{\alpha_1}  \beta_1 L_{\rho,\abs\beta}^2\inner{H_\lambda^{\tilde\beta}\partial^{\alpha} \nabla_{\rm h} u, H_\lambda^\beta\partial^{\alpha} \nabla_{\rm h}u}_{L^2}\leq C\lambda  \abs{\nabla_{\rm h}u}_{\tilde{X}_{\rho,\lambda}}^2,
	\end{equation*}
    and
	\begin{multline*}
		 -\lambda \sum_{\stackrel{\beta \in \mathbb Z_+^2}{\beta_1\ge 1}} r^{\alpha_1}  \beta_1 L_{\rho,\abs\beta}^2\inner{H_\lambda^{\tilde\beta}\partial^{\alpha}\big( (u\cdot \nabla) u\big), H_\lambda^\beta\partial^{\alpha} u}_{L^2}\\
		 \leq  C\lambda r^{-\frac{9}{2}}\abs{u}_{\tilde{X}_{\rho,\lambda}}\abs{u}_{\tilde{Y}_{\rho,\lambda}}^2+C\lambda r^{-\frac{9}{2}}\abs{u}_{\tilde{X}_{\rho,\lambda}}\abs{\nabla_{\rm h} u}_{\tilde{X}_{\rho,\lambda}}^2,
	\end{multline*}
where the second inequality is obtained by following the proofs of Lemmas \ref{lem:S1} and \ref{lem:S2} with minor modifications. As a result, combining the above estimates and recalling  the fact $0<\lambda< 1,$  we obtain the validity of \eqref{k5} for any multi-index  $\alpha\in\mathbb Z_+^3$ with  $\abs\alpha\leq 3$ and $\alpha_1=0.$  Consequently,  \eqref{k5} holds for all multi-indices $\alpha\in\mathbb Z_+^3$ with $|\alpha|\leq 3.$
	
Now, we substitute estimates \eqref{est:K14} and \eqref{k5} into \eqref{est:t} to get
\begin{equation*}
	\begin{aligned}
		& \frac{1}{2}\frac{d}{dt} \sum_{\stackrel{\beta \in \mathbb Z_+^2}{\beta_1\ge 1}} r^{\alpha_1}L_{\rho,\abs{\beta}}^2\norm{H_\lambda^\beta \partial^{\alpha}  u}_{L^2}^2+ \sum_{\stackrel{\beta \in \mathbb Z_+^2}{\beta_1\ge 1}}r^{\alpha_1}L_{\rho,\abs\beta}^2\norm{  H_\lambda^\beta \partial^{ \alpha} \nabla_{\rm h}u}_{L^2}^2\\
&\qquad  -\frac{\rho^{'}}{\rho}  \sum_{\stackrel{\beta \in \mathbb Z_+^2}{\beta_1\ge 1}}r^{\alpha_1}(\abs{\beta }+1)L_{\rho,\abs\beta}^2\norm{H_\lambda^\beta \partial^{\alpha} u}_{L^2}^2\\
&\leq C  r^{-\frac{9}{2}}\abs{u}_{\tilde{X}_{\rho,\lambda}}\abs{u}_{\tilde{Y}_{\rho,\lambda}}^2+C \Big( r^{-\frac{9}{2}}\abs{u}_{\tilde{X}_{\rho,\lambda}}+ r^{\frac14}+\lambda\Big)\abs{\nabla_{\rm h} u}_{\tilde{X}_{\rho,\lambda}}^2.
	\end{aligned}
\end{equation*}
Using \eqref{prpalge}, we extend the validity of the above estimate  to general indices $(j,\alpha)\in\mathbb{Z}_+\times\mathbb{Z}_+^3$ satisfying $j+|\alpha|\leq 3$:
\begin{equation*}
	\begin{aligned}
		& \frac{1}{2}\frac{d}{dt} \sum_{\stackrel{\beta \in \mathbb Z_+^2}{\beta_1\ge 1}} r^{\alpha_1}L_{\rho,\abs{\beta}}^2\norm{H_\lambda^\beta \partial_t^j \partial^{\alpha}  u}_{L^2}^2+ \sum_{\stackrel{\beta \in \mathbb Z_+^2}{\beta_1\ge 1}}r^{\alpha_1}L_{\rho,\abs\beta}^2\norm{ H_\lambda^\beta \partial_t^j  \partial^{ \alpha} \nabla_{\rm h} u}_{L^2}^2\\
&\qquad  -\frac{\rho^{'}}{\rho}  \sum_{\stackrel{\beta \in \mathbb Z_+^2}{\beta_1\ge 1}}r^{\alpha_1}(\abs{\beta }+1)L_{\rho,\abs\beta}^2\norm{H_\lambda^\beta \partial_t^j  \partial^{\alpha} u}_{L^2}^2\\
&\leq C  r^{-\frac{9}{2}}\abs{u}_{\tilde{X}_{\rho,\lambda}}\abs{u}_{\tilde{Y}_{\rho,\lambda}}^2+C \Big( r^{-\frac{9}{2}}\abs{u}_{\tilde{X}_{\rho,\lambda}}+ r^{\frac14}+\lambda\Big)\abs{\nabla_{\rm h} u}_{\tilde{X}_{\rho,\lambda}}^2.
	\end{aligned}
\end{equation*}
This  with estimate \eqref{ces} and Remark \ref{xyt} yields the desired estimate \eqref{des}. The proof of Lemma \ref{vecx} is completed.
\end{proof}

\begin{lemma}\label{lemshort}
Let the norms $\abs{\cdot}_{X_{\rho}},\abs{\cdot}_{\tilde{X}_{\rho,\lambda}}$ and $\abs{\cdot}_{\tilde{Y}_{\rho,\lambda}}$  be given in Definitions \ref{def:normx3} and \ref{def:normt}. If
\begin{equation}\label{sht+}
\sup_{t\leq 1} \abs{u(t)}_{\tilde{X}_{\rho,\lambda}}^2+\int_0^{1}   \big(\abs{u(t)}_{\tilde{Y}_{\rho,\lambda}}^2+\abs{\nabla_{\rm h} u(t)}_{\tilde{X}_{\rho,\lambda}}^2\big) dt <+\infty,
\end{equation}
then
\begin{equation*}
 \lim_{t\to 0} \abs{u}_{\tilde{X}_{\rho,\lambda}}^2= \lim_{t\to0} \abs{u}_{X_{\rho}}^2.		
\end{equation*}
\end{lemma}

\begin{proof}
In  what follows, let $(j, \alpha) \in \mathbb Z_+ \times \mathbb Z_+^3$ be a fixed multi-index satisfying $j + |\alpha| \leq 3$.
%Then
%\begin{equation*}
%\int_0^{1} \abs{\nabla_{\rm h} u}_{\tilde{X}_{\rho,\lambda}}^2 dt <+\infty \  \Longrightarrow \ \forall\  \gamma\in\mathbb Z_+^2,\quad 	\int_0^{1} \norm{H_\lambda^{\gamma}\partial_t^j\partial^{ \alpha } \nabla_{\rm h}u}^2_{L^2}dt<+\infty,
%\end{equation*}
%recalling $H^{\beta}_{\lambda}$ is given in \eqref{hlambda}.
We repeat the argument used in the estimate \eqref{k5} of $K_5$ with $\partial^\alpha$ replaced by $\partial_t^j\partial^\alpha$ to conclude that
\begin{align*}	
&\int_0^{1} t^{-1} \sum_{\stackrel{\beta \in \mathbb Z_+^2}{\beta_1\ge 1}} r^{\alpha_1}L_{\rho,\abs{\beta}}^2 \norm{H_\lambda^{\beta}\partial_t^j\partial^{ \alpha } u}^2_{L^2} dt\leq \int_0^{1}   \sum_{\stackrel{\beta \in \mathbb Z_+^2}{\beta_1\ge 1}} r^{\alpha_1}\frac{\beta_1}{t}L_{\rho,\abs{\beta}}^2 \norm{H_\lambda^{\beta}\partial_t^j\partial^{ \alpha } u}^2_{L^2} dt\\
	 &\qquad\qquad\leq C  \int_0^1\Big( r^{-\frac{9}{2}}\abs{u}_{\tilde{X}_{\rho,\lambda}}\abs{u}_{\tilde{Y}_{\rho,\lambda}}^2+   r^{-\frac{9}{2}}\abs{u}_{\tilde{X}_{\rho,\lambda}}\abs{\nabla_{\rm h} u}_{\tilde{X}_{\rho,\lambda}}^2+ \lambda  \abs{\nabla_{\rm h}u}_{\tilde{X}_{\rho,\lambda}}^2\Big) dt.
\end{align*}
Thus if \eqref{sht+} holds, then
\begin{eqnarray*}
	\int_0^{1} t^{-1} \sum_{\stackrel{\beta \in \mathbb Z_+^2}{\beta_1\ge 1}} r^{\alpha_1}L_{\rho,\abs{\beta}}^2 \norm{H_\lambda^{\beta}\partial_t^j\partial^{ \alpha } u}^2_{L^2} dt<+\infty.
\end{eqnarray*}
This,  with the continuity of
\begin{equation*}
 	t\mapsto \sum_{\stackrel{\beta \in \mathbb Z_+^2}{\beta_1\ge 1}} r^{\alpha_1}L_{\rho,\abs{\beta}}^2 \norm{H_\lambda^{\beta}\partial_t^j\partial^{ \alpha } u}^2_{L^2},
\end{equation*}
implies that
\begin{equation*}
\forall\ (j,\alpha)\in\mathbb Z_+\times\mathbb Z_+^3\  \textrm{ with }\  j+\abs \alpha\leq 3,\quad	\lim_{t\rightarrow 0}\sum_{\stackrel{\beta \in \mathbb Z_+^2}{\beta_1\ge 1}} r^{\alpha_1}L_{\rho,\abs{\beta}}^2 \norm{H_\lambda^{\beta}\partial_t^j\partial^{ \alpha } u}^2_{L^2}=0.
\end{equation*}
Consequently,  the desired assertion of Lemma \ref{lemshort}
  follows by observing that
\begin{equation*}
	 \lim_{t\to0} \abs{u}_{\tilde{X}_{\rho,\lambda}}^2= \lim_{t\to0} \abs{u}_{X_{\rho}}^2+\lim_{t\rightarrow 0} \sum_{\stackrel{(j,\alpha)\in\mathbb Z_+\times\mathbb Z_+^3}{ j+\abs \alpha\leq 3}} \, \sum_{\stackrel{ \beta\in\mathbb Z_+^2}{\beta_1\geq 1} }r^{\alpha_1}  L_{\rho,\abs{\beta}}^2  \norm{H_\lambda^{\beta}\partial_t^j\partial^{ \alpha } u}^2_{L^2}
\end{equation*}
due to Remark \ref{rek:xyt}. The proof of Lemma \ref{lemshort} is completed.
\end{proof}
	
\begin{proof}[Completing the proof of Proposition \ref{prop:t}]
Combining     estimate \eqref{des} in Lemma \ref{vecx} with the  assertion of Lemma  \ref{lemshort}, we apply the same bootstrap argument as that after \eqref{ces} to conclude  estimate \eqref{pri:ret3} for sufficiently small $\lambda$.  The proof of Proposition \ref{prop:t} is thus completed.
\end{proof}

\section{Proof of Theorem \ref{thm:smoothing}: smoothing effect for the variable $x_2$}\label{sec:x2}

 Having established the global existence and uniqueness for system \eqref{ANS} (Theorem \ref{thm:main1}) as well as the Gevrey smoothing effect in the time variable $t$ (Proposition \ref{prop:t}),  we  derive in this part the Gevrey smoothing effect for the variable $x_2$.  Before stating the main result, we first introduce   $M_{\rho,\beta}$ and two new auxiliary norms $\abs{\cdot}_{ X_{\rho,\lambda}}$ and $ \abs{\cdot}_{ Y_{\rho,\lambda}},$ which correspond to
$L_{\rho,m}$ and the two norms $\abs{\cdot}_{X_{\rho}}$ and $\abs{\cdot}_{Y_{\rho}}$ in \eqref{def:L} and Definition \ref{def:normx3}, respectively.

\begin{definition}
	\label{def:normx2}   Let $\rho$ be as given in \eqref{def:rerho} and let  $0<r,\lambda<1$ be the parameters determined in Theorem \ref{prop:wellposedness} and Proposition \ref{prop:t},respectively.   We define  the  space-time derivatives $D_\lambda^\beta$ by
	\begin{equation}\label{dlambda}
	D_\lambda^\beta\stackrel{\rm def}{=}\lambda^{\beta_1}t^{\frac{\beta_1}{2}}  \partial_2 ^{\beta_1}\partial_3^{\beta_2} \ \textrm{ for } \  \beta=(\beta_1,\beta_2)\in\mathbb Z_+^2,
\end{equation}
and  introduce  two new auxiliary norms $\abs{\cdot}_{  X_{\rho,\lambda}}$ and $\abs{\cdot}_{ Y_{\rho,\lambda}} $   as follows:
 \begin{eqnarray*}
 \left\{
\begin{aligned}
	 &\abs{g}_{  X_{\rho,\lambda}}^2\stackrel{\rm def}{=} \sum_{\stackrel{(j,\alpha)\in\mathbb Z_+\times\mathbb Z_+^3}{ j+\abs \alpha\leq 3}} \, \sum_{\beta\in\mathbb Z_+^2}   r^{ \alpha_1}  M_{\rho,\beta}^2  \norm{ D_\lambda^{\beta}\partial_t^j\partial^{ \alpha }g}^2_{L^2},\\
	 &\abs{g}_{  Y_{\rho,\lambda}}^2\stackrel{\rm def}{=}   \sum_{\stackrel{(j,\alpha)\in\mathbb Z_+\times\mathbb Z_+^3}{ j+\abs \alpha\leq 3}} \, \sum_{ \beta\in\mathbb Z_+^2  } r^{\alpha_1} \inner{\abs\beta +1}    M_{\rho,\beta}^2  \norm{ D_\lambda^{\beta}\partial_t^j\partial^{ \alpha }g}^2_{L^2},
	 \end{aligned}
	 \right.
\end{eqnarray*}
where, throughout the paper,
\begin{equation}\label{mrho}
	M_{\rho,\beta}\stackrel{\rm def}{=}\frac{\rho^{\abs\beta+1}(\abs\beta+1)^{6+2\sigma}}{(\beta_2!)^{\sigma-\delta}(\abs\beta!)^{\delta}} \ \textrm{ for }\  \beta=(\beta_1,\beta_2)\in\mathbb Z_+^2,
\end{equation}
with $\delta=\delta(\sigma)$ defined as in \eqref{gamma1}, namely,
\begin{equation*}
	\delta=\max\Big\{1,\frac{\sigma+1}{3},\frac{2\sigma-1}{4}\Big\}.
\end{equation*}
 \end{definition}

  \begin{remark}\label{rmk:xyz}
  Note for $\beta=(\beta_1,\beta_2)\in\mathbb Z_+^2$ with  $\beta_1=0,$
  \begin{equation*}
  	M_{\rho,\beta}=L_{\rho,\beta_2}.
  \end{equation*}
 Therefore, recalling $\abs{\cdot}_{X_{\rho}}$ is given in  Definition \ref{def:normx3},  we obtain
 \begin{equation}\label{xandtx}
 \abs{g}_{  X_{\rho,\lambda}  }^2=\abs{g}_{X_{\rho}  }^2+  \sum_{\stackrel{(j,\alpha)\in\mathbb Z_+\times\mathbb Z_+^3}{ j+\abs \alpha\leq 3}} \, \sum_{\stackrel{ \beta\in\mathbb Z_+^2}{\beta_1\geq 1} } r^{\alpha_1} M_{\rho,\beta}^2  \norm{ D_\lambda^{\beta}\partial_t^j\partial^{ \alpha }g}^2_{L^2},
 \end{equation}
 and a similar identity  holds for  $\abs{g}_{  Y_{\rho,\lambda}}^2.$ Moreover, we have
% By direct verification it holds that
 \begin{equation}\label{xyx2}
 \abs{g}_{  X_{\rho,\lambda}  }^2\leq \abs{g}_{  Y_{\rho,\lambda}}^2,
 \end{equation}
 which is analogous  to inequality \eqref{xyone}.
\end{remark}

Using the above notations,  the main result concerning the Gevrey smoothing effect in variable $x_2$ is stated as follows.

\begin{proposition}\label{prop:x2}
Let $u\in L^{\infty}([0,+\infty[, G_{\rho,\sigma,3})$ be the solution to system \eqref{ANS} constructed in Theorem \ref{thm:main1}.
  Then by shrinking  the numbers $r,\eps_0$ in  Theorem \ref{prop:wellposedness}  if necessary, we can find  a small constant $0<\lambda<1$ such that
\begin{equation}\label{pri:ret2}
\forall\ t\ge 0,\quad e^{\frac t2}\abs{u(t)}_{ X_{\rho,\lambda}}^2+\int^t_0e^{\frac s2}\abs{\nabla_{\rm h}u(s)}_{ X_{\rho,\lambda}}^2ds\leq \varepsilon_0^2,
\end{equation}
recalling the norm $\abs{\cdot}_{  X_{\rho,\lambda}}$ is given in Definition \ref{def:normx2}.
\end{proposition}

The proof of Proposition \ref{prop:x2} is presented in Subsection \ref{subsec:proof}. As a preliminary step, we first establish two technical lemmas in Subsection \ref{subsec:lemm} which with Lemma \ref{lem:J2} reveal  the intrinsic anisotropic Gevrey regularity in the variable $x_2$.

\subsection{Technical lemmas}\label{subsec:lemm}

To improve the $x_2$-regularity of the solution $u$  at positive times, the major difficulty arises from  handling the convection term $u_3\partial_3 u$. This prevents us from achieving  the analyticity, limiting the regularity to the Gevrey class with index $\max\{\frac{2\sigma-1}{4}, \frac{\sigma+1}{3}\}$.  We need the following two technical lemmas when dealing with $u_3\partial_3 u$, the first one requiring the condition $\delta\geq \frac{2\sigma-1}{4}$ and the second one asking  $\delta\geq \frac{\sigma+1}{3}.$ Note these technical lemmas are only used in the proof of Lemma \ref{lem:J2}.

\begin{lemma}\label{cr1}
 Let  $M_{\rho,\beta}$ be given in \eqref{mrho},
 namely,
 \begin{equation*}
 	M_{\rho,\beta}=\frac{\rho^{\abs\beta+1}(\abs\beta+1)^{6+2\sigma}}{(\beta_2!)^{\sigma-\delta}(\abs\beta!)^{\delta}} \ \textrm {with } \  1\leq \delta\leq\sigma.
 \end{equation*}
If $ \delta\geq \frac{2\sigma-1}{4},$  then  for    any $\gamma\leq\beta\in \mathbb Z_+^2$ with $\gamma_1=\beta_1\ge 1$  and $1\leq \abs\gamma\leq [\frac{\abs\beta}{2}]$, we have
\begin{equation*}
\binom{\beta}{\gamma}  \frac{M_{\rho,\beta }^2}{M_{\rho,\tilde\beta}M_{\rho,\gamma}M_{\rho,\beta-\gamma+(0,1)}(\abs\beta-\abs\gamma+2)^{\frac12} } \leq\frac{C}{(\abs{\gamma}+1)^{6}},
\end{equation*}
recalling $\tilde \beta=\beta-(1,0)=(\beta_1-1,\beta_2).$
\end{lemma}

\begin{proof}
For  $\gamma\leq \beta\in\mathbb Z_+^2$ with  $\gamma_1=\beta_1$ and $1\leq\abs{\gamma}\leq [\frac{\abs{\beta}}{2}]$,  we have
\begin{equation}\label{fact1}
	\beta_2-\gamma_2=\abs\beta-\abs\gamma\approx \abs\beta\ \textrm{ and }\ \binom{\beta}{\gamma}=\frac{\beta_2!}{\gamma_2!(\beta_2-\gamma_2)!},
\end{equation}
 and thus
 \begin{align*}
	&\frac{1}{M_{\rho,\beta-\gamma+(0,1)}(\abs\beta-\abs\gamma+2)^{\frac12}} =\frac{[(\beta_2-\gamma_2+1)!]^{\sigma-\delta}[( \abs\beta-\abs\gamma+1)!]^\delta}{\rho^{\abs\beta-\abs\gamma+2}(\abs \beta-\abs\gamma+2)^{6+2\sigma +\frac12}}\\
	&=  \frac{[(\beta_2-\gamma_2+1)!]^{\sigma} }{\rho^{\abs\beta-\abs\gamma+2}(\abs \beta-\abs\gamma+2)^{6+2\sigma } (\beta_2-\gamma_2+2)^{\frac12} } \leq C \frac{[(\beta_2-\gamma_2)!]^{\sigma} \abs\beta^{\sigma-\frac12} }{\rho^{\abs\beta-\abs\gamma+2}(\abs \beta-\abs\gamma+2)^{6+2\sigma } }.
\end{align*}
As a result, by the above estimate and \eqref{fact1},
\begin{equation*}
	\begin{aligned}
	&  \frac{M_{\rho,\beta }}{M_{\rho,\gamma}M_{\rho,\beta-\gamma+(0,1)}(\abs\beta-\abs\gamma+2)^{\frac12} }\\
	&\leq C
	     \frac{ \rho^{\abs\beta+1}(\abs\beta+1)^{6+2\sigma}} {(\beta_2!)^{\sigma-{\delta} } (\abs\beta !) ^{{\delta} }} \frac{(\gamma_2!)^{\sigma-{\delta} } (\abs\gamma !) ^{{\delta} }}{ \rho^{\abs\gamma+1}(\abs\gamma+1)^{6+2\sigma}}\frac{[(\beta_2-\gamma_2)!]^{\sigma} \abs\beta^{\sigma-\frac12} }{\rho^{\abs\beta-\abs\gamma+2}(\abs \beta-\abs\gamma+2)^{6+2\sigma } }\\
	&\leq \frac{C \abs\beta^{\sigma-\frac12}  }{(\abs\gamma+1)^{6+2\sigma}}     \frac{(\gamma_2!)^{\sigma-{\delta}} (\abs\gamma !) ^{{\delta} } [( \beta_2- \gamma_2)!]^{\sigma}}{(\beta_2!)^{\sigma-{\delta}}(\abs\beta!)^\delta}.\\
	\end{aligned}
\end{equation*}
For the last factor above, we use the fact that $	p!q!\leq (p+q)!$ together with the condition  $\delta\leq\sigma$ to compute
\begin{equation*}
    \begin{aligned}
        &\frac{(\gamma_2!)^{\sigma-{\delta}} (\abs\gamma !) ^{{\delta} } [( \beta_2- \gamma_2)!]^{\sigma}}{(\beta_2!)^{\sigma-{\delta}}(\abs\beta!)^\delta}=\frac{(\gamma_2!)^{\sigma-{\delta}} (\abs\gamma !) ^{{\delta} } [( \beta_2- \gamma_2)!]^{\sigma-\delta}[( \beta_2- \gamma_2)!]^{\delta}}{(\beta_2!)^{\sigma-{\delta}}(\abs\beta!)^\delta}\\
        &= \frac{[\gamma_2!(\beta_2-\gamma_2)!]^{\sigma-\delta}}{(\beta_2!)^{\sigma-{\delta}}} \frac{[(\abs{\gamma}-1)!]^\delta[(\beta_2-\gamma_2+1)!]^{\delta}}{(\abs\beta!)^\delta} \frac{\abs{\gamma}^\delta}{(\beta_2-\gamma_2+1)^\delta}\\
        &\leq  C\bigg (\frac{ (\abs{\gamma}-1)!  (\beta_2-\gamma_2+1)! }{ \abs\beta!} \bigg)^{\delta}\frac{\abs{\gamma}^\delta}{\abs{\beta}^\delta}\\
      &  \leq C \frac{[\gamma_2!(\beta_2-\gamma_2)!]^\delta}{(\beta_2!)^\delta}\frac{\abs{\gamma}^\delta}{\abs{\beta}^\delta},
    \end{aligned}
\end{equation*}
where the first inequality holds because of \eqref{fact1} and the last one relies on the estimate below. Noting $\gamma_1=\beta_1\ge 1$ and using the property \eqref{factorial}  of   binomial coefficients,  we obtain
\begin{multline*}
 \frac{(\abs{\gamma}-1)!(\beta_2-\gamma_2+1)!}{\abs{\beta}!}=\frac{(\beta_1+\gamma_2-1)!(\beta_2-\gamma_2+1)!}{(\beta_1+\beta_2)!}=\binom{ |(\beta_1,\beta_2)|}{|(\beta_1-1,\gamma_2)|}^{-1}\\
\leq \bigg[\binom{ \beta_1}{ \beta_1-1}\binom{ \beta_2}{ \gamma_2}\bigg]^{-1}
 \leq \frac{(\beta_1-1)!}{\beta_1!}\frac{\gamma_2!(\beta_2-\gamma_2)!}{\beta_2!}\leq  \frac{\gamma_2!(\beta_2-\gamma_2)!}{\beta_2!}.
\end{multline*}
 Combining the above estimates and using $\delta\ge 1$ and the fact \eqref{fact1}, we have
\begin{align*}
    \frac{M_{\rho,\beta }}{M_{\rho,\gamma}M_{\rho,\beta-\gamma+(0,1)}(\abs\beta-\abs\gamma+2)^{\frac12} }
  \leq  \frac{C \abs\beta^{\sigma-\delta-\frac12}  }{(\abs\gamma+1)^{6+2\sigma-\delta}}     \frac{[\gamma_2!(\beta_2-\gamma_2)!]^{\delta}}{(\beta_2!)^{\delta}}
\end{align*}
 and thus, recalling $\binom{\beta}{\gamma}=\binom{\beta_2}{\gamma_2}$ for $\beta_1=\gamma_1,$
\begin{multline*}
  \binom{\beta}{\gamma}  \frac{M_{\rho,\beta }}{M_{\rho,\gamma}M_{\rho,\beta-\gamma+(0,1)}(\abs\beta-\abs\gamma+2)^{\frac12} } \\
  \leq  \frac{C \abs\beta^{\sigma-\delta-\frac12}  }{(\abs\gamma+1)^{6+2\sigma-\delta}}     \frac{[\gamma_2!(\beta_2-\gamma_2)!]^{\delta-1}}{(\beta_2!)^{\delta-1}}\leq  \frac{C \abs\beta^{\sigma-\delta-\frac12}  }{(\abs\gamma+1)^{6+2\sigma-\delta}}.
\end{multline*}
This, with the fact that
\begin{equation}\label{betatilde}
	\begin{aligned}
	\frac{M_{\rho,\beta}}{M_{\rho,\tilde\beta}}= \frac{ \rho^{\abs{\beta}+1}(\abs{\beta}+1)^{6+2\sigma}}{(\beta_2!)^{\sigma-{\delta} }(\abs{\beta}!)^{{\delta} }}\frac{(\beta_2!)^{\sigma-{\delta} }[(\abs{\beta}-1)!]^{{\delta} }}{\rho^{\abs{\beta}}\abs{\beta}^{6+2\sigma}}\leq \frac{\rho_02^{6+2\sigma}}{\abs\beta^{\delta}},
	\end{aligned}
\end{equation}
yields
\begin{multline*}
	 \binom{\beta}{\gamma}  \frac{M_{\rho,\beta }^2}{M_{\rho,\tilde\beta}M_{\rho,\gamma}M_{\rho,\beta-\gamma+(0,1)}(\abs\beta-\abs\gamma+2)^{\frac12} }\\
	 = \frac{M_{\rho,\beta }}{M_{\rho,\tilde\beta}}\times \binom{\beta}{\gamma}  \frac{M_{\rho,\beta }}{M_{\rho,\gamma}M_{\rho,\beta-\gamma+(0,1)}(\abs\beta-\abs\gamma+2)^{\frac12} }	 \\
	 \leq \frac{C \abs\beta^{\sigma-2\delta-\frac12}  }{(\abs\gamma+1)^{6+2\sigma-\delta}}     \leq \frac{C }{(\abs\gamma+1)^{6}},
	\end{multline*}
provided $2\delta-\sigma+\frac{1}{2}\ge 0.$ The proof of Lemma \ref{cr1} is completed.
\end{proof}

\begin{lemma}\label{cr2}
 Let  $M_{\rho,\beta}$ be given in \eqref{mrho},
 namely,
 \begin{equation*}
 	M_{\rho,\beta}=\frac{\rho^{\abs\beta+1}(\abs\beta+1)^{6+2\sigma}}{(\beta_2!)^{\sigma-\delta}(\abs\beta!)^{\delta}} \ \textrm {with } \  1\leq \delta\leq\sigma.
 \end{equation*}
If $ \delta\geq \frac{ \sigma+1}{3},$   then  for    any $\gamma\leq\beta\in \mathbb Z_+^2$ with $\gamma_1<\beta_1$  and $1\leq \abs\gamma\leq [\frac{\abs\beta}{2}]$, we have
	 \begin{equation*}
 \binom{\beta}{\gamma} \frac{M_{\rho,\beta}^2}{M_{\rho,\tilde\beta}M_{\rho,\gamma}M_{\rho,\beta-\gamma+(-1,1)} } \leq\frac{C}{(\abs{\gamma}+1)^{6}},
\end{equation*}
recalling $\tilde \beta=\beta-(1,0)=(\beta_1-1,\beta_2).$
\end{lemma}

\begin{proof}
For any $\gamma\leq\beta$ with $\gamma_1<\beta_1$  and $1\leq \abs\gamma\leq [\frac{\abs\beta}{2}]$, we have $\abs\beta-\abs\gamma+1\approx \abs\beta$ and thus
\begin{equation*}
	\begin{aligned}
	&\binom{\beta}{\gamma}  \frac{M_{\rho,\beta}}{M_{\rho,\gamma}M_{\rho,\beta-\gamma+(-1,1)} }\leq \binom{\abs \beta}{\abs \gamma}  \frac{M_{\rho,\beta}}{M_{\rho,\gamma}M_{\rho,\beta-\gamma+(-1,1)} }\\
	&\leq \frac{\abs \beta!}{\abs \gamma!(\abs \beta-\abs \gamma)!} \frac{\rho^{\abs\beta+1}(\abs\beta+1)^{6+2\sigma}}{(\beta_2!)^{\sigma-{\delta} } (\abs\beta !) ^{{\delta}}}\\
	&\qquad\times \frac{(\gamma_2!)^{\sigma-{\delta} } (\abs\gamma !) ^{{\delta} }}{ \rho^{\abs\gamma+1}(\abs\gamma+1)^{6+2\sigma}}\frac{[(\beta_2-\gamma_2+1)!]^{\sigma-\delta}[( \abs\beta-\abs\gamma)!]^\delta}{\rho^{\abs\beta-\abs\gamma+1}(\abs \beta-\abs\gamma+1)^{6+2\sigma }}\\
	&\leq   \frac{C}{(\abs\gamma+1)^{6+2\sigma}}  \frac{  (\abs\gamma !) ^{{\delta} -1} [(\abs\beta-\abs\gamma)!]^{\delta-1}}{ (\abs \beta!)^{\delta-1}}\frac{(\gamma_2!)^{\sigma-{\delta} }  [(\beta_2-\gamma_2+1)!]^{\sigma-\delta} }{(\beta_2!)^{\sigma-\delta} }.
	\end{aligned}
\end{equation*}
Moreover, using  the fact that $p!q!\leq (p+q)!$ and $\delta\leq\sigma$ gives that
\begin{multline*}
	\frac{(\gamma_2!)^{\sigma-{\delta} }[ (\beta_2-\gamma_2+1)!]^{\sigma-\delta}}{(\beta_2!)^{\sigma-\delta}}=\frac{(\gamma_2!)^{\sigma-{\delta} }[ (\beta_2-\gamma_2)!]^{\sigma-\delta} (\beta_2-\gamma_2+1)^\delta}{(\beta_2!)^{\sigma-\delta}}\\
    \leq (\beta_2-\gamma_2+1)^{\sigma-\delta}  \leq  (\abs \beta-\abs\gamma+1)^{\sigma-\delta}
\end{multline*}
and that, observing $\abs\gamma\geq1$ and $\delta\ge 1$,
\begin{equation*}
\begin{aligned}
	\frac{  (\abs\gamma !) ^{{\delta} -1} [(\abs\beta-\abs\gamma)!]^{\delta-1}}{ (\abs \beta!)^{\delta-1}}&\leq \frac{  [(\abs\gamma-1)!]^{\delta-1}[(\abs\beta-\abs\gamma+1)!]^{\delta-1}}{  (\abs \beta!)^{\delta-1} } \frac{\abs\gamma^{\delta-1}}{(\abs\beta-\abs\gamma+1)^{\delta-1}}\\
	&\leq     \frac{ \abs\gamma^{\delta-1}}{(\abs\beta-\abs\gamma+1)^{\delta-1}}.
	\end{aligned}
\end{equation*}
Consequently, combining the above estimates, we conclude that, for any $1\leq\abs\gamma\leq [\frac{\abs\beta}{2}]$ with $\gamma_1<\beta_1$,
\begin{equation*}
	\binom{\beta}{\gamma}  \frac{M_{\rho,\beta}}{M_{\rho,\gamma}M_{\rho,\beta-\gamma+(-1,1)} } \leq  \frac{C}{(\abs\gamma+1)^{6+2\sigma-(\delta-1)}}  \frac{1}{(\abs\beta-\abs\gamma+1)^{\delta-1-(\sigma-\delta)}},
\end{equation*}
which with   \eqref{betatilde} yields
\begin{multline*}
	\binom{\beta}{\gamma}  \frac{M_{\rho,\beta}^2}{M_{\rho,\tilde\beta}M_{\rho,\gamma}M_{\rho,\beta-\gamma+(-1,1)} }=\frac{M_{\rho,\beta} }{M_{\rho,\tilde\beta}} \times \binom{\beta}{\gamma}  \frac{M_{\rho,\beta}}{ M_{\rho,\gamma}M_{\rho,\beta-\gamma+(-1,1)} } \\
	\leq  \frac{C}{(\abs\gamma+1)^{6+2\sigma-(\delta-1)}}  \frac{1}{(\abs\beta-\abs\gamma+1)^{2\delta-1-(\sigma-\delta)}}\leq \frac{C}{(\abs\gamma+1)^{6}} ,
\end{multline*}
provided $3\delta-1-\sigma\geq 0.$ This completes the proof of Lemma \ref{cr2}.
\end{proof}

 \subsection{Proof of Proposition \ref{prop:x2}} \label{subsec:proof}

The rest part of this section is devoted to the proof of Proposition \ref{prop:x2}. As shown in Section \ref{sec:wellposdeness}, it suffices to estimate the terms involving pure spatial derivatives in the definition of  $\abs{u}_{X_{\rho,\lambda}}$; the terms with time derivatives can be handled analogously. Following a procedure similar to the derivation of \eqref{s} without additional difficulty,  we have that , for any $\alpha\in\mathbb Z_+^3$ with $\abs \alpha\leq 3,$
\begin{multline}\label{est:x2}
 \frac{1}{2}\frac{d}{dt} \sum_{\stackrel{\beta \in \mathbb Z_+^2}{\beta_1\ge 1}} r^{\alpha_1}M_{\rho,\beta}^2\norm{D_\lambda^\beta\partial^{\alpha} u}_{L^2}^2+ \sum_{\stackrel{\beta \in \mathbb Z_+^2}{\beta_1\ge 1}}r^{\alpha_1}M_{\rho,\beta}^2\norm{D_\lambda^\beta\partial^{\alpha} \nabla_{\rm h} u}_{L^2}^2\\
 -\frac{\rho^{'}}{\rho}  \sum_{\stackrel{\beta \in \mathbb Z_+^2}{\beta_1\ge 1}}r^{\alpha_1}(\abs{\beta }+1)M_{\rho,\beta}^2\norm{D_\lambda^\beta\partial^{\alpha} u}_{L^2}^2\leq \sum_{\ell=1}^5 J_\ell,
\end{multline}
where
\begin{equation}
	\label{J1-J5}
	  \left\{
    \begin{aligned}
&J_1= - \sum_{\stackrel{\beta \in \mathbb Z_+^2}{\beta_1\ge 1}}r^{\alpha_1}M_{\rho,\beta}^2  \inner{  D_\lambda^\beta\partial^{\alpha}\big ((u_{\rm h}\cdot\nabla_{\rm h}) u\big ),\   D_\lambda^\beta\partial^{\alpha} u}_{L^2},\\
&J_2=- \sum_{\stackrel{\beta \in \mathbb Z_+^2}{\beta_1\ge 1}}r^{\alpha_1}M_{\rho,\beta}^2\inner{ D_\lambda^\beta\partial^{\alpha} (u_3\partial_3u),\   D_\lambda^\beta\partial^{\alpha} u}_{L^2},\\
&J_3= \sum_{\stackrel{\beta \in \mathbb Z_+^2}{\beta_1\ge 1}}r^{\alpha_1}M_{\rho,\beta}^2\inner{  D_\lambda^\beta\partial^{\tilde \alpha}\nabla p, \   D_\lambda^\beta\partial^{\alpha}\partial_1 u}_{L^2},\\
&J_4= \sum_{\stackrel{\beta \in \mathbb Z_+^2}{\beta_1\ge 1}}r^{\alpha_1}M_{\rho,\beta}^2\int_{\mathbb R^2} \big[ (  D_\lambda^\beta\partial^{\alpha} u)  D_\lambda^\beta\partial^{\tilde \alpha} \big(\partial_tu+u\cdot\nabla u-\partial_2^2u\big)\big]\big|^{x_1=1}_{x_1=0}\ dx_2dx_3,\\
&J_5= \frac12   \sum_{\stackrel{\beta \in \mathbb Z_+^2}{\beta_1\ge 1}} r^{\alpha_1} \frac{ \beta_1}{t} M_{\rho,\beta}^2\norm{D_\lambda^\beta\partial^{\alpha} u}_{L^2}^2,
    \end{aligned}
    \right.
\end{equation}
recalling  $\tilde\alpha=\alpha-(1,0,0)=(\alpha_1-1,\alpha_2,\alpha_3).$   In the following discussion we will proceed to estimate the terms $J_\ell, 1\leq \ell\leq 5.$

\begin{lemma}[Estimate on $J_1$]\label{lem:J1}
Let $J_1$ be given in \eqref{J1-J5}. For any $\alpha\in\mathbb Z_+^3$ with $\abs\alpha \leq 3,$ it holds that
 \begin{equation}\label{est:J1}
J_1\leq   C r^{-\frac{9}{2}}\abs{u}_{ X_{\rho,\lambda}}^2\abs{\nabla_{\rm h}u}_{ X_{\rho,\lambda}}\leq   C r^{-\frac{9}{2}}\abs{u}_{ X_{\rho,\lambda}}\abs{u}_{ Y_{\rho,\lambda}}^2+ Cr^{-\frac{9}{2}}\abs{u}_{ X_{\rho,\lambda}} \abs{\nabla_{\rm h}u}_{ X_{\rho,\lambda}}^2,
 \end{equation}
 recalling $\abs{\cdot}_{ X_{\rho,\lambda}}$ and $\abs{\cdot}_{ Y_{\rho,\lambda}}$ are given in Definition \ref{def:normx2}.
\end{lemma}

\begin{proof}
The proof  is quite similar to that in  Lemma \ref{lem:S1}.
Indeed, analogous to \eqref{ineq1} and \eqref{ineq2}, we have accordingly the following two estimates (see Appendix \ref{sec:appendix} for the proof):
\begin{equation}\label{ineq5}
\forall\ \gamma\leq\beta\in\mathbb Z_+^2\ \textrm{ with }\ 0\leq \abs{\gamma }\leq \Big[\frac{\abs{\beta }}{2}\Big],  \quad     \binom{\beta }{\gamma }\frac{M_{\rho,\beta}}{M_{\rho,\gamma}M_{\rho,\beta -\gamma}}\leq\frac{C}{(\abs{\gamma }+1)^6},
\end{equation}
and
\begin{equation}\label{ineq6}
 \forall\ \gamma\leq\beta\in\mathbb Z_+^2\ \textrm{ with } \  \Big[\frac{\abs{\beta }}{2}\Big]+1\leq \abs{\gamma }\leq \abs{\beta },\quad   \binom{\beta }{\gamma }\frac{M_{\rho,\beta}}{M_{\rho,\gamma}M_{\rho,\beta -\gamma}} \leq\frac{C}{(\abs{\beta -\gamma }+1)^6}.
\end{equation}
These estimates allow us to adapt the argument from the proof of Lemma \ref{lem:S1} and use the fact \eqref{xyx2}  to conclude  assertion \eqref{est:J1}.  The proof of Lemma \ref{lem:J1} is completed.
\end{proof}

\begin{lemma}[Estimate on $J_2$]\label{lem:J2}
Let $J_2$ be given in \eqref{J1-J5}, namely,
\begin{equation*}
	J_2=- \sum_{\stackrel{\beta \in \mathbb Z_+^2}{\beta_1\ge 1}}r^{\alpha_1}M_{\rho,\beta}^2\inner{ D_\lambda^\beta\partial^{\alpha} (u_3\partial_3u),\   D_\lambda^\beta\partial^{\alpha} u}_{L^2}.
\end{equation*}
For any $\alpha\in\mathbb Z_+^3$ with $\abs\alpha\leq 3,$ it holds that
\begin{equation}\label{est:J2}
    J_2\leq  Cr^{-\frac{9}{2}}\abs{u}_{ X_{\rho,\lambda}}\abs{u}_{ Y_{\rho,\lambda}}^2
    +C\lambda^2 t r^{-\frac{9}{2}}\abs{u}_{ X_{\rho,\lambda}}\abs{\nabla_{\rm h}u}_{ X_{\rho,\lambda}}^2,
\end{equation}
recalling $\abs{\cdot}_{ X_{\rho,\lambda}}$ and $\abs{\cdot}_{ Y_{\rho,\lambda}}$ are given in Definition \ref{def:normx2}.
\end{lemma}

\begin{proof}
We begin by decomposing $J_2$ as follows:
\begin{equation}\label{+J2}
    \begin{aligned}
 J_2=&-\sum_{\stackrel{\beta \in \mathbb Z_+^2}{\beta_1\ge 1}}r^{\alpha_1}M_{\rho,\beta}^2\inner{\partial^{ \alpha}(u_3 D_\lambda^{\beta} \partial_3u),\  D_\lambda^\beta\partial^{\alpha} u}_{L^2}\\
&- \sum_{\stackrel{\beta \in \mathbb Z_+^2}{\beta_1\ge 1}}\ \sum_{\stackrel{\gamma \leq\beta }{1\leq\abs{\gamma }\leq[\frac{\abs{\beta }}{2}]}}r^{\alpha_1}\binom{\beta }{\gamma }M_{\rho,\beta}^2\inner{\partial^{\alpha}[(D_{\lambda}^{\gamma }u_3)D_{\lambda}^{\beta -\gamma }\partial_3u],\  D_\lambda^\beta\partial^{\alpha} u}_{L^2}\\
&- \sum_{\stackrel{\beta \in \mathbb Z_+^2}{\beta_1\ge 1}}\ \sum_{\stackrel{\gamma \leq\beta }{[\frac{\abs{\beta }}{2}]+1\leq \abs{\gamma }\leq \abs{\beta }}}r^{\alpha_1}\binom{\beta }{\gamma }M_{\rho,\beta}^2\inner{\partial^{\alpha}[(D_{\lambda}^{\gamma }u_3)D_{\lambda}^{\beta -\gamma }\partial_3u],\  D_\lambda^\beta\partial^{\alpha} u}_{L^2}\\
\stackrel{\rm def}{=}&J_{2,1}+J_{2,2}+J_{2,3}.
    \end{aligned}
\end{equation}
We then estimate the right-hand side terms of \eqref{+J2} through three steps.

{\it Step 1.} The term $J_{2,1}$ can be estimated similarly to the bound for $S_{2,1}$ in \eqref{+s2} (established in \eqref{est:S21}). Repeating the proof of \eqref{est:S21} yields
\begin{equation}
\label{est:J21}
 J_{2,1}\leq  C \sum_{\stackrel{\beta \in \mathbb Z_+^2}{\beta_1\ge 1}}M_{\rho,\beta}^2 \norm{u}_{H^3} \norm{D_{\lambda}^{\beta }u}_{H^3}^2
 \leq  Cr^{-\frac{9}{2}}\abs{u}_{ X_{\rho,\lambda}}^3\leq Cr^{-\frac{9}{2}}\abs{u}_{ X_{\rho,\lambda}}\abs{u}_{ Y_{\rho,\lambda}}^2.
\end{equation}

{\it Step 2.}
To estimate   $J_{2,2}$  in \eqref{+J2},  we split it as:
\begin{equation}\label{+J22}
    \begin{aligned}
J_{2,2}= &- \sum_{\stackrel{\beta\in \mathbb Z_+^2}{\beta_1\ge 1}}\,\sum_{\stackrel{\gamma\leq\beta,\gamma_1 =\beta_1}{1\leq\abs{\gamma}\leq[\frac{\abs{\beta}}{2}]}}r^{\alpha_1}\binom{\beta}{\gamma} M_{\rho,\beta}^2\inner{\partial^{ \alpha}[(D_{\lambda}^\gamma u_3)D_{\lambda}^{\beta-\gamma} \partial_3u],\ D_\lambda^\beta\partial^{\alpha} u}_{L^2}\\
&-\sum_{\stackrel{\beta\in \mathbb Z_+^2}{\beta_1\ge 1}}\,\sum_{\stackrel{\gamma\leq\beta,\gamma_1 <\beta_1}{1\leq\abs{\gamma}\leq[\frac{\abs{\beta}}{2}]}}r^{\alpha_1}\binom{\beta}{\gamma} M_{\rho,\beta}^2\inner{\partial^{ \alpha}[(D_{\lambda}^\gamma u_3)D_{\lambda}^{\beta-\gamma} \partial_3u],\ D_\lambda^\beta\partial^{\alpha} u}_{L^2}.
    \end{aligned}
\end{equation}
For the first term on the right-hand side of \eqref{+J22},  using  Lemma \ref{cr1} and recalling  $D_\lambda^{\beta}=\lambda t^{\frac12}D_\lambda^{\tilde \beta}\partial_2$ for  $\tilde \beta=\beta-(1,0)$  from \eqref{dlambda}, we obtain
\begin{equation}\label{j221}
\begin{aligned}
&- \sum_{\stackrel{\beta\in \mathbb Z_+^2}{\beta_1\ge 1}}\,\sum_{\stackrel{\gamma\leq\beta,\gamma_1 =\beta_1}{1\leq\abs{\gamma}\leq[\frac{\abs{\beta}}{2}]}}r^{\alpha_1}\binom{\beta}{\gamma} M_{\rho,\beta}^2\inner{\partial^{ \alpha}[(D_{\lambda}^\gamma u_3)D_{\lambda}^{\beta-\gamma} \partial_3u],\ D_\lambda^\beta\partial^{\alpha} u}_{L^2}\\
& \leq   C\sum_{\stackrel{\beta\in \mathbb Z_+^2}{\beta_1\ge 1}}\sum_{\stackrel{\gamma\leq\beta,\gamma_1=\beta_1}{1\leq\abs{\gamma}\leq[\frac{\abs{\beta}}{2}]}}\binom{\beta}{\gamma}\frac{M_{\rho,\beta}^2}{M_{\rho,\tilde\beta} M_{\rho,\gamma}M_{\rho,\beta-\gamma+(0,1)}(\abs{\beta}-\abs{\gamma}+2)^\frac{1}{2}}\big(M_{\rho,\gamma}\norm{D_{\lambda}^{\gamma}u}_{H^3}\big)\\
&\qquad\times\big[(\abs{\beta}-\abs{\gamma}+2)^\frac{1}{2}M_{\rho,\beta-\gamma+(0,1)}\norm{D_{\lambda}^{\beta-\gamma} \partial_3u}_{H^3}\big]\times \big(\lambda t^{\frac12}M_{\rho,\tilde\beta}\norm{D_\lambda^{\tilde\beta} \partial_2u }_{H^3}\big)\\
&\leq  C\sum_{\stackrel{\beta\in \mathbb Z_+^2}{\beta_1\ge 1}}\sum_{\stackrel{\gamma\leq\beta,\gamma_1=\beta_1}{1\leq\abs{\gamma}\leq[\frac{\abs{\beta}}{2}]}} \frac{M_{\rho,\gamma}\norm{D_{\lambda}^{\gamma}u}_{H^3}}{(\abs\gamma+1)^6}\\
&\qquad \times \big[(\abs{\beta}-\abs{\gamma}+2)^\frac{1}{2}M_{\rho,\beta-\gamma+(0,1)}\norm{D^{\beta-\gamma}_{\lambda}\partial_3u}_{H^3}\big] \times \big(\lambda t^{\frac12}M_{\rho,\tilde\beta}\norm{D_\lambda^{\tilde\beta}(\partial_2u)}_{H^3}\big)\\
&\leq C\lambda t^{\frac12} r^{-\frac{9}{2}} \abs{u}_{ X_{\rho,\lambda}} \abs{u}_{ Y_{\rho,\lambda}} \abs{\nabla_{\rm h}u}_{ X_{\rho,\lambda}},
\end{aligned}
\end{equation}
where the second inequality follows from Lemma \ref{cr1}, and the third from the arguments in \eqref{S11} and \eqref{applyYong}.  Similarly, for the second term on the right-hand side of \eqref{+J22},  we use Lemma \ref{cr2} to obtain that
\begin{equation*}
	\begin{aligned}
	&-\sum_{\stackrel{\beta\in \mathbb Z_+^2}{\beta_1\ge 1}}\,\sum_{\stackrel{\gamma\leq\beta,\gamma_1 <\beta_1}{1\leq\abs{\gamma}\leq[\frac{\abs{\beta}}{2}]}}r^{\alpha_1}\binom{\beta}{\gamma} M_{\rho,\beta}^2\inner{\partial^{ \alpha}[(D_{\lambda}^\gamma u_3)D_{\lambda}^{\beta-\gamma} \partial_3u],\ D_\lambda^\beta\partial^{\alpha} u}_{L^2}	\\
	&\leq     C\sum_{\stackrel{\beta\in \mathbb Z_+^2}{\beta_1\ge 1}}\sum_{\stackrel{\gamma\leq\beta,\gamma_1<\beta_1}{1\leq\abs{\gamma}\leq[\frac{\abs{\beta}}{2}]}}\binom{\beta}{\gamma}\frac{M_{\rho,\beta}^2}{M_{\rho,\tilde\beta} M_{\rho,\gamma}M_{\rho,\beta-\gamma+(-1,1)}}\big(M_{\rho,\gamma}\norm{D_{\lambda}^{\gamma}u}_{H^3}\big)\\
&\qquad\times\big( M_{\rho,\beta-\gamma+(-1,1)}\norm{D_{\lambda}^{\beta-\gamma} \partial_3u}_{H^3}\big) \big(\lambda t^{\frac12}M_{\rho,\tilde\beta}\norm{D_\lambda^{\tilde\beta}(\partial_2u)}_{H^3}\big)\\
&\leq  C\sum_{\stackrel{\beta\in \mathbb Z_+^2}{\beta_1\ge 1}}\sum_{\stackrel{\gamma\leq\beta,\gamma_1<\beta_1}{1\leq\abs{\gamma}\leq[\frac{\abs{\beta}}{2}]}} \frac{M_{\rho,\gamma}\norm{D_{\lambda}^{\gamma}u}_{H^3}}{(\abs\gamma+1)^6} \\
&\qquad\qquad\times\big( \lambda t^{\frac12}M_{\rho,\beta-\gamma+(-1,1)}\norm{D_{\lambda}^{\beta-\gamma+(-1,1)} \partial_2u}_{H^3}\big) \big(\lambda t^{\frac12}M_{\rho,\tilde\beta}\norm{D_\lambda^{\tilde\beta}(\partial_2u)}_{H^3}\big)\\
&\leq C\lambda^2 t r^{-\frac{9}{2}} \abs{u}_{ X_{\rho,\lambda}}  \abs{\nabla_{\rm h}u}_{ X_{\rho,\lambda}}^2,
	\end{aligned}
\end{equation*}
where in the second inequality we use \eqref{dlambda}  to write  $D_{\lambda}^{\beta-\gamma}\partial_3u=\lambda t^{\frac12}D_{\lambda}^{\beta-\gamma+(-1,1)} \partial_2u$ for $\gamma_1<\beta_1.$  Substituting the above estimate and \eqref{j221} into \eqref{+J22} yields
 \begin{equation}\label{est:J22}
 \begin{aligned}
J_{2,2}&\leq C\lambda t^{\frac12} r^{-\frac{9}{2}} \abs{u}_{ X_{\rho,\lambda}} \abs{u}_{Y_{\rho,\lambda}} \abs{\nabla_{\rm h}u}_{ X_{\rho,\lambda}}+C\lambda^2 t r^{-\frac{9}{2}} \abs{u}_{ X_{\rho,\lambda}}  \abs{\nabla_{\rm h}u}_{ X_{\rho,\lambda}}^2\\
&\leq   C\lambda^2t r^{-\frac{9}{2}}\abs{u}_{  X_{\rho,\lambda}} \abs{\nabla_{\rm h}u}_{ X_{\rho,\lambda}}^2+Cr^{-\frac{9}{2}}\abs{u}_{ X_{\rho,\lambda}} \abs{u}_{ Y_{\rho,\lambda}}^2.
\end{aligned}
\end{equation}

 {\it Step 3.} It remains to estimate  $J_{2,3}$ in \eqref{+J2}  and  the treatment  is straightforward. Indeed,  observe
\begin{equation}\label{ineq9}
\forall\ \gamma\leq\beta \ \textrm{ with }\  \Big[\frac{\abs{\beta }}{2}\Big]+1\leq\abs{\gamma }\leq\abs{\beta },\quad \binom{\beta }{\gamma }\frac{M_{\rho,\beta}}{M_{\rho,\gamma}M_{\rho,\beta -\gamma +(0,1)}}\leq\frac{C}{(\abs{\beta -\gamma }+1)^{6}},
\end{equation}
 seeing Appendix \ref{sec:appendix} for the proof.
 Then we follow the arguments in \eqref{S11} and \eqref{applyYong} and use \eqref{xyx2}  to conclude that
\begin{equation}\label{est:J23}
    \begin{aligned}
J_{2,3}\leq &C\sum_{\stackrel{\beta \in \mathbb Z_+^2}{\beta_1\ge 1}}\ \sum_{\stackrel{\gamma \leq\beta }{[\frac{\abs{\beta }}{2}]+1\leq\abs{\gamma }\leq\abs{\beta }}}\binom{\beta }{\gamma }\frac{M_{\rho,\beta}}{M_{\rho,\gamma}M_{\rho,\beta -\gamma +(0,1)}}\big(M_{\rho,\gamma}\norm{D_{\lambda}^{\gamma }u}_{H^3}\big)\\
&\times\big(M_{\rho,\beta -\gamma +(0,1)}\norm{D_{\lambda}^{\beta -\gamma }\partial_3u}_{H^3}\big) \big(M_{\rho,\beta}\norm{D_\lambda^\beta u}_{H^3}\big)\\
\leq& Cr^{-\frac{9}{2}}\abs{u}_{ X_{\rho,\lambda}}^3\leq Cr^{-\frac{9}{2}}\abs{u}_{ X_{\rho,\lambda}}\abs{u}_{ Y_{\rho,\lambda}}^2.
    \end{aligned}
\end{equation}
Substituting estimates \eqref{est:J21}, \eqref{est:J22}  and \eqref{est:J23} into \eqref{+J2} yields the desired assertion \eqref{est:J2}.
This completes the proof of Lemma \ref{lem:J2}.
\end{proof}

\begin{lemma}[Estimates on $J_3$ and $J_4$]\label{lem:J34}
Let $J_3$ and $J_4$ be given in \eqref{J1-J5}. For any $\alpha\in\mathbb Z_+^3$ with $\abs\alpha\leq 3$, it holds that
\begin{equation}\label{est:J34}
    J_3+J_4\leq
    Cr^{-\frac92}\abs{u}_{ X_{\rho,\lambda}} \abs{u}_{ Y_{\rho,\lambda}}^2+C\Big( r^{-\frac92}\abs{u}_{ X_{\rho,\lambda}} +r^{\frac{1}{4}}\Big)\abs{\nabla_{\rm h}u}_{ X_{\rho,\lambda}}^2,
\end{equation}
recalling $\abs{\cdot}_{ X_{\rho,\lambda}}$ and $\abs{\cdot}_{ Y_{\rho,\lambda}}$ are given in Definition \ref{def:normx2}.
\end{lemma}

\begin{proof}
    The treatment of $J_3$ and $J_4$ are quite similar to those of $S_3$ and $S_4$ established in Lemmas \ref{lem:S3}-\ref{lem:S4}.  Indeed,   estimate \eqref{est:J34} will follow  if we follow  the  proofs of Lemmas  \ref{lem:S3}-\ref{lem:S4}, replacing inequalities \eqref{ineq1} and \eqref{ineq2} with \eqref{ineq5} and \eqref{ineq6}.
     The proof of Lemma \ref{lem:J34} is completed.
\end{proof}

\begin{lemma}[Estimates on $J_5$]\label{lem:J5}
  For any $\alpha\in\mathbb Z_+^3$ with $\abs\alpha\leq 3,$ it holds that
\begin{equation*}%\label{est:J5}
   J_5= \frac12   \sum_{\stackrel{\beta \in \mathbb Z_+^2}{\beta_1\ge 1}} r^{\alpha_1} \frac{ \beta_1}{t} M_{\rho,\beta}^2\norm{D_\lambda^\beta\partial^{\alpha} u}_{L^2}^2 \leq  C\lambda^2\abs{\nabla_{\rm h}u}_{ X_{\rho,\lambda}}^2,
\end{equation*}
recalling $\abs{\cdot}_{ X_{\rho,\lambda}}$ is given in Definition \ref{def:normx2}.
\end{lemma}

\begin{proof}
For $\tilde\beta=\beta-(1,0),$ we use \eqref{dlambda}
 and the fact  $\beta_1\geq 1$ to obtain
\begin{equation*}
	D_{\lambda}^{\beta}=\lambda t^{\frac12} D_{\lambda}^{\tilde \beta} \partial_2,
\end{equation*}
and moreover, in view of \eqref{betatilde},
\begin{equation*}
	  \frac{M_{\rho,\beta}^2}{M_{\rho,\tilde\beta}^2}\leq \frac{C}{\abs\beta^{2\delta}}.
\end{equation*}
As a result,
\begin{equation*}
J_5 \leq C \lambda^2 \sum_{\stackrel{\beta \in \mathbb Z_+^2}{\beta_1\ge 1}}r^{\alpha_1}\frac{\beta_1}{\abs\beta^{2\delta}}M_{\rho,\tilde\beta}^2\norm{ D_{\lambda}^{\tilde\beta} \partial^{ \alpha }\partial_2 u}_{L^2}^2
\leq  C\lambda^2 \abs{\nabla_{\rm h}u}_{ X_{\rho,\lambda}}^2.
\end{equation*}
This completes the proof of Lemma \ref{lem:J5}.
\end{proof}

\begin{proof}[Completing the proof of Proposition \ref{prop:x2}]
Substituting the estimates in Lemmas \ref{lem:J1}-\ref{lem:J5} into \eqref{est:x2} yields, for any $\alpha\in\mathbb Z_+^3$ with $\abs\alpha\leq 3,$
\begin{equation}\label{mrbr}
\begin{aligned}
& \frac{1}{2}\frac{d}{dt} \sum_{\stackrel{\beta \in \mathbb Z_+^2}{\beta_1\ge 1}} r^{\alpha_1}M_{\rho,\beta}^2\norm{D_\lambda^\beta\partial^{\alpha} u}_{L^2}^2+ \sum_{\stackrel{\beta \in \mathbb Z_+^2}{\beta_1\ge 1}}r^{\alpha_1}M_{\rho,\beta}^2\norm{D_\lambda^\beta\partial^{\alpha} \nabla_{\rm h} u}_{L^2}^2\\
 &\qquad -\frac{\rho^{'}}{\rho}  \sum_{\stackrel{\beta \in \mathbb Z_+^2}{\beta_1\ge 1}}r^{\alpha_1}(\abs{\beta }+1)M_{\rho,\beta}^2\norm{D_\lambda^\beta\partial^{\alpha} u}_{L^2}^2 \\
&\leq  Cr^{-\frac{9}{2}}\abs{u}_{ X_{\rho,\lambda}}\abs{u}_{ Y_{\rho,\lambda}}^2
    +C\Big( \lambda^2 t r^{-\frac{9}{2}}\abs{u}_{ X_{\rho,\lambda}}+r^{-\frac92}\abs{u}_{ X_{\rho,\lambda}} +r^{\frac{1}{4}}+\lambda^2\Big)\abs{\nabla_{\rm h}u}_{ X_{\rho,\lambda}}^2.
\end{aligned}
\end{equation}
Using \eqref{prpalge}, we extend the validity of \eqref{mrbr} to general indices $(j,\alpha)\in\mathbb{Z}_+\times\mathbb{Z}_+^3$ satisfying $j+|\alpha|\leq 3$:
\begin{align*}
& \frac{1}{2}\frac{d}{dt} \sum_{\stackrel{\beta \in \mathbb Z_+^2}{\beta_1\ge 1}} r^{\alpha_1}M_{\rho,\beta}^2\norm{D_\lambda^\beta\partial_t^j\partial^{\alpha} u}_{L^2}^2+ \sum_{\stackrel{\beta \in \mathbb Z_+^2}{\beta_1\ge 1}}r^{\alpha_1}M_{\rho,\beta}^2\norm{D_\lambda^\beta\partial_t^j\partial^{\alpha} \nabla_{\rm h} u}_{L^2}^2\\
 &\qquad -\frac{\rho^{'}}{\rho}  \sum_{\stackrel{\beta \in \mathbb Z_+^2}{\beta_1\ge 1}}r^{\alpha_1}(\abs{\beta }+1)M_{\rho,\beta}^2\norm{D_\lambda^\beta\partial_t^j\partial^{\alpha} u}_{L^2}^2\\
&\leq  Cr^{-\frac{9}{2}}\abs{u}_{ X_{\rho,\lambda}}\abs{u}_{ Y_{\rho,\lambda}}^2
    +C\Big( \lambda^2t r^{-\frac{9}{2}}\abs{u}_{ X_{\rho,\lambda}}+r^{-\frac92}\abs{u}_{ X_{\rho,\lambda}} +r^{\frac{1}{4}}+\lambda^2\Big)\abs{\nabla_{\rm h}u}_{ X_{\rho,\lambda}}^2.
\end{align*}
This, with estimate \eqref{ces}  and Remark \ref{rmk:xyz}, yields
\begin{equation}\label{DDD}
\begin{aligned}
&\frac12 \frac{d}{dt}\abs{u}_{ X_{\rho,\lambda}}^2+ \abs{\nabla_{\rm h} u}_{ X_{\rho,\lambda}}^2-\frac{\rho'}{\rho}\abs{\nabla_{\rm h} u}_{ Y_{\rho,\lambda}}^2  \\
&\leq Cr^{-\frac{9}{2}}\abs{u}_{ X_{\rho,\lambda}}\abs{u}_{ Y_{\rho,\lambda}}^2
    +C\Big(\lambda^2t r^{-\frac{9}{2}}\abs{u}_{ X_{\rho,\lambda}}+r^{-\frac92}\abs{u}_{ X_{\rho,\lambda}} +r^{\frac{1}{4}}+\lambda^2\Big)\abs{\nabla_{\rm h}u}_{ X_{\rho,\lambda}}^2.
    \end{aligned}
\end{equation}
We now claim the following implication:
\begin{equation}\label{sht}
\int_0^{+\infty} \abs{\nabla_{\rm h} u}_{ X_{\rho,\lambda}}^2 dt <+\infty\Longrightarrow
    \lim_{t\to0} \abs{u}_{ X_{\rho,\lambda}}^2= \lim_{t\to0} \abs{u}_{X_{\rho}}^2,
\end{equation}
whose proof will be provided subsequently.
 Combining    \eqref{DDD} with  implication \eqref{sht}, we use the bootstrap argument after \eqref{ces}
to conclude that,  for sufficiently small $\lambda$,
\begin{equation*}
\forall\ t\ge 0,\quad e^{\frac t2}\abs{u(t)}_{ X_{\rho,\lambda}}^2+\int^t_0 e^{\frac s2}\abs{\nabla_{\rm h}u(s)}_{ X_{\rho,\lambda}}^2ds\leq  \varepsilon_0^2.
\end{equation*}
This establishes \eqref{pri:ret2}, leaving only \eqref{sht} to be proven.
 In what follows  we let $(j,\alpha)\in\mathbb Z_+\times\mathbb Z_+^3$   be any given multi-index with $j+\abs\alpha\leq 3.$
For any $\beta\in\mathbb Z_+^2$ with $\beta_1\geq 1$, we use \eqref{dlambda} to get, recalling $\tilde\beta=\beta-(1,0),$
\begin{align*}
	&\int_0^{+\infty} t^{-1} \sum_{\stackrel{ \beta\in\mathbb Z_+^2}{\beta_1\geq 1} } r^{\alpha_1}M_{\rho,\beta}^2 \norm{ D_\lambda^{\beta}\partial_t^j\partial^{ \alpha }u}^2_{L^2} dt\\
    &=\int_0^{+\infty} t^{-1}\sum_{\stackrel{ \beta\in\mathbb Z_+^2}{\beta_1\geq 1} } r^{\alpha_1}M_{\rho,\beta}^2 \big(\lambda^2 t\norm{ D_\lambda^{\tilde \beta}\partial_t^j\partial^{ \alpha }\partial_2u}^2_{L^2} \big)dt \\
	&\leq C \int_0^{+\infty}   \sum_{\stackrel{ \beta\in\mathbb Z_+^2}{\beta_1\geq 1} } r^{\alpha_1}M_{\rho,\tilde \beta}^2 \norm{ D_\lambda^{\tilde \beta}\partial_t^j\partial^{ \alpha }\partial_2u}^2_{L^2}dt\leq C\int^{+\infty}_0\abs{\nabla_{\rm h} u}_{ X_{\rho,\lambda}}^2dt<+\infty,
\end{align*}
the last line using \eqref{betatilde} and $0<\lambda<1$.
This, with the continuity of
\begin{equation*}
 t\mapsto \sum_{\stackrel{ \beta\in\mathbb Z_+^2}{\beta_1\geq 1} } r^{\alpha_1}M_{\rho,\beta}^2 \norm{D_\lambda^{\beta}\partial_t^j\partial^{ \alpha } u}^2_{L^2},
\end{equation*}
implies that
\begin{equation}\label{shotbeha}
\forall\ (j,\alpha)\in\mathbb Z_+\times\mathbb Z_+^3\  \textrm{ with }\  j+\abs \alpha\leq 3,\quad		\lim_{t\rightarrow 0}\sum_{\stackrel{ \beta\in\mathbb Z_+^2}{\beta_1\geq 1} } r^{\alpha_1}M_{\rho,\beta}^2 \norm{ D_\lambda^{\beta}\partial_t^j\partial^{ \alpha }u}^2_{L^2}=0.
\end{equation}
Thus implication \eqref{sht} follows by observing that
\begin{equation*}
	 \lim_{t\to0} \abs{u}_{ X_{\rho,\lambda}}^2= \lim_{t\to0} \abs{u}_{X_{\rho}}^2+\lim_{t\rightarrow 0} \sum_{\stackrel{(j,\alpha)\in\mathbb Z_+\times\mathbb Z_+^3}{ j+\abs \alpha\leq 3}} \, \sum_{\stackrel{ \beta\in\mathbb Z_+^2}{\beta_1\geq 1} } r^{\alpha_1}M_{\rho,\beta}^2  \norm{ D_\lambda^{\beta}\partial_t^j\partial^{ \alpha }u}^2_{L^2}
\end{equation*}
due to \eqref{xandtx}. The proof of Proposition \ref{prop:x2} is completed.
\end{proof}

 \section{Proof of Theorem \ref{thm:smoothing}: smoothing effect for the all  variables}\label{sec:x1}

 In this section, we establish the Gevrey smoothing effect of the solution $u$ (obtained in Theorem \ref{thm:main1}) with respect to the variable $x_1$, thereby completing the proof of Theorem \ref{thm:smoothing}. Prior to presenting the main result (Proposition \ref{prop:x1} below), we combine  the argument  from the preceding two sections to derive the following global estimate in the variables $t$ and $x_2$.

 \begin{proposition}
 	 \label{prop:tx2}
Let $u\in L^{\infty}([0,+\infty[, G_{\rho,\sigma,3})$ be the solution constructed in Theorem \ref{thm:main1}, which satisfies the properties stated in  Propositions  \ref{prop:t} and   \ref{prop:x2}. Then, by possibly shrinking the parameters $r,\lambda$ and $\varepsilon_0$ from Theorem  \ref{prop:wellposedness} and  Proposition \ref{prop:t}, we obtain
\begin{equation}\label{alfl}
	\forall\ t\geq0,\quad e^{\frac t2}\abs{u(t)}_{\mathcal X_{\rho,\lambda}}^2+\int^t_0e^{\frac s2}\abs{\nabla_{\rm h}u(s)}_{\mathcal X_{\rho,\lambda}}^2ds\leq \varepsilon_0^2,
\end{equation}
where the norm $\abs{\cdot}_{\mathcal X_{\rho,\lambda}}$ is defined by
 \begin{equation}\label{lasx}
 	\abs{g}_{ \mathcal X_{\rho,\lambda}}^2\stackrel{\rm def}{=} \sum_{\stackrel{(j,\alpha)\in\mathbb Z_+\times\mathbb Z_+^3}{ j+\abs \alpha\leq 3}} \, \sum_{(k,n,m)\in\mathbb Z_+^3}   r^{ \alpha_1}\Big( \widetilde M_{\rho,(k,n,m)}   \lambda^{k+n} t^{k+\frac{n}{2}}  \norm{\partial_t^k  \partial_2^n\partial_3^m  \partial_t^j\partial^{ \alpha }g}_{L^2} \Big)^2
 \end{equation}
 with
 \begin{equation*}
 \widetilde	M_{\rho,(k,n,m)}\stackrel{\rm def}{=}\frac{\rho^{k+n+m+1}(k+n+m+1)^{6+2\sigma}}{[(k+m)!]^{\sigma-\delta} [(k+n+m)!]^\delta}.
 \end{equation*}
 \end{proposition}

 \begin{proof}
 	 The proof is  analogous  to that of Propositions \ref{prop:t} and \ref{prop:x2}, so we omit it for brevity.
 \end{proof}

 As an immediate consequence of Proposition \ref{prop:tx2}, we have the following anisotropic Gevrey-class regularity which is global in time and valid
up to the boundary.

\begin{corollary}\label{corpur}
	Under the hypothesis of Proposition \ref{prop:tx2}, we can find a constant $A\geq1,$ depending only on $\lambda,r$ and $\rho_0,$ such that
	for any $j\in\mathbb Z_+$ and any $\alpha=(\alpha_1,\alpha_2,\alpha_3)\in\mathbb Z_+^3$ with $\alpha_1=0$, it holds that
	\begin{align*}
		 \sup_{t\geq 0} e^{\frac t4}  t^{j+\frac{\alpha_2}{2}}\norm{ \partial^{\alpha}  \partial_t^{j}  u}_{H^3} 		 \leq
    \frac{\eps_0 A^{j+\alpha_2+\alpha_3+1} [(j+ \alpha_3)!]^{\sigma-\delta} [( j+\alpha_2+\alpha_3)!]^{\delta}}{(j+\alpha_2+\alpha_3+1)^{6+2\sigma}}.
		\end{align*}	
\end{corollary}

\begin{proof}
	   For $\alpha=(0,\alpha_2,\alpha_3)$,  we have  $\partial^\alpha=\partial_2^{\alpha_2}\partial_3^{\alpha_3}$ and thus  use \eqref{lasx} to get that
	\begin{multline*}
	 e^{\frac t4} t^{j+\frac{\alpha_2}{2}}\norm{ \partial^{\alpha}  \partial_t^{j}  u}_{H^3}=e^{\frac t4}  t^{j+\frac{\alpha_2}{2}}\norm{ \partial_t^{j}\partial_2^{\alpha_2}\partial_3^{\alpha_3}u}_{H^3}\\
	 \leq   e^{\frac t4}\abs{u}_{\mathcal X_{\rho,\lambda}}	r^{-\frac 32} \lambda^{-(j+\alpha_2)}   \frac{[(j+\alpha_3)!]^{\sigma-\delta} [(j+\alpha_2+\alpha_3)!]^\delta}{\rho^{j+\alpha_2+\alpha_3+1}(j+\alpha_2+\alpha_3+1)^{6+2\sigma}}.
	\end{multline*}
	  Then assertion of Corollary \ref{corpur} follows from \eqref{alfl} if we choose  $A$ large enough such that
	\begin{equation*}
		A\geq \frac{1}{\lambda\rho r^\frac{3}{2}}.
	\end{equation*}
	The proof is thus completed.
 \end{proof}

 \begin{proposition}\label{prop:x1}
 	Suppose the hypothesis of Proposition \ref{prop:tx2} is fulfilled and let $A\geq 1$ be the constant constructed in Corollary \ref{corpur}.  Then there exists  a constant  $B\geq A$, depending only on $\sigma$, the constant  $A$ and the Sobolev embedding constants, such that for any $T\geq 1$ and for any $j\in\mathbb Z_+$ and any $\alpha=(\alpha_1,\alpha_2,\alpha_3) \in\mathbb Z_+^3$, it holds that
	\begin{multline*}
	\sup_{t\leq T} e^{\frac t4}  t^{j+\frac{\alpha_1+\alpha_2}{2}}\norm{ \partial^{\alpha} \partial_t^{j}  u}_{H^2} \\
    \leq
    \frac{ \eps_0 A^{j+\alpha_2+\alpha_3+1}T^{\frac{\alpha_1}{2}}B^{\alpha_1} [(j+\alpha_1+\alpha_3)!]^{\sigma-\delta} [(j+\abs\alpha)!]^{\delta}}{(j+\abs\alpha+1)^{6+2\sigma}}.
  \end{multline*}
 \end{proposition}

The rest of this section is devoted to the proof of Proposition \ref{prop:x1}, and  the key part in the proof is the treatment of the pressure. To address this,  we introduce the following  space-time derivatives
\begin{equation}\label{degamma}
	\mathcal D^\alpha\stackrel{\rm def}{=}t^{\frac{\alpha_1+\alpha_2}{2}}\partial_1^{\alpha_1}\partial_2^{\alpha_2}\partial_3^{\alpha_3} \ \textrm{ for }\ \alpha\in\mathbb Z_+^3,
\end{equation}
and state the estimate for the pressure as follows.
\begin{lemma}[Estimate on the pressure]\label{lem-pre}
For any $(j,\beta)\in\mathbb Z_+\times\mathbb Z_+^3$,  it holds that
 \begin{equation}\label{crues}
 \begin{aligned}
 	&\norm{\mathcal{D}^{\beta} \partial_t^{j}  \nabla p}_{H^2}
 	\leq C\norm{\mathcal{D}^{\beta} \partial_t^{j}\big (u\cdot\nabla) u\big)}_{H^2}+C\norm{\mathcal{D}^{\beta} \partial_t^{j+1}  u}_{H^2}+\frac12\norm{\mathcal{D}^{\beta} \partial_t^{j}\partial_1^2  u}_{H^2} \\
    &\quad+Ct^{\frac{1}{2}}\Big( \norm{\mathcal{D}^{\tilde \beta} \partial_t^{j+1}u}_{H^2}+ \norm{\mathcal{D}^{\tilde \beta} \partial_t^{j+1}\partial_2u}_{H^2}+\norm{\mathcal{D}^{\tilde \beta} \partial_t^{j+1}\partial
     _3u}_{H^2}\Big)\\
 	&\quad +C\Big(   \norm{\mathcal{D}^{\beta} \partial_t^{j}\nabla  u}_{H^2}+  \norm{\mathcal{D}^{\beta} \partial_t^{j}\partial_1\partial_2  u}_{H^2}+ \norm{\mathcal{D}^{\beta} \partial_t^{j}\partial_1\partial_3  u}_{H^2}
     +  \norm{\mathcal{D}^{\beta} \partial_t^{j}\partial_2^2  u}_{H^2}\Big),
 	\end{aligned}
 \end{equation}
 recalling $\tilde\beta=\beta-(1,0,0).$
 \end{lemma}

 \begin{remark}\label{remlocal}

 Note that the term $\norm{\mathcal{D}^{\beta} \partial_t^{j}\partial_1\partial_3 u}_{H^2}$, along with the terms in the second line on the right-hand side of \eqref{crues}, prevents  us from achieving   a global-in-time estimate for the Gevrey radius. Furthermore, the term $\norm{\mathcal{D}^{\beta} \partial_t^{j}\partial_1\partial_3 u}_{H^2}$ also impedes any improvement in the Gevrey smoothing effect in the $x_1$ direction-specifically, it prevents enhancing the Gevrey index from $\sigma$ to $\delta$.
 \end{remark}

\begin{proof}[Proof of Lemma \ref{lem-pre}]
   	Observe $p$ satisfies  the elliptic equation $-\Delta p=\divv \big( (u\cdot\nabla ) u \big )$. Then
    for any $(j,\beta)\in\mathbb Z_+\times\mathbb Z_+^3$ and any $\gamma\in\mathbb Z_+^3$ with $\abs{\gamma}\leq 2$,
 \begin{equation*}%\label{dtxb}
  -\partial^\gamma\mathcal{D}^{\beta}\partial_t^j \Delta p=\partial^\gamma\mathcal{D}^{\beta}\partial_t^j  \divv \big( (u\cdot\nabla ) u \big ).
   \end{equation*}
   Then repeating the argument in the derivation of \eqref{p+} gives
\begin{multline*}
		   \sum_{\abs\gamma\leq 2} \norm{ \partial^\gamma \mathcal{D}^{\beta} \partial_t^{j}  \nabla p}_{L^2}^2
		 \leq  \norm{\mathcal{D}^{\beta} \partial_t^{j}\big (u\cdot\nabla) u\big)}_{H^2}\norm{\mathcal{D}^{\beta} \partial_t^{j}\nabla p}_{H^2} \\
		 + \sum_{\abs\gamma\leq 2} \bigg|\int_{\mathbb R^2}   (\partial^{\gamma}\mathcal{D}^{\beta} \partial_t^{j}\partial_tu_1)(\partial^\gamma \mathcal{D}^{\beta} \partial_t^{j}  p)\big|^{x_1=1}_{x_1=0} \  dx_2dx_3\bigg|\\
		  + \sum_{\abs\gamma\leq 2} \bigg|\int_{\mathbb R^2}   (\partial^{\gamma}\mathcal{D}^{\beta} \partial_t^{j}\Delta_{\rm h} u_1)(\partial^\gamma \mathcal{D}^{\beta} \partial_t^{j}  p)\big|^{x_1=1}_{x_1=0} \  dx_2dx_3\bigg|,
	\end{multline*}
	and thus,
	\begin{equation}\label{pres}
	\begin{aligned}
		&  \norm{\mathcal{D}^{\beta} \partial_t^{j}  \nabla p}_{H^2}^2=  \sum_{\abs\gamma\leq 2} \norm{ \partial^\gamma \mathcal{D}^{\beta} \partial_t^{j}  \nabla p}_{L^2}^2 \\
		 &\leq  \norm{\mathcal{D}^{\beta} \partial_t^{j}\big (u\cdot\nabla) u\big)}_{H^2}^2  		 +2 \sum_{\abs\gamma\leq 2} \bigg|\int_{\mathbb R^2}   (\partial^{\gamma}\mathcal{D}^{\beta} \partial_t^{j}\partial_tu_1)(\partial^\gamma \mathcal{D}^{\beta} \partial_t^{j}  p)\big|^{x_1=1}_{x_1=0} \  dx_2dx_3\bigg|\\
		  &\quad+2 \sum_{\abs\gamma\leq 2} \bigg|\int_{\mathbb R^2}   (\partial^{\gamma}\mathcal{D}^{\beta} \partial_t^{j}\Delta_{\rm h} u_1)(\partial^\gamma \mathcal{D}^{\beta} \partial_t^{j}  p)\big|^{x_1=1}_{x_1=0} \  dx_2dx_3\bigg|.
		  \end{aligned}
	\end{equation}
	We will proceed to estimate the last two terms on the right-hand side of \eqref{pres}.
	
For the second term on the right-hand side of \eqref{pres}, we claim that
\begin{equation}\label{eop}
\begin{aligned}
	&\sum_{\abs\gamma\leq 2}\bigg|\int_{\mathbb R^2}   (\partial^{\gamma}\mathcal{D}^{\beta} \partial_t^{j}\partial_tu_1)(\partial^\gamma \mathcal{D}^{\beta} \partial_t^{j}  p)\big|^{x_1=1}_{x_1=0} \  dx_2dx_3\bigg|\\
	 &\leq  Ct^\frac{1}{2}\Big( \norm{\mathcal{D}^{\tilde \beta} \partial_t^{j+1}u}_{H^2}+ \norm{\mathcal{D}^{\tilde \beta} \partial_t^{j+1}\partial_2u}_{H^2}+\norm{\mathcal{D}^{\tilde \beta} \partial_t^{j+1}\partial
     _3u}_{H^2}\Big)\norm{\mathcal{D}^{\beta} \partial_t^{j}\nabla p}_{H^2}\\
    &\quad +C\norm{\mathcal{D}^{\beta} \partial_t^{j+1}  u}_{H^2}\norm{\mathcal{D}^{\beta} \partial_t^{j}\nabla p}_{H^2},
     \end{aligned}
\end{equation}
recalling $\tilde\beta=\beta-(1,0,0)=(\beta_1-1,\beta_2,\beta_3).$
It suffices to prove \eqref{eop} for $  \gamma_1+\beta_1\geq 1$, since for  $ \gamma_1+\beta_1=0$   it follows from the fact
$
   u_1|_{x_1=0,1}=0$ that
\begin{equation*}
	  \bigg|\int_{\mathbb R^2}   (\partial^{\gamma}\mathcal{D}^{\beta} \partial_t^{j}\partial_tu_1)(\partial^\gamma \mathcal{D}^{\beta} \partial_t^{j}  p)\big|^{x_1=1}_{x_1=0} \  dx_2dx_3\bigg|= 0.
\end{equation*}
 Now we consider the case of $\beta_1+\gamma_1\geq1$.  If $\beta_1\geq 1$, then using the fact
 \begin{equation}\label{rela}
 	\mathcal D^\beta=t^{\frac{1}{2}}\mathcal D^{\tilde\beta}\partial_1 \ \textrm{ for }\ \tilde\beta=\beta-(1,0,0)
 \end{equation}
 as well as
$  \partial_1u_1 =- \partial_2u_2- \partial_3u_3
$,  we have, recalling $\abs\gamma\leq 2,$
 \begin{equation*}
 \begin{aligned}
& \bigg|\int_{\mathbb R^2}   (\partial^{\gamma}\mathcal{D}^{\beta} \partial_t^{j}\partial_tu_1)(\partial^\gamma \mathcal{D}^{\beta} \partial_t^{j}  p)\big|^{x_1=1}_{x_1=0} \  dx_2dx_3\bigg| \\
&\leq  2t^{\frac{1}{2}}\sup_{x_1\in[0,1]}  \left|\int_{\mathbb{R}^2}(\partial^{\gamma}\mathcal{D}^{\tilde \beta} \partial_t^{j+1} \partial_2u_2)(\partial^\gamma \mathcal{D}^{\beta} \partial_t^{j} p)dx_2dx_3\right|\\
&\quad+2t^{\frac{1}{2}} \sup_{x_1\in[0,1]}  \left|\int_{\mathbb{R}^2}(\partial^{\gamma}\mathcal{D}^{\tilde \beta} \partial_t^{j+1} \partial_3u_3)(\partial^\gamma \mathcal{D}^{\beta} \partial_t^{j} p)dx_2dx_3\right|\\
&\leq Ct^{\frac{1}{2}}\Big(\norm{\mathcal{D}^{\tilde \beta} \partial_t^{j+1}u}_{H^2}+\norm{\mathcal{D}^{\tilde \beta} \partial_t^{j+1}\nabla u}_{H^2}\Big)\norm{\mathcal{D}^{\beta} \partial_t^{j}\nabla p}_{H^2},
 \end{aligned}
 \end{equation*}
 the last inequality following the same argument as that after \eqref{q1q2}. Similarly, if $\abs\gamma\leq 2$ with $\gamma_1\geq 1$, then we have, recalling $\tilde\gamma=\gamma-(1,0,0),$
 \begin{equation*}
 	 \begin{aligned}
& \bigg|\int_{\mathbb R^2}   (\partial^{\gamma}\mathcal{D}^{\beta} \partial_t^{j}\partial_tu_1)(\partial^\gamma \mathcal{D}^{\beta} \partial_t^{j}  p)\big|^{x_1=1}_{x_1=0} \  dx_2dx_3\bigg| \\
&\leq  2 \sup_{x_1\in[0,1]}  \left|\int_{\mathbb{R}^2}(\partial^{\tilde\gamma}\mathcal{D}^{ \beta} \partial_t^{j+1} \partial_2u_2)(\partial^\gamma \mathcal{D}^{\beta} \partial_t^{j} p)dx_2dx_3\right|\\
&\quad+2  \sup_{x_1\in[0,1]}  \left|\int_{\mathbb{R}^2}(\partial^{\tilde\gamma}\mathcal{D}^{ \beta} \partial_t^{j+1} \partial_3u_3)(\partial^\gamma \mathcal{D}^{\beta} \partial_t^{j} p)dx_2dx_3\right|\\
&\leq C \Big(\norm{\mathcal{D}^{  \beta} \partial_t^{j+1}u}_{H^1}+\norm{\mathcal{D}^{ \beta} \partial_t^{j+1}\nabla u}_{H^1}\Big)\norm{\mathcal{D}^{\beta} \partial_t^{j}\nabla p}_{H^2}\\
&\leq C \norm{\mathcal{D}^{\beta} \partial_t^{j+1}  u}_{H^2} \norm{\mathcal{D}^{\beta} \partial_t^{j}\nabla p}_{H^2}.
 \end{aligned}
 \end{equation*}
 Combining the above estimates and observing
 \begin{equation*}
    t^{\frac{1}{2}} \norm{\mathcal{D}^{\tilde \beta} \partial_t^{j+1}\partial_1 u}_{H^2} =  \norm{\mathcal{D}^{\beta} \partial_t^{j+1}  u}_{H^2}
 \end{equation*}
by relation \eqref{rela}, we get the desired  conclusion \eqref{eop}.

 It remains to estimate the last term on the right-hand side of \eqref{pres}, and the argument is  similar to that above. In fact,  we use the argument after \eqref{r1r2} to obtain that
 \begin{equation*}
 	\begin{aligned}
 	   & \sum_{\abs\gamma\leq 2} \bigg|\int_{\mathbb R^2}   (\partial^{\gamma}\mathcal{D}^{\beta} \partial_t^{j}\partial_1^2 u_1)(\partial^\gamma \mathcal{D}^{\beta} \partial_t^{j}  p)\big|^{x_1=1}_{x_1=0} \  dx_2dx_3\bigg|\\
 	   & \leq  2 \sum_{\abs\gamma\leq 2} \sup_{x_1\in[0,1]}	 \bigg|\int_{\mathbb R^2}   (\partial^{\gamma}\mathcal{D}^{\beta} \partial_t^{j}\partial_1  \partial_2u_2)(\partial^\gamma \mathcal{D}^{\beta} \partial_t^{j}  p)  dx_2dx_3\bigg|\\
 	   &\quad +2\sum_{\abs\gamma\leq 2} \sup_{x_1\in[0,1]}  \bigg|\int_{\mathbb R^2}   (\partial^{\gamma}\mathcal{D}^{\beta} \partial_t^{j}\partial_1\partial_3u_3)(\partial^\gamma \mathcal{D}^{\beta} \partial_t^{j}  p)   dx_2dx_3\bigg|\\
 	   & \leq C\Big( \norm{\mathcal{D}^{\beta} \partial_t^{j}\partial_1 u}_{H^2}^\frac{1}{2}+  \norm{\mathcal{D}^{\beta} \partial_t^{j}\partial_1\partial_2 u}_{H^2}^\frac{1}{2}+  \norm{\mathcal{D}^{\beta} \partial_t^{j}\partial_1\partial_3 u}_{H^2}^\frac{1}{2}    \Big)\\
       &\qquad\times \Big( \norm{\mathcal{D}^{\beta} \partial_t^{j}\partial_1 u}_{H^2}^\frac{1}{2}+  \norm{\mathcal{D}^{\beta} \partial_t^{j}\partial_1^2 u}_{H^2}^\frac{1}{2}   \Big)\norm{\mathcal{D}^{\beta} \partial_t^{j}\nabla p}_{H^2}\\
 	   &\leq   C\Big( \norm{\mathcal{D}^{\beta} \partial_t^{j}\partial_1 u}_{H^2}+  \norm{\mathcal{D}^{\beta} \partial_t^{j}\partial_1\partial_2 u}_{H^2}+  \norm{\mathcal{D}^{\beta} \partial_t^{j}\partial_1\partial_3 u}_{H^2}    \Big)\norm{\mathcal{D}^{\beta} \partial_t^{j}\nabla p}_{H^2}\\
       &\quad+\frac12\norm{\mathcal{D}^{\beta} \partial_t^{j}\partial_1^2 u}_{H^2}\norm{\mathcal{D}^{\beta} \partial_t^{j}\nabla p}_{H^2}.	
       \end{aligned}
 \end{equation*}
 Similarly,
 \begin{equation*}
 	\begin{aligned}
 	   & \sum_{\abs\gamma\leq 2} \bigg|\int_{\mathbb R^2}   (\partial^{\gamma}\mathcal{D}^{\beta} \partial_t^{j}\partial_2^2 u_1)(\partial^\gamma \mathcal{D}^{\beta} \partial_t^{j}  p)\big|^{x_1=1}_{x_1=0} \  dx_2dx_3\bigg|\\
 	& \leq C\Big( \norm{\mathcal{D}^{\beta} \partial_t^{j}\partial_2 u}_{H^2}^\frac{1}{2}+  \norm{\mathcal{D}^{\beta} \partial_t^{j}\partial_2^2 u}_{H^2}^\frac{1}{2}\Big)\\
       &\qquad\times \Big( \norm{\mathcal{D}^{\beta} \partial_t^{j}\partial_2 u}_{H^2}^\frac{1}{2}+  \norm{\mathcal{D}^{\beta} \partial_t^{j}\partial_1\partial_2 u}_{H^2}^\frac{1}{2}   \Big)\norm{\mathcal{D}^{\beta} \partial_t^{j}\nabla p}_{H^2}\\
 	   & \leq C \Big(\norm{\mathcal{D}^{\beta} \partial_t^{j}\partial_2  u}_{H^2} + \norm{\mathcal{D}^{\beta} \partial_t^{j}\partial_1\partial_2  u}_{H^2} + \norm{\mathcal{D}^{\beta} \partial_t^{j}\partial_2^2  u}_{H^2}  \Big)\norm{\mathcal{D}^{\beta} \partial_t^{j}\nabla p}_{H^2}.
 \end{aligned}
 \end{equation*}
 As a result, combining the above estimates yields that  \begin{multline*}
 	 \sum_{\abs\gamma\leq 2} \bigg|\int_{\mathbb R^2}   (\partial^{\gamma}\mathcal{D}^{\beta} \partial_t^{j}\Delta_{\rm h} u_1)(\partial^\gamma \mathcal{D}^{\beta} \partial_t^{j}  p)\big|^{x_1=1}_{x_1=0} \  dx_2dx_3\bigg|\\
 	 \leq \Big(\frac12 \norm{\mathcal{D}^{\beta} \partial_t^{j}\partial_1^2  u}_{H^2}+C \norm{\mathcal{D}^{\beta} \partial_t^{j}\partial_1\partial_2  u}_{H^2}+C \norm{\mathcal{D}^{\beta} \partial_t^{j}\partial_1\partial_3  u}_{H^2}\\
     +C \norm{\mathcal{D}^{\beta} \partial_t^{j}\partial_2^2  u}_{H^2}+C   \norm{\mathcal{D}^{\beta} \partial_t^{j}\nabla  u}_{H^2}    \Big)\norm{\mathcal{D}^{\beta} \partial_t^{j}\nabla p}_{H^2}.
 \end{multline*}
 Substituting the estimate above and \eqref{eop} into \eqref{pres} gives the desired  conclusion of Lemma \ref{lem-pre}. The proof is thus completed.
   \end{proof}

\begin{proof}[Proof of Proposition \ref{prop:x1}]
	We  apply  induction on $\alpha_1$ to prove that  	for any $T\geq 1$ and for any $j\in\mathbb Z_+$ and any $\alpha=(\alpha_1,\alpha_2,\alpha_3) \in\mathbb Z_+^3$, it holds that
	\begin{multline}\label{fes}
\sup_{t\leq T} e^{\frac t4}  t^{j+\frac{\alpha_1+\alpha_2}{2}}\norm{ \partial^{\alpha} \partial_t^{j}  u}_{H^2}\\
\leq
    \frac{ \eps_0 A^{j+\alpha_2+\alpha_3+1}T^{\frac{\alpha_1}{2}}B^{\alpha_1} [(j+\alpha_1+\alpha_3)!]^{\sigma-\delta}[ (j+\abs\alpha)!]^{\delta}}{(j+\abs\alpha+1)^{6+2\sigma}}.
  \end{multline}
  The validity of \eqref{fes} for $\alpha_1=0$ just follows from Corollary  \ref{corpur}.  Moreover,
for all $T\geq 1,j\geq 0$ and any mult-index $\alpha\in\mathbb Z_+^3$ with $\alpha_1= 1$, we have
	\begin{align*}
		&\sup_{t\leq T} e^{\frac t4}  t^{j+\frac{\alpha_1+\alpha_2}{2}}\norm{ \partial^{\alpha} \partial_t^{j}  u}_{H^2}\leq
		T^{\frac{\alpha_1}{2}}\sup_{t\leq T} e^{\frac t4}  t^{j+\frac{\alpha_2}{2}}\norm{ \partial_2^{\alpha_2}\partial_3^{\alpha_3} \partial_t^{j}  u}_{H^3}\\
		& \leq T^{\frac{\alpha_1}{2}}
    \frac{ \eps_0A^{j+\alpha_2+\alpha_3+1} [(j+ \alpha_3)!]^{\sigma-\delta} [( j+\alpha_2+\alpha_3)!]^{\delta}}{(j+\alpha_2+\alpha_3+1)^{6+2\sigma}}\\
		& \leq 2^{6+2\sigma}T^{\frac{\alpha_1}{2}}
    \frac{ \eps_0 A^{j+\alpha_2+\alpha_3+1} [(j+\alpha_1+ \alpha_3)!]^{\sigma-\delta} [( j+\abs\alpha)!]^{\delta}}{(j+\abs\alpha+1)^{6+2\sigma}},
		\end{align*}	
		the second inequality following from Corollary \ref{corpur} and the last one holding because of the fact that
		\begin{equation*}
			\frac{1}{(j+\alpha_2+\alpha_3+1)^{6+2\sigma}}\leq \frac{2^{6+2\sigma}}{(j+\abs\alpha+1)^{6+2\sigma}} \ \textrm{ for }\  \alpha_1= 1.
		\end{equation*}
Thus we have the validity of \eqref{fes} for all $T\geq 1, j\geq 0$ and any mult-index $\alpha\in\mathbb Z_+^3$ with $\alpha_1\leq 1,$ provided  $B\geq 2^{6+2\sigma}.$

Let $N\geq 2$ be a fixed integer in what follows. We assume that for all non-negative integers $j\in\mathbb Z_+$ and any $\gamma=(\gamma_1,\gamma_2,\gamma_2)\in\mathbb Z_+^3$ with $\gamma_1\leq N-1$, the following estimate holds:
\begin{multline}\label{inductiveass+++}
\sup_{t\leq  T} e^{\frac t4}  t^{j+\frac{\gamma_1+\gamma_2}{2}}\norm{  \partial^{\gamma} \partial_t^{j} u}_{H^2}\\
 \leq
    \frac{\eps_0 A^{j+\gamma_2+\gamma_3+1}T^{\frac{\gamma_1}{2}}B^{\gamma_1} [(j+\gamma_1+\gamma_3)!]^{\sigma-\delta} [(j+\abs\gamma)!]^{\delta}}{(j+\abs\gamma+1)^{6+2\sigma}}.
  \end{multline}
We shall prove by induction that \eqref{inductiveass+++} remains valid for  any $j\in\mathbb Z_+$ and any $\gamma\in\mathbb Z_+^3$ with $\gamma_1=N.$

{\it Step 1 (Notations).}  Fix an integer $j\geq 0$ and a multi-index $\alpha=(\alpha_1,\alpha_2,\alpha_3)\in\mathbb Z_+^3$ with $\alpha_1=N.$ Denote
\begin{equation}\label{debeta}
	\beta=(\beta_1,\beta_2,\beta_3)=\alpha-(2,0,0)=(\alpha_1-2,\alpha_2,\alpha_3),
\end{equation}
so that
\begin{equation}\label{sm1}
	\beta_1= N-2.
\end{equation}
 Recalling the operator  $\mathcal D^\alpha=t^{\frac{\alpha_1+\alpha_2}{2}}\partial_1^{\alpha_1}\partial_2^{\alpha_2}\partial_3^{\alpha_3} \ \textrm{ for }\ \alpha\in\mathbb Z_+^3$, we obtain from \eqref{debeta} the identity
\begin{equation}\label{ab}
	\mathcal D^{\alpha}=t \mathcal D^{\beta}\partial_1^2.
\end{equation}
  Under the above notations,  the inductive assumption \eqref{inductiveass+++} says  that the following estimate
 \begin{equation}\label{inductiveass}
\sup_{t\leq T} e^{\frac t4}  t^{j}\norm{\mathcal D^{\gamma} \partial_t^{j} u}_{H^2} \leq
    \frac{\eps_0 A^{j+\gamma_2+\gamma_3+1}T^{\frac{\gamma_1}{2}}B^{\gamma_1} [(j+\gamma_1+\gamma_3)!]^{\sigma-\delta} [(j+\abs\gamma)!]^{\delta}}{(j+\abs\gamma+1)^{6+2\sigma}}
  \end{equation}
  holds true for any $j\geq 0$ and  any $\gamma\in\mathbb Z_+^3$ with $\gamma_1\leq N-1.$

{\it Step 2.}  We  recall $\alpha\in\mathbb Z_+^3$   is an arbitrary multi-index   with $\alpha_1=N\geq 2$ and $\beta=\alpha-(2,0,0)$ is defined as in   \eqref{debeta}. In this step we will show that if the inductive hypothesis \eqref{inductiveass} holds, then
\begin{multline}\label{linpart}
		\sup_{t\leq T} e^{\frac t4}  t^{j}\norm{\mathcal D^{\alpha} \partial_t^{j}   u}_{H^2}\leq C\sup_{t\leq T} e^{\frac t4}  t^{j+1}  \norm{\mathcal D^{\beta} \partial_t^{j}  \big((u\cdot\nabla) u\big)}_{H^2}\\
	+ \frac{CA^2}{B} \frac{\eps_0 A^{ j+\alpha_2+\alpha_3+1}T^{\frac{\alpha_1}{2}}B^{\alpha_1} [( j+ \alpha_1+\alpha_3)!]^{\sigma-\delta} [(j +\abs\alpha)!]^{\delta}}{(j+\abs\alpha+1)^{6+2\sigma}}.
\end{multline}
To do so, we use \eqref{degamma} as well as \eqref{ab}   and observe $\partial_1^2u=\partial_tu-\partial_2^2u+(u\cdot \nabla) u+\nabla p;$ this gives
\begin{equation}\label{splt}
\begin{aligned}
	& e^{\frac t4}  t^{j+\frac{\alpha_1+\alpha_2}{2}}\norm{\partial^{\alpha} \partial_t^{j}   u}_{H^2} = e^{\frac t4}  t^{j}\norm{\mathcal D^{\alpha} \partial_t^{j}   u}_{H^2} =e^{\frac t4}  t^{j+1}\norm{\mathcal D^{\beta} \partial_t^{j}  \partial_1^2 u}_{H^2}\\
	&\leq e^{\frac t4}  t^{j+1}  \norm{\mathcal D^{\beta} \partial_t^{j+1}  u}_{H^2}+e^{\frac t4}  t^{j+1} \norm{\mathcal D^{\beta}\partial_2^2 \partial_t^{j}   u}_{H^2} \\
	&\qquad+e^{\frac t4}  t^{j+1}  \norm{\mathcal D^{\beta} \partial_t^{j}  \big((u\cdot\nabla) u\big)}_{H^2}+e^{\frac t4}  t^{j+1}  \norm{\mathcal D^{\beta} \partial_t^{j}  \nabla p}_{H^2}.
	\end{aligned}
\end{equation}
Moreover, it follows from  Lemma \ref{lem-pre} that \begin{equation*}
 \begin{aligned}
 	&e^{\frac t4}  t^{j+1}\norm{\mathcal{D}^{\beta} \partial_t^{j}  \nabla p}_{H^2}\leq Ce^{\frac t4}  t^{j+1}\norm{\mathcal{D}^{\beta} \partial_t^{j}\big (u\cdot\nabla) u\big)}_{H^2}+\frac12e^{\frac t4}  t^{j+1}\norm{\mathcal{D}^{\beta} \partial_t^{j}\partial_1^2  u}_{H^2} \\
    &\quad+Ce^{\frac t4}  t^{j+1+\frac{1}{2}}\Big( \norm{\mathcal{D}^{\tilde \beta} \partial_t^{j+1}u}_{H^2}+ \norm{\mathcal{D}^{\tilde \beta} \partial_t^{j+1}\partial_2u}_{H^2}+\norm{\mathcal{D}^{\tilde \beta} \partial_t^{j+1}\partial
     _3u}_{H^2}\Big)\\
 	&\quad +Ce^{\frac t4}  t^{j+1}(\norm{\mathcal{D}^{\beta} \partial_t^{j}\partial_1\partial_2  u}_{H^2}+ \norm{\mathcal{D}^{\beta} \partial_t^{j}\partial_1\partial_3  u}_{H^2}
     + \norm{\mathcal{D}^{\beta} \partial_t^{j}\partial_2^2  u}_{H^2})\\
     &\quad +Ce^{\frac t4}  t^{j+1}\Big(\norm{\mathcal{D}^{\beta} \partial_t^{j+1}  u}_{H^2}+\norm{\mathcal{D}^{\beta} \partial_t^{j}\nabla  u}_{H^2}\Big),
 	\end{aligned}
 \end{equation*}
 recalling $\tilde\beta=\beta-(1,0,0).$
Combining the estimate above and \eqref{splt} and observing the fact that
\begin{equation*}
	\frac12 e^{\frac t4}  t^{j+1} \norm{\mathcal{D}^{\beta} \partial_t^{j}\partial_1^2  u}_{H^2}=\frac12 e^{\frac t4}  t^{j}\norm{\mathcal D^{\alpha} \partial_t^{j}   u}_{H^2}
\end{equation*}
in view of \eqref{debeta}, we obtain that
\begin{equation}\label{abeta}
\begin{aligned}
	&e^{\frac t4}  t^{j}\norm{\mathcal D^{\alpha} \partial_t^{j}   u}_{H^2} \leq Ce^{\frac t4}  t^{j+1}  \norm{\mathcal D^{\beta} \partial_t^{j}  \big((u\cdot\nabla) u\big)}_{H^2}\\
    &\quad+Ce^{\frac t4}  t^{j+1+\frac{1}{2}}\Big( \norm{\mathcal{D}^{\tilde \beta} \partial_t^{j+1}u}_{H^2}+\norm{\mathcal{D}^{\tilde \beta} \partial_t^{j+1}\partial_2u}_{H^2}+\norm{\mathcal{D}^{\tilde \beta} \partial_t^{j+1}\partial
     _3u}_{H^2}\Big)\\
 	&\quad +Ce^{\frac t4}  t^{j+1}(\norm{\mathcal{D}^{\beta} \partial_t^{j}\partial_1\partial_2  u}_{H^2}+ \norm{\mathcal{D}^{\beta} \partial_t^{j}\partial_1\partial_3  u}_{H^2}
     + \norm{\mathcal{D}^{\beta} \partial_t^{j}\partial_2^2  u}_{H^2})\\
     &\quad +Ce^{\frac t4}  t^{j+1}\Big(\norm{\mathcal{D}^{\beta} \partial_t^{j+1}  u}_{H^2}+\norm{\mathcal{D}^{\beta} \partial_t^{j}\nabla  u}_{H^2}\Big).
	\end{aligned}
\end{equation}
It remains to handle the terms on the right-hand side of \eqref{abeta}.
Condition \eqref{sm1} enables us to
use the inductive assumption \eqref{inductiveass} to conclude that, recalling \eqref{degamma} and $\tilde\beta=\beta-(1,0,0)$,
\begin{align*}
&\sup_{t\leq T}e^{\frac t4}  t^{j+1+\frac{1}{2}}\Big( \norm{\mathcal{D}^{\tilde \beta} \partial_t^{j+1}u}_{H^2}+\norm{\mathcal{D}^{\tilde \beta} \partial_t^{j+1}\partial_2u}_{H^2}+\norm{\mathcal{D}^{\tilde \beta} \partial_t^{j+1}\partial
     _3u}_{H^2}\Big)\\
&\leq T^\frac{1}{2}\sup_{t\leq T}e^{\frac t4}  t^{j+1}\norm{\mathcal{D}^{\tilde \beta} \partial_t^{j+1}u}_{H^2}+\sup_{t\leq T}e^{\frac t4}  t^{j+1}\norm{\mathcal{D}^{\tilde \beta+(0,1,0)} \partial_t^{j+1}u}_{H^2}\\
&\quad+T^\frac{1}{2}\sup_{t\leq T}e^{\frac t4}  t^{j+1}\norm{\mathcal{D}^{\tilde \beta+(0,0,1)} \partial_t^{j+1}u}_{H^2}\\
&\leq CT^\frac{1}{2}\frac{\eps_0 A^{(j+1)+\beta_2+\beta_3+2}T^{\frac{\beta_1-1}{2}}B^{\beta_1-1} [(j+\beta_1+\beta_3+1)!]^{\sigma-\delta} [(j+\abs\beta+1)!]^{\delta}}{(j+\abs\beta+2)^{6+2\sigma}}\\
&\leq \frac{CA^2}{TB^3} \frac{\eps_0 A^{ j+\alpha_2+\alpha_3+1}T^{\frac{\alpha_1}{2}}B^{\alpha_1} [( j+ \alpha_1+\alpha_3)!]^{\sigma-\delta} [(j +\abs\alpha)!]^{\delta}}{(j+\abs\alpha+1)^{6+2\sigma}},
\end{align*}
the last inequality following from the fact that $\beta=\alpha-(2,0,0)$  in view of \eqref{debeta}.  For the terms in the third line  of \eqref{abeta},  we use \eqref{debeta} and the inductive assumption \eqref{inductiveass} again to get that
\begin{align*}
&\sup_{t\leq T}  e^{\frac t4}  t^{j+1}(\norm{\mathcal{D}^{\beta} \partial_t^{j}\partial_1\partial_2  u}_{H^2}+ \norm{\mathcal{D}^{\beta} \partial_t^{j}\partial_1\partial_3  u}_{H^2}
     + \norm{\mathcal{D}^{\beta} \partial_t^{j}\partial_2^2  u}_{H^2})\\
&\leq \sup_{t\leq T}  e^{\frac t4}  t^{j}\norm{\mathcal{D}^{\beta+(1,1,0)} \partial_t^{j} u}_{H^2}+T^\frac{1}{2}\sup_{t\leq T}  e^{\frac t4}  t^{j}\norm{\mathcal{D}^{\beta+(1,0,1)} \partial_t^{j} u}_{H^2}\\
&\quad+\sup_{t\leq T}  e^{\frac t4}  t^{j}\norm{\mathcal{D}^{\beta+(0,2,0)} \partial_t^{j} u}_{H^2}\\
&\leq CT^\frac{1}{2}\frac{\eps_0 A^{j+\beta_2+\beta_3+3}T^{\frac{\beta_1+1}{2}}B^{\beta_1+1} [(j+\beta_1+\beta_3+2)!]^{\sigma-\delta} [(j+\abs\beta+2)!]^{\delta}}{(j+\abs\beta+3)^{6+2\sigma}}\\
&\leq \frac{CA^2}{B} \frac{\eps_0 A^{ j+\alpha_2+\alpha_3+1}T^{\frac{\alpha_1}{2}}B^{\alpha_1} [( j+ \alpha_1+\alpha_3)!]^{\sigma-\delta} [(j +\abs\alpha)!]^{\delta}}{(j+\abs\alpha+1)^{6+2\sigma}}.
\end{align*}
Similarly, for  the  terms in the last line  of \eqref{abeta}, we have
\begin{align*}
&\sup_{t\leq T}e^{\frac t4}  t^{j+1}\Big(\norm{\mathcal{D}^{\beta} \partial_t^{j+1}  u}_{H^2}+\norm{\mathcal{D}^{\beta} \partial_t^{j}\nabla  u}_{H^2}\Big)\\
&\leq C\frac{\eps_0 A^{j+\beta_2+\beta_3+2}T^{\frac{\beta_1+2}{2}}B^{\beta_1+1} [(j+\beta_1+\beta_3+1)!]^{\sigma-\delta} [(j+\abs\beta+1)!]^{\delta}}{(j+1+\abs\beta+1)^{6+2\sigma}}\\
&\leq \frac{CA}{B} \frac{\eps_0 A^{ j+\alpha_2+\alpha_3+1}T^{\frac{\alpha_1}{2}}B^{\alpha_1} [( j+ \alpha_1+\alpha_3)!]^{\sigma-\delta} [(j +\abs\alpha)!]^{\delta}}{(j+\abs\alpha+1)^{6+2\sigma}}.
\end{align*}
As a result, observing $A,B,T\geq 1$ and thus  substituting these estimates above into \eqref{abeta}, we get the desired estimate \eqref{linpart}.

{\it Step 3.}
In this step we deal with the first term on the right-hand side of \eqref{linpart}. Precisely, we will prove that
\begin{multline}\label{nonpart}
	\sup_{t\leq T} e^{\frac t4}  t^{j+1}  \norm{\mathcal D^{\beta} \partial_t^{j}  \big((u\cdot\nabla) u\big)}_{H^2}\\
	\leq  \frac{CA^2}{B} \frac{\eps_0 A^{ j+\alpha_2+\alpha_3+1}T^{\frac{\alpha_1}{2}}B^{\alpha_1} [( j+ \alpha_1+\alpha_3)!]^{\sigma-\delta} [(j +\abs\alpha)!]^{\delta}}{(j+\abs\alpha+1)^{6+2\sigma}}.
\end{multline}
In fact,  using Leibniz's formula the fact that  $H^2([0,1]\times\mathbb R^2)$ is an algebra under pointwise multiplication, we obtain
\begin{equation}\label{halpha++}
\begin{aligned}
		& \sup_{t\leq T} e^{\frac{t}{4}} t^{j+1}\norm{ \mathcal D^{\beta} \partial_t^{j} \big ((u\cdot\nabla)u\big)  }_{H^2} \\
		&\leq C \sup_{t\leq T} \,  \sum_{ (\gamma,k)\leq (\beta,j)}  \binom{\beta}{\gamma}\binom{j}{k} \Big(e^{\frac{t}{4}}t^{k}\norm{ \mathcal D^{\gamma} \partial_t^{k}u }_{H^2}\Big) \Big(   t^{j-k} t  \norm{ \mathcal D^{\beta-\gamma} \partial_t^{j-k}\nabla u }_{H^2}\Big)\\
		&\leq   C\sum_{ (\gamma,k)\leq (\beta,j)}    \frac{(\abs\beta+j)!}{(\abs\gamma+k)!(\abs\beta+j-\abs\gamma-k)!}\\
		&\qquad \times \frac{\eps_0 A^{k+\gamma_2+\gamma_3+1}T^{\frac{\gamma_1}{2}} B^{\gamma_1}  [(k+\gamma_1+\gamma_3)!]^{\sigma-\delta}[(k+\abs\gamma)!]^{ \delta}}{(k+\abs\gamma+1)^{6+2\sigma}}\\
		&\qquad\times \eps_0 A^{j-k+(\beta_2-\gamma_2)+(\beta_3-\gamma_3)+2}T^{\frac{\beta_1-\gamma_1+2}{2}} B^{\beta_1-\gamma_1+1}
		\\
	&\qquad  \times\frac{[(j-k+\beta_1-\gamma_1+\beta_3-\gamma_3+1)!]^{\sigma-\delta}[( j-k+\abs\beta-\abs\gamma+1)!]^{\delta}}{ ( j-k+\abs\beta-\abs\gamma+2)^{6+2\sigma}}\\
	&\leq \frac{C\eps_0 A^2}{B} \frac{\eps_0 A^{ j+\alpha_2+\alpha_3+1}T^{\frac{\alpha_1}{2}}B^{\alpha_1} [( j+ \alpha_1+\alpha_3)!]^{\sigma-\delta} [(j +\abs\alpha)!]^{\delta}}{(j+\abs\alpha+1)^{6+2\sigma}},
		\end{aligned}
	\end{equation}
	where  the second inequality holds because of the inductive assumption \eqref{inductiveass} as well as \eqref{factorial},
	and in
	the last inequality we used the fact that, observing $\beta=\alpha-(2,0,0),$
		\begin{multline*}
		\frac{1}{(k+\abs\gamma+1)^{6+2\sigma}( j-k+\abs\beta-\abs\gamma+2)^{6+2\sigma}} \\
		\leq  \frac{C}{(j+\abs\alpha+1)^{6+2\sigma}} \Big(\frac{1}{(k+\abs\gamma+1)^{6+2\sigma} }+\frac{1}{( j-k+\abs\beta-\abs\gamma+2)^{6+2\sigma}}\Big),
	\end{multline*}
	and
	\begin{align*}
		&\frac{(\abs\beta+j)!}{(\abs\gamma+k)!(\abs\beta+j-\abs\gamma-k)!} [(k+\gamma_1+\gamma_3)!]^{\sigma-\delta}[(k+\abs\gamma)!]^{ \delta}\\
	&\qquad\quad  \times [(j-k+\beta_1-\gamma_1+\beta_3-\gamma_3+1)!]^{\sigma-\delta}[( j-k+\abs\beta-\abs\gamma+1)!]^{\delta}\\
	&\leq   (\abs\beta+j)! [(k+\abs\gamma)!]^{ \delta-1}[( j-k+\abs\beta-\abs\gamma)!]^{\delta-1}\\
	&\qquad\quad  \times [(j+\beta_1+\beta_3+1)!]^{\sigma-\delta} ( j-k+\abs\beta-\abs\gamma+1)^\delta\\
	&\leq  C  [(j+\beta_1+\beta_3)!]^{\sigma-\delta}  [(\abs\beta+j+1)!]^\delta    \leq  C  [(j+\alpha_1+\alpha_3)!]^{\sigma-\delta}  [(j+\abs\alpha)!]^\delta,
	\end{align*}
	due to the fact  $p!q!\leq (p+q)!$ and $1\leq\delta\leq\sigma$. Thus  assertion \eqref{nonpart}  follows from \eqref{halpha++}.

{\it Step 4.} Substituting the estimate \eqref{nonpart} into \eqref{linpart} yields
\begin{multline*}
	\sup_{t\leq T} e^{\frac t4}  t^{j+\frac{\alpha_1+\alpha_2}{2}}\norm{ \partial^{\alpha} \partial_t^{j}  u}_{H^2}=\sup_{t\leq T} e^{\frac t4}  t^{j}\norm{\mathcal D^{\alpha} \partial_t^{j}   u}_{H^2}\\
	\leq  \frac{CA^2}{B} \frac{\eps_0 A^{ j+\alpha_2+\alpha_3+1}T^{\frac{\alpha_1}{2}}B^{\alpha_1} [( j+ \alpha_1+\alpha_3)!]^{\sigma-\delta} [(j +\abs\alpha)!]^{\delta}}{(j+\abs\alpha+1)^{6+2\sigma}}\\
	\leq  \frac{\eps_0 A^{ j+\alpha_2+\alpha_3+1}T^{\frac{\alpha_1}{2}}B^{\alpha_1} [( j+ \alpha_1+\alpha_3)!]^{\sigma-\delta} [(j +\abs\alpha)!]^{\delta}}{(j+\abs\alpha+1)^{6+2\sigma}},
\end{multline*}
provided that $B$ is chosen sufficiently large so that  $B\geq CA^2+1.$ Note this estimate holds  for all $j\geq 0$ and any multi-index $\alpha\in\mathbb Z_+^3$ with $\alpha_1=N$ under the inductive assumption \eqref{inductiveass}, and hence extends to all $(j,\alpha)\in\mathbb Z_+\times\mathbb Z_+^3$. The proof of Proposition \ref{prop:x1} is completed.
  \end{proof}

  \begin{proof}[Completing the proof of Theorem \ref{thm:smoothing}]
By Proposition \ref{prop:x1}, we have that for any $(j,\alpha)\in\mathbb Z_+\times\mathbb Z_+^3$,
 \begin{align*}
	&\sup_{t\leq T} e^{\frac t4}  t^{j+\frac{\alpha_1+\alpha_2}{2}}\norm{ \partial^{\alpha} \partial_t^{j}  u}_{H^2}\\
	&\leq  \frac{\eps_0 A^{ j+\alpha_2+\alpha_3+1}T^{\frac{\alpha_1}{2}}B^{\alpha_1} [( j+ \alpha_1+\alpha_3)!]^{\sigma-\delta} [(j +\abs\alpha)!]^{\delta}}{(j+\abs\alpha+1)^{6+2\sigma}}\\
	&\leq   \eps_0 (2^\delta A)^{ j+\alpha_2+\alpha_3+1}T^{\frac{\alpha_1}{2}}(2^\delta B)^{\alpha_1} [( j+ \alpha_1+\alpha_3)!]^{\sigma}  ( \alpha_2 !)^{\delta},
\end{align*}
the last line using the fact $(p+q)!\leq 2^{p+q}p!q!$, which yields
\begin{equation*}
	[(j +\abs\alpha)!]^{\delta}\leq 2^{(j+\abs\alpha)\delta} (\alpha_2!)^\delta [(j+\alpha_1+\alpha_3)!]^{\delta}.
\end{equation*}
As a result, assertion \eqref{result:2} follows by choosing $C_1=2^\delta A$ and $C_2=2^\delta B$, which completes the proof of Theorem \ref{thm:smoothing}.
  \end{proof}

\section{Refined Gevrey radius for the variable $x_2$}\label{sec:radius}

Let $u$ be the solution to system \eqref{ANS}, as constructed in Theorem \ref{thm:main1}.
In this section, we investigate the Gevrey radius $R_{x_2}^{{\delta}}(t)$ of the solution  $u$ near the initial time. Specifically, we  present the proof of Theorem \ref{thm:radius} by  establishing a refined estimate for $R_{x_2}^{{\delta} }(t)$ in the variable $x_2$.

\subsection{Decomposition of the solution and proof of Theorem \ref{thm:radius}}

Note  estimate \eqref{result:2} in Theorem \ref{thm:smoothing} implies $R_{x_2}^{{\delta} }(t)$ is bounded from below by
\begin{equation*}
\rho\lambda t^{\frac{1}{2}}
\end{equation*}
with $0<\lambda<1$ and $\frac{\rho_0}{2}\leq \rho\leq \rho_0.$   To analyze the behavior of the Gevrey radius near the initial time, we perform a Gevrey-type energy estimate for  $u$ on $[0,T]$, where $0<T\leq T_0\leq 1$ for some sufficiently small $T_0$ to be chosen later (see \eqref{dep}). For this purpose, it is natural to replace the small constant $0<\lambda<1$ in \eqref{dlambda}  by a time-dependent function that grows appropriately near $t=0$. Precisely, for any given time $T>0$  we define accordingly
  \begin{equation}\label{mut}
 	\mu(t)=\mu_{T}(t)\stackrel{\rm def}{=} \frac{1}{ 2^{5+2\sigma} 3\rho_0}\frac{\sqrt{\psi(T)}}{T}t,
 \end{equation}
 where $\psi(T)\geq 1$ will be specified later (see \eqref{psi}), and additionally, we introduce the function
\begin{equation}\label{taut}
\eta(t)=\eta_T(t)\stackrel{\rm def}{=}\frac{\rho_0}{4}+\frac{\rho_0}{4}e^{-\frac{t}{\sqrt{T}}},
\end{equation}
which satisfies
\begin{equation}\label{est:tau}
  \forall\ t\geq 0,\quad \frac{\rho_0}{4}\leq \eta_T(t)\leq \frac{\rho_0}{2}.
\end{equation}
 In the subsequent discussion, we will write $\mu$ and $\eta$ instead of $\mu_T$ and $\eta_T$, omitting the subscript  if no confusion occurs.

\begin{definition}
	\label{def:normx2+} Let $T>0$ be an arbitrary time,   and    let $(\mu(t),\psi(T))$  and $\eta(t)$  be defined as in \eqref{mut} and \eqref{taut}, respectively.   We define   two new norms $\abs{\cdot}_{X_{\eta,\mu,T}}$ and $\abs{\cdot}_{Y_{\eta,\mu,T}} $ as follows:
	\begin{equation*}
		\abs{g}_{X_{\eta,\mu,T}}\stackrel{\rm def}{=}e^{-\frac{\psi(T)}{T}t}\abs{g}_{X_{\eta, \mu}}\  \textrm{ and }\ 	 \abs{g}_{Y_{\eta,\mu,T}}\stackrel{\rm def}{=}e^{-\frac{\psi(T)}{T}t}\abs{g}_{Y_{\eta,\mu}},
	\end{equation*}
	where  $\abs{\cdot}_{X_{\eta,\mu}}$ and $\abs{\cdot}_{Y_{\eta,\mu}}$ given in Definition \ref{def:normx2}. Explicitly, these norms are defined by
 \begin{eqnarray*}
 \left\{
\begin{aligned}
	 &\abs{g}_{X_{\eta,\mu,T}}^2 =   e^{-\frac{2\psi(T)}{T}t}\sum_{\stackrel{(j,\alpha)\in\mathbb Z_+\times\mathbb Z_+^3}{ j+\abs \alpha\leq 3}}   \sum_{\beta\in\mathbb Z_+^2}   r^{ \alpha_1}  M_{\eta,\beta}^2  \norm{ D_\mu^{\beta}\partial_t^j\partial^{ \alpha } u}^2_{L^2},\\
	 &\abs{g}_{Y_{\eta,\mu,T}}^2=e^{-\frac{2\psi(T)}{T}t}  \sum_{\stackrel{(j,\alpha)\in\mathbb Z_+\times\mathbb Z_+^3}{ j+\abs \alpha\leq 3}} \, \sum_{ \beta\in\mathbb Z_+^2  } r^{\alpha_1} \inner{\abs\beta +1}    M_{\eta,\beta}^2  \norm{D_\mu^{\beta} \partial_t^j\partial^{ \alpha } u}^2_{L^2}.
	 \end{aligned}
	 \right.
\end{eqnarray*}
Recall $D_\mu^\beta$ and $M_{\eta,\beta}$ are defined as in \eqref{dlambda} and \eqref{mrho}, respectively, namely,
\begin{equation}\label{+dlambda}
	D_{\mu}^\beta= \mu^{\beta_1}t^{\frac{\beta_1}{2}}  \partial_2 ^{\beta_1}\partial_3^{\beta_2} \ \textrm{ for } \  \beta=(\beta_1,\beta_2)\in\mathbb Z_+^2,
\end{equation}
and
 \begin{equation*}%\label{+mtau}
	M_{\eta,\beta}=\frac{\eta^{\abs\beta+1}(\abs\beta+1)^{6+2\sigma}}{(\beta_2!)^{\sigma-\delta}(\abs\beta!)^{\delta}} \ \textrm{ for }\  \beta=(\beta_1,\beta_2)\in\mathbb Z_+^2,
\end{equation*}
where $\delta=\delta(\sigma)$ is defined in \eqref{gamma1}.
 \end{definition}

 \begin{remark}
Note that in Section \ref{sec:x2},  the smallness assumption on $\lambda$  is essential for  estimating  $\abs{u}_{X_{\rho,\lambda}}$ (see Lemma \ref{lem:J5} for details).
However, if we replace the norm $\abs{\cdot}_{X_{\rho,\lambda}}$ with
 $\abs{\cdot}_{X_{\eta,\mu}}$, the argument in Lemma \ref{lem:J5} no longer applies due to the fact that $\mu\geq 1$ is large. To address this, we introduce the weight function $e^{-\frac{\psi(T) }{T}t}$ to control the analogue of the term $J_5$ from Lemma \ref{lem:J5}, seeing \eqref{demu} for details.
\end{remark}

Let $u$ be the solution constructed in Theorem \ref{thm:main1} which satisfies the Gevrey regularity estimate \eqref{result:2}. This implies that the radius $R_{x_2}^\delta$ defined in \eqref{+def:radius} is bounded from below by a constant multiple of $\sqrt{t}.$  In this section we derive  a refined estimate for this radius.
Following the approach in \cite{MR4816041}, we decompose the solution  $u$ as
\begin{equation}
	\label{decom}
	u=u_L+v,
\end{equation}
where the linear part $u_L$ satisfies
\begin{equation}\label{eq:uL}
    \left\{
    \begin{aligned}
&\partial_tu_L-\Delta_{\rm h}u_L=0,\quad(t,x)\in \mathbb{R}_+\times[0,1]\times\mathbb{R}^2,\\
&u_L|_{x_{1}=0,1}=0,\quad u_L|_{t=0}=u_0,
    \end{aligned}
    \right.
\end{equation}
with the initial datum $u_0$ being the same as that in \eqref{ANS}.
Note  $u_0|_{x_1=0,1}=\partial_1^2u_0|_{x_1=0,1}=\partial_1^4u_0|_{x_1=0,1}=0$ by the assumption of Theorem \ref{thm:radius}. Then letting $x_1=0$ or $1$ in equation \eqref{eq:uL} yields
\begin{equation}\label{bdc}
 u_{L}|_{x_1=0, 1}=\partial_1^2 u_L|_{x_1=0, 1}=\partial_1^4 u_L|_{x_1=0, 1}=0.
\end{equation}
Consequently, applying the divergence operator to \eqref{eq:uL} and performing a direct energy estimate yields $\norm{\divv u_L}_{L^2}\leq \norm{\divv u_0}_{L^2}$. This gives $\divv u_L=0.$ Then, since $\divv u=0,$  the
   nonlinear part $v$  in \eqref{decom} satisfies
\begin{equation}\label{eq:v}
    \left\{
    \begin{aligned}
&\partial_tv-\Delta_{\rm h}v+u\cdot\nabla u=-\nabla p,\quad (t,x)\in\mathbb{R}_+\times[0,1]\times\mathbb{R}^2,\\
&\divv v=0,\quad v|_{x_{1}=0,1}=0,\quad v|_{t=0}=0.
    \end{aligned}
    \right.
\end{equation}
We begin by stating the estimates on $u_L$ and $v,$ postponing the proof to later subsections.

  \begin{proposition}[Estimate on the linear part]\label{lem:uL}
Let $u_L$ satisfy the heat equation \eqref{eq:uL}. Then there exists a constant $C_{in}>0$, depending on the Gevrey norm of the initial datum $u_0$ as specified  in assumption \eqref{assume:small}, such that for any  $0<T\leq 1$,
\begin{equation*}
	\sup_{t\leq T} \abs{ u_L(t)}_{X_{\eta,\mu,T}}^2+\int_0^T \abs{\nabla_{\rm h}  u_L(t)}_{X_{\eta,\mu,T}}^2dt+\frac{1}{\sqrt{T}}\int_0^T \abs{ u_L(t)}_{Y_{\eta,\mu,T}}^2dt\leq C_{in},
\end{equation*}
and
\begin{equation*}
	\sup_{t\leq T} \abs{\partial_3 u_L(t)}_{X_{\eta,\mu,T}}^2+\int_0^T \abs{\nabla_{\rm h} \partial_3 u_L(t)}_{X_{\eta,\mu,T}}^2dt+\frac{1}{\sqrt{T}}\int_0^T \abs{\partial_3 u_L(t)}_{Y_{\eta,\mu,T}}^2dt\leq C_{in},
\end{equation*}
where the norms $\abs{\cdot}_{X_{\eta,\mu,T}}$ and $\abs{\cdot}_{Y_{\eta,\mu,T}}$ are given in Definition \ref{def:normx2+}.
\end{proposition}

\begin{proposition}[Estimate on the nonlinear part]\label{prop:v}
Let $v$ satisfy the initial-boundary problem \eqref{eq:v}.  Then there exists  a constant  $\tilde C\ge 1$  such that  for any  $0<T\leq 1$,
	\begin{equation}\label{+est:v}
	\begin{aligned}
	&	\sup_{t\leq T}\abs{v(t)}_{X_{\eta,\mu,T}}^2+\int_0^T \abs{ \nabla_{\rm h}v}_{X_{\eta,\mu,T}}^2dt +\frac{1}{\sqrt{T}}\int_0^T \abs{v}_{Y_{\eta,\mu,T}}^2dt \leq
\tilde Cr^{-9} e^{3\psi(T)}\sqrt{T}\\
&\qquad+\tilde Cr^{-9}e^{ 3\psi(T)} \sqrt{T}\Big(1+\sup_{ t\leq T}\abs{v}_{X_{\eta,\mu,T}}^2\Big)\Big(\sup_{ t\leq T}\abs{v}_{X_{\eta,\mu,T}}^2+\int^T_0\abs{\nabla_{\rm h}v}_{X_{\eta,\mu,T}}^2dt\Big)\\
	&\qquad+\tilde C r^{-9}e^{3\psi(T)}\sqrt T \Big(1 +  \sup_{ t\leq T} \abs{v}_{X_{\eta,\mu,T}}^2 \Big) \frac{1}{\sqrt T}\int_0^T    \abs{v}_{Y_{\eta,\mu,T}}^2dt,
	\end{aligned}
\end{equation}
recalling $\psi(T)$ is given in \eqref{mut} and  $\abs{\cdot}_{X_{\eta,\mu,T}}$ and $\abs{\cdot}_{Y_{\eta,\mu,T}}$ are given in Definition \ref{def:normx2+}.
\end{proposition}

The proofs of Propositions \ref{lem:uL} and \ref{prop:v} are given in  Subsections \ref{subsec:linear} and  \ref{subsec:nonlinear}, respectively. Using  these two propositions, we can now complete the proof of Theorem \ref{thm:radius}.

\begin{proof}[Proof of Theorem \ref{thm:radius}]
Let $\tilde C$ be the constant constructed in Proposition \ref{prop:v}. We define $T_0$ and $\psi(T)$ by setting
\begin{equation}\label{dep}
	\tilde Cr^{-9}e^{ 3\psi(T)}\sqrt{T}=\frac{1}{4} \ \textrm{ for }\  T\leq T_0\stackrel{\rm def}{=}\min\Big\{\frac{r^{18}}{16e^6\tilde C^2},\ \frac{1}{4}\Big\}.
\end{equation}
Then for any $0<T\leq T_0$, it holds that
 \begin{equation} \label{psi}
  \psi(T)= \frac13\ln \frac{r^9}{4\tilde C\sqrt{T}}\approx |\ln T|.
 \end{equation}
 Let $T_0$ and $\psi(T)$ be defined as above.
We will use the standard bootstrap argument to deduce the estimate on $v.$
Suppose that for any  $0<T\leq T_0$,
\begin{equation}\label{apes}
\sup_{t\leq T}\abs{v}_{X_{\eta,\mu,T}}^2+\int_0^T \abs{ \nabla_{\rm h}v}_{X_{\eta,\mu,T}}^2dt +\int_0^T \abs{v}_{Y_{\eta,\mu,T}}^2dt\leq 1.
\end{equation}
 Combining \eqref{apes} and  \eqref{dep} with \eqref{+est:v}, we obtain,  for any $0<T\leq T_0$,
 \begin{equation*}
 	\begin{aligned}
 \frac{1}{2}\sup_{t\leq T}\abs{v}_{X_{\eta,\mu,T}}^2+\frac{1}{2}\int_0^T \abs{ \nabla_{\rm h}v}_{X_{\eta,\mu,T}}^2dt + \frac{1}{2\sqrt{T}} \int_0^T \abs{v}_{Y_{\eta,\mu,T}}^2dt\leq \frac14. 	
 \end{aligned}
 \end{equation*}
 Since $T\leq T_0\leq\frac{1}{4}$, it follows that
  \begin{equation*}
 	\begin{aligned}
\sup_{t\leq T}\abs{v}_{X_{\eta,\mu,T}}^2+\int_0^T \abs{ \nabla_{\rm h}v}_{X_{\eta,\mu,T}}^2dt +\int_0^T \abs{v}_{Y_{\eta,\mu,T}}^2dt\leq \frac12.	
 \end{aligned}
 \end{equation*}
 By a bootstrap argument, we conclude that
 \begin{eqnarray*}
 	 \sup_{t\leq T}\abs{v}_{X_{\eta,\mu,T}}<+\infty,
 \end{eqnarray*}
 which with Proposition \ref{lem:uL} yields that
\begin{equation}\label{finalest}
 \sup_{ t\leq T}\abs{u(t)}_{X_{\eta,\mu,T}}\leq  \sup_{  t\leq T}\abs{ u_{L}(t)}_{X_{\eta,\mu,T}}+ \sup_{ t\leq T}\abs{v(t)}_{X_{\eta,\mu,T}} <+\infty.
\end{equation}
In particular, considering the norm $\abs{u(t)}_{X_{\eta,\mu,T}}$ at time $T,$  we have
\begin{align*}
e^{-\frac{\psi(T) t}{T}}\abs{u(t)}_{X_{\eta,\mu}}\big|_{t=T}=e^{-\psi(T)}\abs{u(T)}_{X_{\eta(T),\mu(T)}}.
\end{align*}
Combining this with \eqref{finalest}, we conclude that for any $T \leq T_0$ with $T_0$  defined in \eqref{dep},
\begin{align*}
    \abs{u(T)}_{X_{\eta(T),\mu(T)}}<+\infty.
\end{align*}
Hence for all $T\leq T_0$,  we recall  $\mu(T)=\frac{1}{   2^{5+2\sigma}3\rho_0} \sqrt{\psi(T)}$ and $\eta(T)=\frac{\rho_0}{4}+\frac{\rho_0}{4}e^{-\sqrt{T}}$ and obtain
 \begin{equation*}
	R_{x_2}^\delta(T) \geq C\eta(T)\mu(T)\sqrt{T}\geq C \sqrt{\psi(T)T}\geq C\sqrt{T|\ln T|},
\end{equation*}
the last inequality following from \eqref{psi}.
This completes the proof of Theorem \ref{thm:radius}.
\end{proof}

 \subsection{Proof of Proposition \ref{lem:uL}: estimate on the linear part}\label{subsec:linear}

This subsection is devoted to proving Proposition \ref{lem:uL}.  Observe $D_\mu^\beta=\mu^{\beta_1}t^{\frac{\beta_1}{2}}\partial_2^{\beta_1}\partial_3^{\beta_2}$ for $\beta=(\beta_1,\beta_2)\in\mathbb Z_+^2.$ Then
	 the boundary condition \eqref{bdc} enables us to   perform a direct energy estimate for $u_L$ to conclude that,  for any $(j,\alpha)\in\mathbb Z_+\times \mathbb Z_+^3$  with $j+\abs\alpha\leq 3$,
	  \begin{equation}\label{lipart}
	\begin{aligned}
& \frac{1}{2}\frac{d}{dt} \sum_{ \beta \in \mathbb Z_+^2} r^{\alpha_1}M_{\eta,\beta}^2\norm{D_\mu^\beta \partial_t^j\partial^{\alpha} u_L}_{L^2}^2+ \sum_{ \beta \in \mathbb Z_+^2}r^{\alpha_1}M_{\eta,\beta}^2\norm{ D_\mu^\beta \partial_t^j\partial^{\alpha} \nabla_{\rm h} u_L}_{L^2}^2\\
 &\qquad -\frac{\eta^{'}}{\eta}  \sum_{\beta \in \mathbb Z_+^2}r^{\alpha_1}(\abs{\beta }+1)M_{\eta,\beta}^2\norm{ D_\mu^\beta \partial_t^j\partial^{\alpha}  u_L}_{L^2}^2 \\
&\leq   \sum_{\stackrel{ \beta \in \mathbb Z_+^2}{\beta_1\geq 1}} r^{\alpha_1}   \beta_1\frac{\big(\mu (t)t^{\frac{1}{2}}\big)'}{\mu(t) t^{\frac{1}{2}}}   M_{\eta,\beta}^2\norm{D_\mu^\beta \partial_t^j\partial^{\alpha} u_L}_{L^2}^2.
\end{aligned}
\end{equation}
Moreover, we recall $\mu(t)t^{\frac12}=\frac{1}{  2^{5+2\sigma}3\rho_0 }\frac{\sqrt{\psi(T)}}{T}t^{\frac32}$
 in view of the definition \eqref{mut} of $\mu(t)$, to get  \begin{equation*}
 \forall\ 0\leq t \leq T,\quad   \big|\big(\mu (t)t^{\frac{1}{2}}\big)'  \big|  \leq   \frac{1}{2^{6+2\sigma}\rho_0} \Big(\frac{\psi(T)}{T}\Big)^{\frac{1}{2}}.
 \end{equation*}
On the other hand,
  for any $\beta=(\beta_1,\beta_2)\in\mathbb Z_+^2$ with $\beta_1\geq 1,$ we use \eqref{+dlambda} to write
 \begin{eqnarray*}
 	D_\mu^\beta=\mu(t) t^{\frac12}D_\mu^{\tilde \beta}\partial_1\  \textrm{ for } \ \tilde\beta=\beta-(1,0),
 \end{eqnarray*}
 and thus
\begin{multline*}
 	   \frac{M_{\eta,\beta}^2\norm{D_\mu^\beta \partial_t^j\partial^{\alpha} u_L}_{L^2}^2}{\mu(t)t^{\frac12}}=
 	        \frac{M_{\eta,\beta}}{M_{\eta,\tilde\beta}}M_{\eta,\beta}M_{\eta,\tilde\beta} \inner{D_\mu^{\tilde\beta} \partial_t^j\partial^{\alpha} \partial_1 u_L,\ D_\mu^\beta \partial_t^j \partial^{\alpha} u_L}_{L^2} \\
 	        \leq
 	        \frac{2^{6+2\sigma}\rho_0}{\abs\beta^{\delta}}M_{\eta,\beta}M_{\eta,\tilde\beta} \norm{D_\mu^{\tilde\beta} \partial_t^j \partial^{\alpha} \partial_1 u_L}_{L^2}\norm{D_\mu^\beta \partial_t^j\partial^{\alpha} u_L}_{L^2},
\end{multline*}
the last inequality following from the same argument as that in \eqref{betatilde}.
Combining the two inequalities above,
we have
\begin{equation}\label{demu}
\begin{aligned}
 	  &\sum_{\stackrel{(j,\alpha)\in\mathbb Z_+\times\mathbb Z_+^3}{ j+\abs \alpha\leq 3}}  \sum_{\stackrel{\beta \in \mathbb Z_+^2}{\beta_1\ge 1}} r^{\alpha_1}   \beta_1\frac{\big(\mu (t)t^{\frac{1}{2}}\big)'}{\mu(t) t^{\frac{1}{2}}}   M_{\eta,\beta}^2\norm{D_\mu^\beta\partial_t^j \partial^{\alpha} u_L}_{L^2}^2\\
 	   &\leq  \Big(\frac{\psi(T)}{T}\Big)^{\frac{1}{2}}   \sum_{\stackrel{(j,\alpha)\in\mathbb Z_+\times\mathbb Z_+^3}{ j+\abs \alpha\leq 3}} \sum_{\stackrel{\beta \in \mathbb Z_+^2}{\beta_1\ge 1}} r^{\alpha_1}      M_{\eta,\tilde\beta}  M_{\eta,\beta} \norm{D_\mu^{\tilde\beta} \partial_t^j\partial^{\alpha} \partial_1 u_L}_{L^2}    \norm{D_\mu^\beta \partial_t^j \partial^{\alpha} u_L}_{L^2}\\
       &\leq  \Big(\frac{\psi(T)}{T}\Big)^{\frac{1}{2}}\abs{\nabla_{\rm h} u_L}_{X_{\eta,\mu}}    \abs{u_L }_{X_{\eta,\mu}}  \leq \frac{1}{4}\abs{\nabla_{\rm h} u_L}_{X_{\eta,\mu}}^2 +  \frac{\psi(T)}{T}  \abs{u_L}_{X_{\eta,\mu}}^2,
       \end{aligned}
\end{equation}
which with \eqref{lipart} yields, recalling $ \abs{\cdot}_{X_{\eta,\mu}}$ and $ \abs{\cdot}_{Y_{\eta,\mu}}$ are given in Definition \ref{def:normx2},
\begin{equation*}
	\frac{1}{2}\frac{d}{dt}\abs{  u_L}_{X_{\eta,\mu}}^2-  \frac{\psi(T)}{T}\abs{  u_L}_{X_{\eta,\mu}}^2+\frac{3}{4}\abs{\nabla_{\rm h} u_L}_{X_{\eta,\mu}}^2-\frac{\eta^{'}}{\eta}\abs{ u_L}_{Y_{\eta,\mu}}^2\leq 0.
\end{equation*}
  Multiplying both sides by    $e^{-\frac{2\psi(T)}{T}t}$  and recalling the definitions of $ \abs{\cdot}_{X_{\eta,\mu,T}}$ and $ \abs{\cdot}_{Y_{\eta,\mu,T}}$ from Definition \ref{def:normx2+},  we obtain
\begin{equation}\label{ulde}
	\frac{1}{2}\frac{d}{dt}  \abs{  u_L}_{X_{\eta,\mu,T}}^2+\frac{3}{4} \abs{\nabla_{\rm h} u_L}_{X_{\eta,\mu,T}}^2-\frac{\eta^{'}}{\eta}  \abs{ u_L}_{Y_{\eta,\mu,T}}^2\leq 0.
\end{equation}
 Moreover, from \eqref{taut} and \eqref{est:tau}, it follows that
\begin{equation}\label{bdrho}
	\forall\ t\leq T\leq 1, \quad -\frac{\eta'(t)}{\eta(t)}=\frac{\rho_0e^{-\frac{t}{\sqrt{T}}}}{4\sqrt{T}\eta(t)}\geq \frac{e^{-\sqrt{T}}}{2\sqrt{T}}\ge\frac{1}{2e\sqrt{T}}.
\end{equation}
As a result, integrating \eqref{ulde} over $[0,T]$ and noting that by  \eqref{xandtx}, \eqref{shotbeha} and \eqref{inda},
\begin{multline*}
	\lim_{t\rightarrow0} \abs{  u_L(t)}_{X_{\eta,\mu,T}}^2=\lim_{t\rightarrow0} e^{-\frac{2\psi(T)}{T}t}\abs{  u_L(t)}_{X_{\eta,\mu}}^2 =\lim_{t\rightarrow0} \abs{  u_L(t)}_{X_{\eta}}^2\\
    \leq C\norm{u_0}_{G_{\frac{\rho_0}{2},\sigma,6}}^2+C\norm{u_0}_{G_{\frac{\rho_0}{2},\sigma,6}}^8,
\end{multline*}
we conclude that for any $0<T\leq 1,$
\begin{equation*}
	\sup_{t\leq T} \abs{  u_L(t)}_{X_{\eta,\mu,T}}^2+\int_0^T \abs{\nabla_{\rm h} u_L(t)}_{X_{\eta,\mu,T}}^2dt+\frac{1}{\sqrt{T}}\int_0^T \abs{ u_L(t)}_{Y_{\eta,\mu,T}}^2dt\leq C_{in},
\end{equation*}
where $C_{in}$ is a constant depending  on the Gevrey norm of $u_0$ specified in assumption \eqref{assume:small}.
A similar estimate holds for  $\partial_3 u_L$ since
$\norm{\partial_3u_0}_{G_{\frac{\rho_0}{2},\sigma,6}}  \leq C\norm{u_0}_{G_{\rho_0,\sigma,6}}.$
 This completes the proof of Proposition \ref{lem:uL}.

 \subsection{Proof of Proposition \ref{prop:v}: estimate on the nonlinear part}\label{subsec:nonlinear}

This subsection is devoted to the proof of Proposition \ref{prop:v} by establishing estimates for the nonlinear term $v=u-u_L$, which satisfies the initial-boundary problem \eqref{eq:v}.  It suffices to prove the following  two lemmas and then Proposition \ref{prop:v} will follow.

\begin{lemma}\label{st1}
Let $u_L$ and $v$ satisfy the heat equation \eqref{eq:uL} and  the initial-boundary problem \eqref{eq:v}, respectively.  Then, by possibly shrinking  the parameter $r$ from Theorem \ref{prop:wellposedness}, we obtain that,
for any $0<T\leq 1$,
\begin{multline*}
	\sup_{t\leq T}\abs{v(t)}_{X_{\eta,\mu,T}}^2+\int_0^T \abs{ \nabla_{\rm h}v(t)}_{X_{\eta,\mu,T}}^2dt +\frac{1}{\sqrt{T}}\int_0^T \abs{v(t)}_{Y_{\eta,\mu,T}}^2dt\\
	\leq  C \int_0^T e^{-\frac{2\psi(T)}{T}t} \mathcal R(t)dt,
\end{multline*}
where
\begin{equation}\label{defrt}
	\begin{aligned}
	\mathcal R(t)\stackrel{\rm def}{=}&   r^{-\frac{9}{2}}\abs{u}_{X_{\eta,\mu}}\Big( \abs{\nabla_{\rm h}u}_{X_{\eta,\mu}}+  \abs{v}_{X_{\eta,\mu}}+  \abs{\partial_3u_L}_{X_{\eta,\mu}}  \Big)\abs{v}_{X_{\eta,\mu}}\\
		& +     r^{-\frac{9}{2}}\Big(  \mu t^{\frac12}   \abs{u}_{X_{\eta,\mu}} \abs{u}_{Y_{\eta,\mu}} \abs{\nabla_{\rm h}v}_{X_{\eta,\mu}}+\mu^2t  \abs{u}_{X_{\eta,\mu}} \abs{\nabla_{\rm h}u}_{X_{\eta,\mu}}\abs{\nabla_{\rm h}v}_{X_{\eta,\mu}}\Big)   \\
		&+       r^{-\frac{9}{2}}   \abs{u}_{X_{\eta,\mu}}^2 \abs{\nabla_{\rm h}v}_{X_{\eta,\mu}}+   r^{-\frac{9}{2}}   \abs{u}_{X_{\eta,\mu}}\abs{u}_{Y_{\eta,\mu}}\abs{v}_{Y_{\eta,\mu}}.	
	\end{aligned}
\end{equation}
Recall that $\abs{\cdot}_{X_{\eta,\mu}}$ and $\abs{\cdot}_{Y_{\eta,\mu}}$ are given in Definition \ref{def:normx2} and  $\abs{\cdot}_{X_{\eta,\mu,T}}$ and $\abs{\cdot}_{Y_{\eta,\mu,T}}$ are given in Definition \ref{def:normx2+}.
\end{lemma}

\begin{proof}
Analogous  to \eqref{est:x2}, we derive from  \eqref{eq:v}
the following energy estimate for $v$: for any $\alpha\in\mathbb Z_+^3$ with $\abs\alpha\leq 3,$
\begin{equation}\label{nolipart}
	\begin{aligned}
& \frac{1}{2}\frac{d}{dt} \sum_{\stackrel{\beta \in \mathbb Z_+^2}{\beta_1\ge 1}} r^{\alpha_1}M_{\eta,\beta}^2\norm{D_\mu^\beta\partial^{\alpha} v}_{L^2}^2+ \sum_{\stackrel{\beta \in \mathbb Z_+^2}{\beta_1\ge 1}}r^{\alpha_1}M_{\eta,\beta}^2\norm{D_\mu^\beta\partial^{\alpha}\nabla_{\rm h} v}_{L^2}^2\\
&\qquad\qquad -\frac{\eta^{'}}{\eta}  \sum_{\stackrel{\beta \in \mathbb Z_+^2}{\beta_1\ge 1}}r^{\alpha_1}(\abs{\beta }+1)M_{\eta,\beta}^2\norm{D_\mu^\beta\partial^{\alpha} v}_{L^2}^2\\
&\leq   \sum_{\stackrel{\beta \in \mathbb Z_+^2}{\beta_1\ge 1}} r^{\alpha_1}   \beta_1\frac{\big(\mu (t)t^{\frac{1}{2}}\big)'}{\mu(t) t^{\frac{1}{2}}}   M_{\eta,\beta}^2\norm{D_\mu^\beta \partial^{\alpha} v}_{L^2}^2+  \sum_{\ell=1}^4 \tilde J_{\ell},
\end{aligned}
\end{equation}
where
\begin{equation}
	\label{H1H4}
	  \left\{
    \begin{aligned}
&\tilde J_1= - \sum_{\stackrel{\beta \in \mathbb Z_+^2}{\beta_1\ge 1}}r^{\alpha_1}M_{\eta,\beta}^2  \inner{  D_\mu^\beta\partial^{\alpha}\big ((u_{\rm h}\cdot\nabla_{\rm h}) u\big ),\   D_\mu^\beta\partial^{\alpha} v}_{L^2},\\
&\tilde J_2=- \sum_{\stackrel{\beta \in \mathbb Z_+^2}{\beta_1\ge 1}}r^{\alpha_1}M_{\eta,\beta}^2\inner{ D_\mu^\beta\partial^{\alpha} (u_3\partial_3u),\   D_\mu^\beta\partial^{\alpha} v}_{L^2},\\
&\tilde  J_3=  \sum_{\stackrel{\beta \in \mathbb Z_+^2}{\beta_1\ge 1}}r^{\alpha_1}M_{\eta,\beta}^2\inner{  D_\mu^\beta\partial^{\tilde \alpha}\nabla p, \   D_\mu^\beta\partial^{\alpha}\partial_1 v}_{L^2},\\
&\tilde  J_4=  \sum_{\stackrel{\beta \in \mathbb Z_+^2}{\beta_1\ge 1}}r^{\alpha_1}M_{\eta,\beta}^2\int_{\mathbb R^2} \big[ (  D_\mu^\beta\partial^{\alpha}v)  D_\mu^\beta\partial^{\tilde \alpha} \big(\partial_tv+u\cdot\nabla u-\partial_2^2v\big)\big]\big|^{x_1=1}_{x_1=0}\ dx_2dx_3,
 \end{aligned}
    \right.
\end{equation}
 recalling $\tilde\alpha=\alpha-(1,0,0).$ We  will estimate each    $\tilde J_\ell,$   by applying the argument used in the proofs of  Lemmas \ref{lem:J1}-\ref{lem:J34}, with minor modifications  where necessary.

 The treatment of $\tilde J_1$ is similar to those of $S_1$ and $J_1$ in Lemmas \ref{lem:S1} and \ref{lem:J1}, and we have
\begin{equation*}
	\tilde J_1\leq Cr^{-\frac{9}{2}}\abs{u}_{X_{\eta,\mu}} \abs{\nabla_{\rm h}u}_{X_{\eta,\mu}}\abs{v}_{X_{\eta,\mu}}.
\end{equation*}
To estimate $\tilde J_2$ in \eqref{H1H4}, we use the fact that $u=u_L+v$ to rewrite
\begin{multline*}
  \tilde J_2=  - \sum_{\stackrel{\beta \in \mathbb Z_+^2}{\beta_1\ge 1}}r^{\alpha_1}M_{\eta,\beta}^2\inner{ D_\mu^\beta\partial^{\alpha} (u_3\partial_3u_L),\   D_\mu^\beta\partial^{\alpha} v}_{L^2}\\
  - \sum_{\stackrel{\beta \in \mathbb Z_+^2}{\beta_1\ge 1}}r^{\alpha_1}M_{\eta,\beta}^2\inner{ D_\mu^\beta\partial^{\alpha} (u_3\partial_3v),\   D_\mu^\beta\partial^{\alpha} v}_{L^2}.
\end{multline*}
Consequently,  using a similar split as that in \eqref{+J2} and then
repeating   the proofs of  estimates \eqref{est:J21}, \eqref{est:J22} and \eqref{est:J23} with minor  modifications,  we obtain
\begin{equation*}
\begin{aligned}
		&\tilde J_2\leq   Cr^{-\frac{9}{2}}  \abs{u}_{X_{\eta,\mu}}\big(\abs{\partial_3u_L}_{X_{\eta,\mu}}+\abs{v}_{X_{\eta,\mu}}\big) \abs{v}_{X_{\eta,\mu}}\\
		 &\quad\ + Cr^{-\frac{9}{2}} \mu t^{\frac12}   \abs{u}_{X_{\eta,\mu}} \abs{u}_{Y_{\eta,\mu}} \abs{\nabla_{\rm h}v}_{X_{\eta,\mu}}+Cr^{-\frac92} \mu^2t  \abs{u}_{X_{\eta,\mu}} \abs{\nabla_{\rm h}u}_{X_{\eta,\mu}}\abs{\nabla_{\rm h}v}_{X_{\eta,\mu}}.
	\end{aligned}
	\end{equation*}
For  $\tilde  J_3$, we first write
\begin{equation}\label{tij3}
		\tilde  J_3
		\leq C   r^{\frac{ \alpha_1}{2}} \abs{\nabla_{\rm h} v}_{X_{\eta,\mu}}\Big(\sum_{\stackrel{\beta \in \mathbb Z_+^2}{\beta_1\ge 1}}M^2_{\eta,\beta}\norm{D_\mu^\beta\partial^{ \tilde \alpha}  \nabla p}_{L^2}^2\Big)^{\frac12}.
	\end{equation}
On the other hand,  similar to \eqref{p+},
\begin{equation}\label{+p+}
	\begin{aligned}
		& \sum_{\stackrel{\beta \in \mathbb Z_+^2}{\beta_1\ge 1}}M^2_{\eta,\beta}\norm{D_\mu^\beta\partial^{ \tilde \alpha}  \nabla p}_{L^2}^2 \leq  \sum_{\stackrel{\beta \in \mathbb Z_+^2}{\beta_1\ge 1}}M^2_{\eta,\beta}\norm{D_\mu^\beta\partial^{\tilde{\alpha}} \big (u\cdot\nabla) u\big)}_{L^2}\norm{D_\mu^\beta\partial^{\tilde{\alpha}} \nabla p}_{L^2}\\
		&\qquad\qquad\qquad+\sum_{\stackrel{\beta \in \mathbb Z_+^2}{\beta_1\ge 1}}M^2_{\eta,\beta}\bigg|\int_{\mathbb R^2}   (D_\mu^\beta\partial^{\tilde{\alpha}} \partial_tv_1)(D_\mu^\beta\partial^{\tilde  \alpha} p)\big|^{x_1=1}_{x_1=0} \  dx_2dx_3\bigg|\\
		&\qquad\qquad\qquad+\sum_{\stackrel{\beta \in \mathbb Z_+^2}{\beta_1\ge 1}}M^2_{\eta,\beta}\bigg|\int_{\mathbb R^2}   (D_\mu^\beta\partial^{\tilde{\alpha}} \Delta_{\rm h} v_1)(D_\mu^\beta\partial^{\tilde  \alpha} p)\big|^{x_1=1}_{x_1=0} \  dx_2dx_3\bigg|.
	\end{aligned}
\end{equation}
By repeating the proofs of estimates \eqref{est:tildeS1}, \eqref{df} and \eqref{est:tildeS3}, we can verify the terms on the right-hand side of \eqref{+p+} are bounded from above by
\begin{equation*}	
 C\Big(r^{-\frac92}\abs{u}_{X_{\eta,\mu}}^2+r^{-\frac{2\alpha_1-1}{4}} \abs{\nabla_{\rm h} v}_{X_{\eta,\mu}} \Big) \Big(\sum_{\stackrel{\beta \in \mathbb Z_+^2}{\beta_1\ge 1}}M^2_{\eta,\beta}\norm{D_\mu^\beta\partial^{ \tilde \alpha}  \nabla p}_{L^2}^2\Big)^{\frac12}.
\end{equation*}
As a result,
\begin{equation*}
	 \Big(\sum_{\stackrel{\beta \in \mathbb Z_+^2}{\beta_1\ge 1}}M^2_{\eta,\beta}\norm{D_\mu^\beta\partial^{ \tilde \alpha}  \nabla p}_{L^2}^2\Big)^{\frac12}\leq  Cr^{-\frac92}\abs{u}_{X_{\eta,\mu}}^2  +C r^{-\frac{2\alpha_1-1}{4}} \abs{\nabla_{\rm h} v}_{X_{\eta,\mu}},
	 \end{equation*}
	 which with \eqref{tij3} yields
	 \begin{equation*}
	 	\begin{aligned}
	 		\tilde J_3 &\leq  Cr^{-\frac92}  \abs{u}_{X_{\eta,\mu}}^2 \abs{\nabla_{\rm h}v}_{X_{\eta,\mu}}+ C r^{\frac{1}{4}} \abs{\nabla_{\rm h}v}_{X_{\eta,\mu}}^2.
	 	\end{aligned}
	 \end{equation*}
	 Repeating  the proof of Lemma \ref{lem:S4} with minor modifications,  we obtain
	 \begin{equation*}
	 	\tilde J_4  \leq  Cr^{-\frac92}   \abs{u}_{X_{\eta,\mu}}^2 \abs{\nabla_{\rm h}v}_{X_{\eta,\mu}}+ C  r^{\frac{1}{4}} \abs{\nabla_{\rm h}v}_{X_{\eta,\mu}}^2.
	 \end{equation*}
	 In summary,  combining these estimates above gives
	 \begin{equation*}
	\begin{aligned}
	 &\sum_{\ell=1}^4\tilde J_\ell  \leq  Cr^{-\frac{9}{2}}\abs{u}_{X_{\eta,\mu}} \Big(\abs{\nabla_{\rm h}u}_{X_{\eta,\mu}}+ \abs{v}_{X_{\eta,\mu}} +  \abs{\partial_3u_L}_{X_{\eta,\mu}} \Big)\abs{v}_{X_{\eta,\mu}}\\
		&\qquad + Cr^{-\frac{9}{2}} \Big(\mu t^{\frac12}   \abs{u}_{X_{\eta,\mu}} \abs{u}_{Y_{\eta,\mu}} \abs{\nabla_{\rm h}v}_{X_{\eta,\mu}}+\mu^2t  \abs{u}_{X_{\eta,\mu}} \abs{\nabla_{\rm h}u}_{X_{\eta,\mu}}\abs{\nabla_{\rm h}v}_{X_{\eta,\mu}} \Big) \\
		&\qquad+  Cr^{-\frac{9}{2}} \abs{u}_{X_{\eta,\mu}}^2 \abs{\nabla_{\rm h}v}_{X_{\eta,\mu}}  + C  r^{\frac{1}{4}} \abs{\nabla_{\rm h}v}_{X_{\eta,\mu}}^2\\
		&\leq  C  \mathcal R(t)+ C  r^{\frac{1}{4}} \abs{\nabla_{\rm h}v}_{X_{\eta,\mu}}^2,
		 \end{aligned}
	\end{equation*}
	where $\mathcal R(t)$ is defined as in \eqref{defrt}.
	Substituting the above estimate into \eqref{nolipart}, we obtain
 \begin{equation*}%\label{les}
	\begin{aligned}
& \frac{1}{2}\frac{d}{dt} \sum_{\stackrel{\beta \in \mathbb Z_+^2}{\beta_1\ge 1}} r^{\alpha_1}M_{\eta,\beta}^2\norm{D_\mu^\beta\partial^{\alpha} v}_{L^2}^2+ \sum_{\stackrel{\beta \in \mathbb Z_+^2}{\beta_1\ge 1}}r^{\alpha_1}M_{\eta,\beta}^2\norm{D_\mu^\beta\partial^{\alpha}\nabla_{\rm h} v}_{L^2}^2\\
 &\quad-\frac{\eta^{'}}{\eta}  \sum_{\stackrel{\beta \in \mathbb Z_+^2}{\beta_1\ge 1}}r^{\alpha_1}(\abs{\beta }+1)M_{\eta,\beta}^2\norm{D_\mu^\beta\partial^{\alpha} v}_{L^2}^2 \\
&\leq     \sum_{\stackrel{\beta \in \mathbb Z_+^2}{\beta_1\ge 1}} r^{\alpha_1}   \beta_1\frac{\big(\mu (t)t^{\frac{1}{2}}\big)'}{\mu(t) t^{\frac{1}{2}}}   M_{\eta,\beta}^2\norm{D_\mu^\beta \partial^{\alpha} v}_{L^2}^2+C \mathcal R(t)+   C  r^{\frac{1}{4}} \abs{\nabla_{\rm h}v}_{X_{\eta,\mu}}^2\\
&\leq   C \mathcal R(t)+  \Big(\frac14+C  r^{\frac{1}{4}} \Big)\abs{\nabla_{\rm h}v}_{X_{\eta,\mu}}^2+\frac{\psi(T)}{T}\abs{v}_{X_{\eta,\mu}}^2,
\end{aligned}
\end{equation*}
where the last inequality follows from \eqref{demu}.
Moreover, to derive the estimate for the case $\beta_1=0,$ we adapt the proof of  Proposition \ref{purespatial}, with the key modification being the replacement of the term
$ r^{-\frac92}\abs{u}_{X_{\eta,\mu}}\abs{u}_{Y_{\eta,\mu}}^2$ by
\begin{equation*}
r^{-\frac92}\abs{u}_{X_{\eta,\mu}}\abs{u}_{Y_{\eta,\mu}}\abs{v}_{Y_{\eta,\mu}}.
\end{equation*}
Then, from the definition \eqref{defrt} of $\mathcal{R}(t)$, we observe that
\begin{equation*}
r^{-\frac92}\abs{u}_{X_{\eta,\mu}}\abs{u}_{Y_{\eta,\mu}}\abs{v}_{Y_{\eta,\mu}}\leq \mathcal{ R}(t),
\end{equation*}
 which implies  that, for any $\alpha\in\mathbb Z_+^3$ with $\abs\alpha\leq 3,$
 \begin{multline*}
 \frac{1}{2}\frac{d}{dt} \sum_{\stackrel{\beta \in \mathbb Z_+^2}{\beta_1=0}} r^{\alpha_1}M_{\eta,\beta}^2\norm{D_\mu^\beta\partial^{\alpha} v}_{L^2}^2+ \sum_{\stackrel{\beta \in \mathbb Z_+^2}{\beta_1=0}}r^{\alpha_1}M_{\eta,\beta}^2\norm{D_\mu^\beta\partial^{\alpha}\nabla_{\rm h} v}_{L^2}^2\\
 -\frac{\eta^{'}}{\eta}  \sum_{\stackrel{\beta \in \mathbb Z_+^2}{\beta_1=0}}r^{\alpha_1}(\abs{\beta }+1)M_{\eta,\beta}^2\norm{D_\mu^\beta\partial^{\alpha} v}_{L^2}^2  \leq   C \mathcal R(t)+   C  r^{\frac{1}{4}}  \abs{\nabla_{\rm h}v}_{X_{\eta,\mu}}^2.
\end{multline*}
Hence, for any $\alpha\in\mathbb Z_+^3$ with $\abs{\alpha}\ge 3$, we have
 \begin{equation*}
	\begin{aligned}
& \frac{1}{2}\frac{d}{dt} \sum_{ \beta \in \mathbb Z_+^2 } r^{\alpha_1}M_{\eta,\beta}^2\norm{D_\mu^\beta\partial^{\alpha} v}_{L^2}^2+ \sum_{ \beta \in \mathbb Z_+^2 }r^{\alpha_1}M_{\eta,\beta}^2\norm{D_\mu^\beta\partial^{\alpha}\nabla_{\rm h} v}_{L^2}^2\\
 &\quad-\frac{\eta^{'}}{\eta}  \sum_{ \beta \in \mathbb Z_+^2 }r^{\alpha_1}(\abs{\beta }+1)M_{\eta,\beta}^2\norm{D_\mu^\beta\partial^{\alpha} v}_{L^2}^2 \\
&\leq  C  \mathcal R(t)+  \Big(\frac14+C  r^{\frac{1}{4}} \Big)\abs{\nabla_{\rm h}v}_{X_{\eta,\mu}}^2+\frac{\psi(T)}{T}\abs{v}_{X_{\eta,\mu}}^2.
\end{aligned}
\end{equation*}
Using \eqref{prpalge}, we can extend the validity of the above estimate  to general indices $(j,\alpha)\in\mathbb{Z}_+\times\mathbb{Z}_+^3$ satisfying $j+|\alpha|\leq 3$.  As a result,
 recalling the norms  $ \abs{\cdot}_{X_{\eta,\mu}}$ and $ \abs{\cdot}_{Y_{\eta,\mu}}$ are given in Definition \ref{def:normx2}, we obtain
\begin{multline*}
	\frac{1}{2}\frac{d}{dt} \abs{v}_{X_{\eta,\mu}}^2+\abs{ \nabla_{\rm h}v}_{X_{\eta,\mu}}^2 -\frac{\eta'}{\eta}\abs{v}_{Y_{\eta,\mu}}^2\\
	\leq C  \mathcal R(t)+  \Big(\frac14+C  r^{\frac{1}{4}} \Big)\abs{\nabla_{\rm h}v}_{X_{\eta,\mu}}^2+\frac{\psi(T)}{T}\abs{v}_{X_{\eta,\mu}}^2,
\end{multline*}
and  thus, shrinking $r$ if necessary such that $Cr^{\frac14}\leq \frac14,$
\begin{equation*}
	\frac{1}{2} \frac{d}{dt} \abs{v}_{X_{\eta,\mu}}^2+\frac12 \abs{ \nabla_{\rm h}v}_{X_{\eta,\mu}}^2 -\frac{\eta'}{\eta}\abs{v}_{Y_{\eta,\mu}}^2
	\leq  \frac{\psi(T)}{T}  \abs{v}_{X_{\eta,\mu}}^2+ C \mathcal R(t).
\end{equation*}
This, with the relation $\abs{v}_{X_{\eta,\mu,T}}=e^{-\frac{\psi(T)}{T}t}\abs{v}_{X_{\eta,\mu}},$ implies
\begin{equation*}
	\frac{1}{2} \frac{d}{dt} \abs{v}_{X_{\eta,\mu,T}}^2+\frac12  \abs{ \nabla_{\rm h}v}_{X_{\eta,\mu,T}}^2 -\frac{\eta'}{\eta}  \abs{v}_{Y_{\eta,\mu,T}}^2
	\leq  C  e^{-\frac{2\psi(T)}{T}t} \mathcal R(t).
\end{equation*}
We integrate the above inequality over $[0,T]$ with $T\leq 1$ and
use   \eqref{bdrho} along with the fact that
\begin{equation*}
\lim_{t\rightarrow 0}\abs{v}_{X_{\eta,\mu,T}}^2=	\lim_{t\rightarrow 0} e^{-\frac{2\psi(T)}{T}t}\abs{v}_{X_{\eta,\mu}}^2=\lim_{t\rightarrow 0}  \abs{v}_{X_{\eta,\mu}}^2=0,
\end{equation*}
which follows from  \eqref{shotbeha} and  the initial condition $v|_{t=0}=0$; this gives
\begin{equation*}
	\begin{aligned}
	&	\sup_{t\leq T}\abs{v}_{X_{\eta,\mu,T}}^2+\int_0^T \abs{ \nabla_{\rm h}v}_{X_{\eta,\mu,T}}^2dt +\frac{1}{\sqrt{T}}\int_0^T \abs{v}_{Y_{\eta,\mu,T}}^2dt
	\leq  C\int_0^T  e^{-\frac{2\psi(T)}{T}t} \mathcal Rdt.
	\end{aligned}
\end{equation*}
The proof of Lemma \ref{st1} is thus completed.
\end{proof}

 \begin{lemma}\label{cor-tildej}
 	Let $\mathcal R(t)$ be given in \eqref{defrt}.
 Then  for any $T\leq 1$ and any $0<\epsilon<1$ we  have
 	\begin{equation*}
	\begin{aligned}
&\int_0^T e^{-\frac{2\psi(T)}{T}t}   \mathcal R(t) dt 	\\
&\leq \epsilon    \bigg[ \int_0^T\abs{\nabla_{\rm h}v}_{X_{\eta,\mu,T}}^2dt+\frac{1}{\sqrt{T}}\int^T_0\abs{v}_{Y_{\eta,\mu,T}}^2dt\bigg]+C\epsilon^{-1}r^{-9}e^{3\psi(T)}\sqrt{T}\\
&\quad +  C\epsilon^{-1}r^{-9}e^{3 \psi(T)} \sqrt{T}\Big( 1+\sup_{ t\leq T}\abs{v}_{X_{\eta,\mu,T}}^2\Big)\bigg[\sup_{ t\leq T}\abs{v}_{X_{\eta,\mu,T}}^2+ \int_0^T   \abs{\nabla_{\rm h}v}_{X_{\eta,\mu,T}}^2  dt\bigg]\\
  &\quad+ C\epsilon^{-1}r^{-9} e^{3\psi(T)} \sqrt{T}\Big(1 +  \sup_{ t\leq T} \abs{v}_{X_{\eta,\mu,T}}^2 \Big) \frac{1}{\sqrt{T}} \int_0^T       \abs{v}_{Y_{\eta,\mu,T}}^2  dt.
	\end{aligned}
\end{equation*}
Recall that the norms $\abs{\cdot}_{X_{\eta,\mu,T}}$ and $\abs{\cdot}_{Y_{\eta,\mu,T}}$ are given in Definition \ref{def:normx2+}.
 \end{lemma}

 \begin{proof} We recall the relation  $\abs{\cdot}_{X_{\eta,\mu,T}}=e^{-\frac{\psi(T)}{T}t}\abs{\cdot}_{X_{\eta,\mu}}$ and
 \begin{equation*}
     \begin{aligned}
	\mathcal R(t)= &  r^{-\frac{9}{2}}\abs{u}_{X_{\eta,\mu}} \Big(\abs{\nabla_{\rm h}u}_{X_{\eta,\mu}}+  \abs{v}_{X_{\eta,\mu}}+  \abs{\partial_3u_L}_{X_{\eta,\mu}}  \Big)\abs{v}_{X_{\eta,\mu}}\\
	& +     r^{-\frac{9}{2}}\Big(  \mu t^{\frac12}   \abs{u}_{X_{\eta,\mu}} \abs{u}_{Y_{\eta,\mu}} \abs{\nabla_{\rm h}v}_{X_{\eta,\mu}}+\mu^2t  \abs{u}_{X_{\eta,\mu}} \abs{\nabla_{\rm h}u}_{X_{\eta,\mu}}\abs{\nabla_{\rm h}v}_{X_{\eta,\mu}}\Big)   \\
		&+       r^{-\frac{9}{2}}   \abs{u}_{X_{\eta,\mu}}^2 \abs{\nabla_{\rm h}v}_{X_{\eta,\mu}}+r^{-\frac92}\abs{u}_{X_{\eta,\mu}}\abs{u}_{Y_{\eta,\mu}}\abs{v}_{Y_{\eta,\mu}}.	
	\end{aligned}
 \end{equation*}
 Then for  $t\leq T$, we have
 	\begin{equation*}
	\begin{aligned}
 &e^{-\frac{2\psi(T)}{T}t}\mathcal R(t)\\
&\leq  r^{-\frac{9}{2}}e^{ \psi(T)}\abs{u}_{X_{\eta,\mu,T}} \Big(\abs{\nabla_{\rm h}u}_{X_{\eta,\mu,T}}+\abs{v}_{X_{\eta,\mu,T}}+ \abs{\partial_3u_L}_{X_{\eta,\mu,T}}\Big)\abs{v}_{X_{\eta,\mu,T}}\\
&\quad +   r^{-\frac{9}{2}}e^{ \psi(T)}\abs{u}_{X_{\eta,\mu,T}}\Big(  \mu  t^{\frac12}    \abs{u}_{Y_{\eta,\mu,T}}   +\mu^2t   \abs{\nabla_{\rm h}u}_{X_{\eta,\mu,T}}\Big)\abs{\nabla_{\rm h}v}_{X_{\eta,\mu,T}} \\
&\quad+   r^{-\frac{9}{2}} e^{ \psi(T)} \Big(   \abs{u}_{X_{\eta,\mu,T}}^2 \abs{\nabla_{\rm h}v}_{X_{\eta,\mu,T}}+\abs{u}_{X_{\eta,\mu,T}} \abs{u}_{Y_{\eta,\mu,T}}\abs{v}_{Y_{\eta,\mu,T}}\Big).
		 \end{aligned}
	\end{equation*}
	So that
	\begin{equation}\label{i1i2i3}
		\begin{aligned}
			  \int_0^T  e^{-\frac{2\psi(T)}{T}t} \mathcal R(t)dt \leq  Cr^{-\frac{9}{2}}e^{ \psi(T)}\Big(I_1+I_2+I_3\big).
		\end{aligned}
	\end{equation}
Here
	\begin{equation}\label{i1i3}
	\left\{
		\begin{aligned}
			I_1&=  \int_0^T\abs{u}_{X_{\eta,\mu,T}} \Big(\abs{\nabla_{\rm h}u}_{X_{\eta,\mu,T}}+\abs{v}_{X_{\eta,\mu,T}}+ \abs{\partial_3u_L}_{X_{\eta,\mu,T}}\Big)\abs{v}_{X_{\eta,\mu,T}}dt,\\
				I_2&=  \int_0^T \abs{u}_{X_{\eta,\mu,T}}\Big(  \mu  t^{\frac12}    \abs{u}_{Y_{\eta,\mu,T}}   +\mu^2t   \abs{\nabla_{\rm h}u}_{X_{\eta,\mu,T}}\Big)\abs{\nabla_{\rm h}v}_{X_{\eta,\mu,T}} dt,\\
	I_3&= \int_0^T \Big(   \abs{u}_{X_{\eta,\mu,T}}^2 \abs{\nabla_{\rm h}v}_{X_{\eta,\mu,T}}+\abs{u}_{X_{\eta,\mu,T}} \abs{u}_{Y_{\eta,\mu,T}}\abs{v}_{Y_{\eta,\mu,T}}\Big)  dt.
		\end{aligned}
		\right.
	\end{equation}
 Next, we proceed to estimate $I_k$ for $1\leq k\leq 3$ in the following three steps. To this end, we first list some estimates that will be frequently used later.    Recalling the decomposition  $u=u_L+v,$ we deduce from Proposition \ref{lem:uL} that
\begin{equation}
	\label{uuv}
	\left\{
	\begin{aligned}
	&\forall\ 0\leq t\leq T,\quad\abs{ u}_{X_{\eta,\mu,T}}\leq C+\abs{ v}_{X_{\eta,\mu,T}},\\
	&\int^T_0\abs{ \nabla_{\rm h} u}_{X_{\eta,\mu,T}}^2dt\leq C+2\int^T_0\abs{ \nabla_{\rm h} v}_{X_{\eta,\mu,T}}^2dt,\\
	&\int^T_0\abs{ u}_{Y_{\eta,\mu,T}}^2dt\leq C\sqrt{T}+2\int^T_0\abs{ v}_{Y_{\eta,\mu,T}}^2dt.
		\end{aligned}
\right.
\end{equation}

 {\it Step 1.}  We first estimate $I_1$ in \eqref{i1i3} and prove  that for any $0<\epsilon<1$,
  \begin{equation}\label{esoni1}
  I_1	\leq \epsilon \int^T_0\abs{\nabla_{\rm h}v}_{X_{\eta,\mu,T}}^2dt+C\epsilon+C\epsilon^{-1}T+C\epsilon^{-1}T\Big( 1+\sup_{ t\leq T}\abs{v}_{X_{\eta,\mu,T}}^2\Big)\sup_{ t\leq T}\abs{v}_{X_{\eta,\mu,T}}^2.
 \end{equation}
 In fact, applying \eqref{uuv} gives, for any $0<\epsilon<1,$
\begin{equation*}
    \begin{aligned}
 &  \int^T_0\abs{u}_{X_{\eta,\mu,T}} \abs{\nabla_{\rm h}u}_{X_{\eta,\mu,T}}\abs{v}_{X_{\eta,\mu,T}}dt\\
 &\leq \epsilon \int^T_0\abs{\nabla_{\rm h}u}_{X_{\eta,\mu,T}}^2dt+   \epsilon^{-1}\int^T_0\abs{u}_{X_{\eta,\mu,T}}^2 \abs{v}_{X_{\eta,\mu,T}}^2 dt\\
 &\leq \epsilon \int^T_0\abs{\nabla_{\rm h}v}_{X_{\eta,\mu,T}}^2dt+C\epsilon +C\epsilon^{-1}T+C\epsilon^{-1}T\sup_{ t\leq T} \abs{v}_{X_{\eta,\mu,T}}^4.
    \end{aligned}
\end{equation*}
On the other hand, using  \eqref{uuv} and Proposition \ref{lem:uL}, we obtain
\begin{multline*}
\int_0^T\abs{u}_{X_{\eta,\mu,T}} \Big(\abs{v}_{X_{\eta,\mu,T}}+ \abs{\partial_3u_L}_{X_{\eta,\mu,T}}\Big)\abs{v}_{X_{\eta,\mu,T}}dt\\
\leq C \int_0^T\big(1+\abs{v}_{X_{\eta,\mu,T}}^4\big)dt
		 \leq    C  T+C T \sup_{ t\leq T} \abs{v}_{X_{\eta,\mu,T}}^4.
\end{multline*}
Consequently, combining the two estimates above yields   assertion \eqref{esoni1}.

 {\it Step 2.}  In this step we will estimate $I_2$ in \eqref{i1i3}  and prove that for any $0<\epsilon<1$,
 \begin{equation}\label{eon2}
 	\begin{aligned}
 	I_2
    & \leq    \epsilon  \int_0^T\abs{\nabla_{\rm h}v}_{X_{\eta,\mu,T}}^2dt+C\epsilon^{-1}e^{\psi(T)} T+C\epsilon^{-1}e^{\psi(T)} T  \sup_{ t\leq T}\abs{v}_{X_{\eta,\mu,T}}^2\\
    &\quad+ C\epsilon^{-1} e^{\psi(T)} T\Big(1 +  \sup_{ t\leq T} \abs{v}_{X_{\eta,\mu,T}}^2 \Big) \int_0^T  \bigg(  \abs{\nabla_{\rm h}v}_{X_{\eta,\mu,T}}^2 + \frac{ \abs{v}_{Y_{\eta,\mu,T}}^2}{\sqrt{T}}    \bigg)dt.
 	\end{aligned}
 \end{equation}
Observe $|\mu(t)|\leq C \sqrt{\psi(T)}$ for $t\in[0,T].$ Thus
\begin{equation*}
	\begin{aligned}
	&   \int_0^T \mu(t) t^{\frac12}   \abs{u}_{X_{\eta,\mu,T}} \abs{u}_{Y_{\eta,\mu,T}} \abs{\nabla_{\rm h}v}_{X_{\eta,\mu,T}}  dt	\\
	&\leq \epsilon  \int_0^T\abs{\nabla_{\rm h}v}_{X_{\eta,\mu,T}}^2dt  +  C\epsilon^{-1} \psi(T) T\int_0^T   \abs{u}_{X_{\eta,\mu,T}}^2 \abs{u}_{Y_{\eta,\mu,T}}^2  dt.	
	\end{aligned}
\end{equation*}
For the last term on the right-hand side, we use   \eqref{uuv} and $0<T\leq 1$ to obtain that
\begin{equation*}
	\begin{aligned}
	  &\int_0^T   \abs{u}_{X_{\eta,\mu,T}}^2 \abs{u}_{Y_{\eta,\mu,T}}^2 dt \leq C \Big(1 +  \sup_{ t\leq T} \abs{v}_{X_{\eta,\mu,T}}^2 \Big)\int_0^T     \abs{u}_{Y_{\eta,\mu,T}}^2 dt \\
     &\leq C\Big(1 +  \sup_{ t\leq T} \abs{v}_{X_{\eta,\mu,T}}^2 \Big) \Big(\sqrt{T}+\int_0^T    \abs{v}_{Y_{\eta,\mu,T}}^2dt\Big)\\
     &\leq C \Big(1 +  \sup_{ t\leq T} \abs{v}_{X_{\eta,\mu,T}}^2 \Big)+ C\sqrt{T}\Big(1 +  \sup_{ t\leq T} \abs{v}_{X_{\eta,\mu,T}}^2 \Big)\frac{1}{\sqrt{T}}  \int_0^T    \abs{v}_{Y_{\eta,\mu,T}}^2dt.
	\end{aligned}
\end{equation*}
As the result, combining the two estimates above gives
 \begin{align*}
	  & \int_0^T \mu(t) t^{\frac12}   \abs{u}_{X_{\eta,\mu,T}} \abs{u}_{Y_{\eta,\mu,T}} \abs{\nabla_{\rm h}v}_{X_{\eta,\mu,T}}  dt	\\
	& \leq   \epsilon  \int_0^T\abs{\nabla_{\rm h}v}_{X_{\eta,\mu,T}}^2dt + C\epsilon^{-1}\psi(T)T+C\epsilon^{-1} \psi(T)T \sup_{ t\leq T}\abs{v}_{X_{\eta,\mu,T}}^2\\
   &\quad + C\epsilon^{-1} \psi(T) T\Big(1 +  \sup_{ t\leq T} \abs{v}_{X_{\eta,\mu,T}}^2 \Big) \frac{1}{\sqrt{T}} \int_0^T    \abs{v}_{Y_{\eta,\mu,T}}^2dt.
	\end{align*}
Using   \eqref{uuv}  and the fact $|\mu(t)|^2 \leq C \psi(T)$ for $t\in[0,T],$ we have
\begin{equation*}
	\begin{aligned}
	& \int_0^T \mu(t)^2t  \abs{u}_{X_{\eta,\mu,T}} \abs{\nabla_{\rm h}u}_{X_{\eta,\mu,T}}\abs{\nabla_{\rm h}v}_{X_{\eta,\mu,T}} dt \\
	&\leq C\psi(T)T\Big (1+\sup_{t\leq T} \abs{v}_{X_{\eta,\mu,T}}\Big ) \int_0^T \Big( \abs{\nabla_{\rm h}u}_{X_{\eta,\mu,T}}^2+\abs{\nabla_{\rm h}v}_{X_{\eta,\mu,T}}^2\Big)dt	\\
	&\leq C\psi(T)T\Big (1+\sup_{t\leq T} \abs{v}_{X_{\eta,\mu,T}}^2\Big ) +C\psi(T)T\Big (1+\sup_{t\leq T} \abs{v}_{X_{\eta,\mu,T}}^2\Big )     \int_0^T  \abs{\nabla_{\rm h}v}_{X_{\eta,\mu,T}}^2 dt .
	\end{aligned}
\end{equation*}
Combining the two estimates above and using the inequality  $\psi(T)\leq Ce^{\psi(T)}$, we obtain assertion \eqref{eon2}.

{\it Step 3.} It remains to estimate $I_3$ in \eqref{i1i3} and we will derive that, for any $0<\epsilon<1,$
 \begin{equation}\label{eon3}
 	\begin{aligned}
 	I_3
    &\leq  \epsilon\bigg[ \int_0^T\abs{\nabla_{\rm h}v}_{X_{\eta,\mu,T}}^2dt+\frac{1}{\sqrt{T}}\int^T_0\abs{v}_{Y_{\eta,\mu,T}}^2dt\bigg]+C\epsilon^{-1}T\\
   &\quad+C\epsilon^{-1}T \sup_{ t\leq T}\abs{v}_{X_{\eta,\mu,T}}^4  + C\epsilon^{-1}T\Big(1 + \sup_{ t\leq T} \abs{v}_{X_{\eta,\mu,T}}^2\Big) \frac{1}{\sqrt{T}}\int^T_0\abs{v}_{Y_{\eta,\mu,T}}^2dt.
 	\end{aligned}
 \end{equation}
 Using   \eqref{uuv}, we obtain, for any $0<\epsilon<1,$
 \begin{multline*}%\label{uhu}
    \int_0^T   \abs{u}_{X_{\eta,\mu,T}}^2 \abs{\nabla_{\rm h}v}_{X_{\eta,\mu,T}} dt \leq  \epsilon \int_0^T\abs{\nabla_{\rm h}v}_{X_{\eta,\mu,T}}^2dt+  \epsilon^{-1}\int_0^T   \abs{u}_{X_{\eta,\mu,T}}^4dt \\
		 \leq \epsilon   \int_0^T\abs{\nabla_{\rm h}v}_{X_{\eta,\mu,T}}^2dt+ C\epsilon^{-1} T +C\epsilon^{-1}  T  \sup_{ t\leq T}\abs{v}_{X_{\eta,\mu,T}}^4,
 	\end{multline*}
and
\begin{equation*}
    \begin{aligned}
&\int_0^T \abs{u}_{X_{\eta,\mu,T}} \abs{u}_{Y_{\eta,\mu,T}}\abs{v}_{Y_{\eta,\mu,T}} dt\\
&\leq \frac{\epsilon}{\sqrt{T}}\int^T_0\abs{v}_{Y_{\eta,\mu,T}}^2dt+C\epsilon^{-1}\sqrt{T}\Big(1 +  \sup_{ t\leq T} \abs{v}_{X_{\eta,\mu,T}}^2 \Big) \int^T_0\abs{u}_{Y_{\eta,\mu,T}}^2dt\\
&\leq \frac{\epsilon}{\sqrt{T}}\int^T_0\abs{v}_{Y_{\eta,\mu,T}}^2dt+ C\epsilon^{-1}T\Big(1 +  \sup_{ t\leq T} \abs{v}_{X_{\eta,\mu,T}}^2\Big)\\
&\quad + C\epsilon^{-1}T\Big(1 + \sup_{ t\leq T} \abs{v}_{X_{\eta,\mu,T}}^2\Big) \frac{1}{\sqrt{T}}\int^T_0\abs{v}_{Y_{\eta,\mu,T}}^2dt.
    \end{aligned}
\end{equation*}
Then the desired estimate \eqref{eon3} follows by combining the two estimates above.

Finally, by substituting estimates \eqref{esoni1}, \eqref{eon2} and \eqref{eon3} into \eqref{i1i2i3}, we  conclude that  for any $0<\epsilon<1,$
\begin{equation*}
    \begin{aligned}
&\int_0^T e^{-\frac{2\psi(T)}{T}t}\mathcal R(t)dt \leq  Cr^{-\frac{9}{2}}e^{ \psi(T)}\Big(I_1+I_2+I_3\big)\\
&\leq \epsilon r^{-\frac92} e^{\psi(T)}   \bigg[ \int_0^T\abs{\nabla_{\rm h}v}_{X_{\eta,\mu,T}}^2dt+\frac{1}{\sqrt{T}}\int^T_0\abs{v}_{Y_{\eta,\mu,T}}^2dt\bigg]+C\epsilon^{-1}r^{-\frac{9}{2}}e^{2 \psi(T)}T\\
&\quad+C \epsilon r^{-\frac92} e^{\psi(T)}+  C\epsilon^{-1}r^{-\frac{9}{2}}e^{2 \psi(T)} T\Big( 1+\sup_{ t\leq T}\abs{v}_{X_{\eta,\mu,T}}^2\Big)\sup_{ t\leq T}\abs{v}_{X_{\eta,\mu,T}}^2\\
  &\quad+ C\epsilon^{-1}r^{-\frac{9}{2}} e^{2\psi(T)} T\Big(1 +  \sup_{ t\leq T} \abs{v}_{X_{\eta,\mu,T}}^2 \Big) \int_0^T \bigg(   \abs{\nabla_{\rm h}v}_{X_{\eta,\mu,T}}^2+ \frac{ \abs{v}_{Y_{\eta,\mu,T}}^2}{\sqrt{T}}     \bigg)dt.
    \end{aligned}
\end{equation*}
The assertion of Lemma \ref{cor-tildej} now follows by choosing $\epsilon=\tilde\epsilon r^{\frac92}e^{-\psi(T)}\sqrt{T}$ with  $\tilde\epsilon$  being an arbitrarily small positive number.
 \end{proof}

\appendix
\section{Proof of some  inequalities}\label{sec:appendix}
In this part we verify several  inequalities used in the text; the arguments are  straightforward.
\begin{proof}[Proof of inequalities \eqref{ineq1} and \eqref{ineq2}]
Recall $L_{\rho,m}$ is defined by \eqref{def:L}. For $0\leq k\leq [\frac{m}{2}]$, we have $\frac{m}{2}\leq m-k\leq m$ and thus a direct computation gives that
   \begin{align*}
       &\binom{m}{k}\frac{L_{\rho,m}}{L_{\rho,k}L_{\rho,m-k}}\\
       &=\frac{m!}{k!(m-k)!}\frac{\rho^{m+1}(m+1)^{6+2\sigma}}{(m!)^\sigma}\frac{(k!)^\sigma}{\rho^{k+1}(k+1)^{6+2\sigma}}\frac{[(m-k)!]^\sigma}{\rho^{m-k+1}(m-k+1)^{6+2\sigma}}\\
       &\leq \frac{C}{(k+1)^{6+2\sigma}}\left[\frac{k!(m-k)!}{m!}\right]^{\sigma-1}\leq \frac{C}{(k+1)^{6+2\sigma}},
   \end{align*}
   the last line using $\sigma\ge 1$ and $\frac{\rho_0}{2}\leq\rho\leq \rho_0$. This completes the proof of \eqref{ineq1}.

For $[\frac{m}{2}]+1\leq k\leq m$, a direct computation gives
   \begin{align*}
       &\binom{m}{k}\frac{L_{\rho,m}}{L_{\rho,k}L_{\rho,m-k}}\\
       &=\frac{m!}{k!(m-k)!}\frac{\rho^{m+1}(m+1)^{6+2\sigma}}{(m!)^\sigma}\frac{(k!)^\sigma}{\rho^{k+1}(k+1)^{6+2\sigma}}\frac{[(m-k)!]^\sigma}{\rho^{m-k+1}(m-k+1)^{6+2\sigma}}\\
       &\leq \frac{C}{(m-k+1)^{6+2\sigma}}\left[\frac{k!(m-k)!}{m!}\right]^{\sigma-1}\leq \frac{C}{(m-k+1)^{6+2\sigma}},
   \end{align*}
 which completes the proof of \eqref{ineq2}.
\end{proof}

\begin{proof}[Proof of inequalities \eqref{ineq3} and \eqref{ineq4}]
For $1\leq k\leq [\frac{m}{2}]$, we use $\sigma\ge 1$ compute
    \begin{align*}
    &\binom{m}{k}\frac{L_{\rho,m}}{L_{\rho,k}L_{\rho,m-k+1}(m-k+2)^\frac{1}{2}(m+1)^\frac{1}{2}}\\
    &=\frac{m!}{k!(m-k)!}\frac{\rho^{m+1}(m+1)^{6+2\sigma}}{(m!)^\sigma}\frac{(k!)^\sigma}{\rho^{k+1}(k+1)^{6+2\sigma}}\\
    &\qquad\times\frac{[(m-k+1)!]^\sigma}{\rho^{m-k+2}(m-k+2)^{6+2\sigma}}\frac{1}{(m-k+2)^\frac{1}{2}(m+1)^\frac{1}{2}}\\
    &\leq C\frac{m!}{k(k-1)!(m-k+1)!}\frac{[(k-1)!]^\sigma[(m-k+1)!]^\sigma}{(m!)^\sigma}\frac{1}{(k+1)^{6+\sigma}}\\
    &\leq \frac{C}{(k+1)^{6+\sigma}}\left[\frac{(k-1)!(m-k+1)!}{m!}\right]^{\sigma-1}\leq \frac{C}{(k+1)^{6+\sigma}}.
    \end{align*}
This gives \eqref{ineq3}.

For $[\frac{m}{2}]+1\leq k\leq m$, a direct computation gives
\begin{align*}
&\binom{m}{k}\frac{L_{\rho,m}}{L_{\rho,k}L_{\rho,m-k+1}}\\
&=\frac{m!}{k!(m-k)!}\frac{\rho^{m+1}(m+1)^{6+2\sigma}}{(m!)^\sigma}\frac{(k!)^\sigma}{\rho^{k+1}(k+1)^{6+2\sigma}}\frac{[(m-k+1)!]^\sigma}{\rho^{m-k+2}(m-k+2)^{6+2\sigma}}\\
&\leq \frac{C}{(m-k+1)^{6+\sigma}}\left[\frac{k!(m-k)!}{m!}\right]^{\sigma-1}\leq \frac{C}{(m-k+1)^{6+\sigma}}.
\end{align*}
The proof of \eqref{ineq4} is thus completed.
\end{proof}

\begin{proof}[Proof of inequalities \eqref{ineq5} and \eqref{ineq6}]
Recall $M_{\rho,\beta}$ is given in \eqref{mrho}. For $0\leq \abs{\gamma}\leq [\frac{\abs{\beta}}{2}]$, we have $\frac{\abs{\beta}}{2}\leq \abs{\beta}-\abs{\gamma}\leq \abs{\beta}$. Then a direct computation with the fact
\begin{align*}
    \binom{\beta}{\gamma}\leq \binom{\abs{\beta}}{\abs{\gamma}}
\end{align*}
gives
   \begin{align*}
       &\binom{\beta }{\gamma }\frac{M_{\rho,\beta}}{M_{\rho,\gamma}M_{\rho,\beta -\gamma}}\leq \binom{\abs{\beta}}{\abs{\gamma}}\frac{M_{\rho,\beta}}{M_{\rho,\gamma}M_{\rho,\beta -\gamma}}\\
       &\leq\frac{\abs{\beta}!}{\abs{\gamma}!(\abs{\beta}-\abs{\gamma})!}\frac{ \rho^{\abs\beta+1}(\abs\beta+1)^{6+2\sigma}} {(\beta_2!)^{\sigma-{\delta} } (\abs\beta !) ^{{\delta} }} \frac{(\gamma_2!)^{\sigma-{\delta} } (\abs\gamma !) ^{{\delta} }}{ \rho^{\abs\gamma+1}(\abs\gamma+1)^{6+2\sigma}}\\
       &\qquad\times\frac{[(\beta_2-\gamma_2)!]^{\sigma-\gamma} [(\abs{\beta}-\abs{\gamma})!]^\gamma }{\rho^{\abs\beta-\abs\gamma+1}(\abs \beta-\abs\gamma+1)^{6+2\sigma } }\\
       &\leq \frac{C}{(\abs{\gamma}+1)^{6+2\sigma}}\left[\frac{\gamma_2!(\beta_2-\gamma_2)!}{\beta_2!}\right]^{\sigma-{\delta}}\left[\frac{\abs{\gamma}!(\abs{\beta}-\abs{\gamma})!}{\abs{\beta}!}\right]^{{\delta}-1}\leq \frac{C}{(\abs{\gamma}+1)^{6+2\sigma}},
   \end{align*}
   the last line using $1\leq {\delta}\leq\sigma$. This completes the proof of \eqref{ineq5}.

For $[\frac{\abs{\beta}}{2}]+1\leq \abs{\gamma}\leq \abs{\beta}$, a direct computation with $1\leq\delta\leq\sigma$ gives
   \begin{align*}
       \binom{\beta }{\gamma }\frac{M_{\rho,\beta}}{M_{\rho,\gamma}M_{\rho,\beta -\gamma}}
       &\leq \frac{C}{(\abs{\beta}-\abs{\gamma}+1)^{6+2\sigma}}\left[\frac{\gamma_2!(\beta_2-\gamma_2)!}{\beta_2!}\right]^{\sigma-{\delta}}\left[\frac{\abs{\gamma}!(\abs{\beta}-\abs{\gamma})!}{\abs{\beta}!}\right]^{{\delta}-1}\\
       &\leq \frac{C}{(\abs{\beta}-\abs{\gamma}+1)^{6+2\sigma}}.
   \end{align*}
This completes the proof of  \eqref{ineq6}.
\end{proof}

\begin{proof}[Proof of inequality \eqref{ineq9}]
Recall $M_{\rho,\beta}$ is given in \eqref{mrho}. For $[\frac{\abs{\beta}}{2}]+1\leq \abs{\gamma}\leq \abs{\beta}$, we compute
   \begin{align*}
       &\binom{\beta }{\gamma }\frac{M_{\rho,\beta}}{M_{\rho,\gamma}M_{\rho,\beta -\gamma +(0,1)}}\\
&\leq\frac{\abs{\beta}!}{\abs{\gamma}!(\abs{\beta}-\abs{\gamma})!}\frac{ \rho^{\abs\beta+1}(\abs\beta+1)^{6+2\sigma}} {(\beta_2!)^{\sigma-{\delta} } (\abs\beta !) ^{{\delta} }} \frac{(\gamma_2!)^{\sigma-{\delta} } (\abs\gamma !) ^{{\delta} }}{ \rho^{\abs\gamma+1}(\abs\gamma+1)^{6+2\sigma}}\\
&\qquad\times\frac{[(\beta_2-\gamma_2+1)!]^{\sigma-\gamma} [(\abs{\beta}-\abs{\gamma}+1)!]^\gamma }{\rho^{\abs\beta-\abs\gamma+2}(\abs \beta-\abs\gamma+2)^{6+2\sigma } }\\
 &\leq \frac{C(\beta_2-\gamma_2+1)^{\sigma-\delta}}{(\abs{\beta}-\abs{\gamma}+1)^{6+2\sigma-\delta}}\left[\frac{\gamma_2!(\beta_2-\gamma_2)!}{\beta_2!}\right]^{\sigma-{\delta}}\left[\frac{\abs{\gamma}!(\abs{\beta}-\abs{\gamma})!}{\abs{\beta}!}\right]^{{\delta}-1}\\
      &\leq \frac{C}{(\abs{\beta}-\abs{\gamma}+1)^{6+\sigma}},
   \end{align*}
   the last inequality using $1\leq {\delta}\leq\sigma$. This completes the proof of  \eqref{ineq9}.
\end{proof}

\section{The algebra property of the space $\mathcal{H}^3$}\label{sec:algebra}

In this part we will show the algebra property of the space $\mathcal{H}^3$ defined in \eqref{def:mathcalH3}-namely, proving  assertion \eqref{prpalge} in Remark \ref{algebra}.
 Recalling the norm $\norm{\cdot}_{\mathcal{H}^3}$ is given in \eqref{def:mathcalH3}, we have for any $f,g\in\mathcal{H}^3$,
\begin{equation}\label{est:fgproduct}
    \begin{aligned}
\norm{fg}_{\mathcal{H}^3}&=\sum_{\stackrel{(j,   \alpha) \in\mathbb Z_+\times \mathbb{Z}_+^3}{j+\abs{ \alpha }\leq 3}}\   \norm{ \partial_t^j\partial^{ \alpha }(fg)}_{L^2}\\
&\leq \sum_{\stackrel{(j,   \alpha) \in\mathbb Z_+\times \mathbb{Z}_+^3}{j+\abs{ \alpha }\leq 3}}\sum_{k\leq j,\ \beta\leq\alpha}\  \binom{j}{k}\binom{\alpha}{\beta} \norm{ (\partial_t^k\partial^{ \beta }f)\partial_t^{j-k}\partial^{\alpha-\beta}g}_{L^2}\\
&\leq C\sum_{\stackrel{(j,   \alpha) \in\mathbb Z_+\times \mathbb{Z}_+^3}{j+\abs{ \alpha }\leq 3}}\sum_{k\leq j,\ \beta\leq\alpha}\  \norm{ (\partial_t^k\partial^{ \beta }f)\partial_t^{j-k}\partial^{\alpha-\beta}g}_{L^2}.
    \end{aligned}
\end{equation}
When $k+\abs{\beta}=0,$ we use the Sobolev inequality to conclude that for any $(j,   \alpha) \in\mathbb Z_+\times \mathbb{Z}_+^3$ with $j+\abs{ \alpha }\leq 3$,
\begin{multline*}
\norm{ (\partial_t^k\partial^{ \beta }f)\partial_t^{j-k}\partial^{\alpha-\beta}g}_{L^2}\leq \norm{f}_{L^\infty}\norm{\partial_t^j\partial^\alpha g}_{L^2}\\
\leq C\norm{f}_{H^3}\norm{\partial_t^j\partial^\alpha g}_{L^2}\leq C  \norm{f}_{\mathcal{H}^3}\norm{g}_{\mathcal{H}^3}.
\end{multline*}
When $k+\abs{\beta}=1,$ we have $j-k+\abs{\alpha-\beta}\leq 2$ and thus use the Sobolev inequality again to get that for any $(j,   \alpha) \in\mathbb Z_+\times \mathbb{Z}_+^3$ with $j+\abs{ \alpha }\leq 3$,
\begin{multline*}
\norm{ (\partial_t^k\partial^{ \beta }f)\partial_t^{j-k}\partial^{\alpha-\beta}g}_{L^2}\leq \norm{\partial_t^k\partial^{ \beta }f}_{L_{x_1}^2L_{x_2,x_3}^\infty}\norm{\partial_t^{j-k}\partial^{\alpha-\beta}g}_{L_{x_1}^\infty L_{x_2,x_3}^2}\\
\leq C\norm{\partial_t^k\partial^{ \beta }f}_{H^2}\norm{\partial_t^{j-k}\partial^{\alpha-\beta}g}_{H^1}\leq C  \norm{f}_{\mathcal{H}^3}\norm{g}_{\mathcal{H}^3}.
\end{multline*}
The case of $k+\abs{\beta}=2,3$ is similar. As a result, we have for any $(j,   \alpha) \in\mathbb Z_+\times \mathbb{Z}_+^3$ with $j+\abs{ \alpha }\leq 3$ and $k\leq j,\ \beta\leq\alpha$,
\begin{align*}
    \norm{ (\partial_t^k\partial^{ \beta }f)\partial_t^{j-k}\partial^{\alpha-\beta}g}_{L^2}\leq C  \norm{f}_{\mathcal{H}^3}\norm{g}_{\mathcal{H}^3}.
\end{align*}
Substituting the above estimate into \eqref{est:fgproduct} yields that for any $f,g\in\mathcal{H}^3$,
\begin{align*}
    \norm{fg}_{\mathcal{H}^3}\leq C  \norm{f}_{\mathcal{H}^3}\norm{g}_{\mathcal{H}^3}.
\end{align*}
This completes the proof of \eqref{prpalge}.

  \subsection*{\bf Acknowledgements}
The research of W.-X.Li was supported by Natural Science Foundation of China
(Nos. 12325108, 12131017, 12221001), and the Natural Science Foundation of Hubei
Province (No. 2019CFA007).   The research of P.Zhang was partially
supported by National Key R\&D Program of China under grant 2021YFA1000800
and by   National Natural Science Foundation of China (Nos. 12421001, 12494542, and 12288201).
The authors would like to thank the support from the Research Centre for
Nonlinear Analysis in The Hong Kong Polytechnic University.

% \bibliographystyle{abbrv}
%\bibliography{references}

\end{document}